\documentclass[12pt]{amsart}
\usepackage{ amsmath,amssymb,amsthm,color, euscript, enumerate,longtable}
\usepackage{dsfont}
\usepackage{graphicx}
\usepackage{longtable}
\usepackage{comment}
\newcommand{\1}{ \mathds{1}}

\newcommand{\qdim}{\mathop{\mathrm{qdim}}}
\newcommand{\Com}{\mathrm{Com}}

\newcommand{\maru}[1]{{\ooalign{\hfil#1\/\hfil\crcr
\raise.167ex\hbox{\mathhexbox20D}}}}

\newcommand{\ruby}[2]{%
 \leavevmode
 \setbox0=\hbox{#1}%
 \setbox1=\hbox{\tiny #2}%
 \ifdim\wd0>\wd1 \dimen0=\wd0 \end{lemma}se \dimen0=\wd1 \fi
 \hbox{%
   \kanjiskip=0pt plus 2fil
   \xkanjiskip=0pt plus 2fil
   \vbox{%
     \hbox to \dimen0{%
       \tiny \hfil#2\hfil}%
     \nointerlineskip
     \hbox to \dimen0{\mathstrut\hfil#1\hfil}}}}

\DeclareMathOperator*{\fusion}{\boxtimes}
\newcommand{\Z}{\mathbb{Z}}
\newcommand{\C}{\mathbb{C}}
\newcommand{\R}{\mathbb{R}}

\newcommand{\Q}{\mathbb{Q}}

\newcommand{\g}{\mathfrak{g}}

\newcommand{\Fg}{\mathfrak{g}}
\newcommand{\End}{\mathrm{End}}

\newcommand{\Aut}{\mathrm{Aut}\,}
\newcommand{\Inn}{\mathrm{Inn}\,}

\newcommand{\ch}{\mathrm{ch}\,}

\makeatletter \@addtoreset{equation}{section}

\theoremstyle{plain}
\newtheorem{maintheorem}{Main Theorem}
\newtheorem{theorem}{Theorem}[section]
\newtheorem{proposition}[theorem]{Proposition}
\newtheorem{lemma}[theorem]{Lemma}

\theoremstyle{definition}
\newtheorem{definition}[theorem]{Definition}

\theoremstyle{remark}
\newtheorem{remark}[theorem]{Remark}
\numberwithin{equation}{section}

\title[Holomorphic vertex operator algebras]{Inertia groups and uniqueness of holomorphic vertex operator algebras}
\author{Ching Hung Lam} %
  \address[C. H. Lam] {Institute of Mathematics, Academia Sinica, Taipei 10617, Taiwan and National Center for Theoretical Sciences of  Taiwan.}
  \email{chlam@math.sinica.edu.tw}
\author[H. Shimakura]{Hiroki Shimakura}%
\address[H. Shimakura]{Graduate School of Information Sciences,
Tohoku University,
Sendai 980-8579, Japan }%
\email {shimakura@tohoku.ac.jp}%
\date{}
\thanks{C.\,H. Lam was partially supported by MoST grant 104-2115-M-001-004-MY3 of Taiwan}
\thanks{H.\ Shimakura was partially supported by JSPS KAKENHI Grant Numbers 26800001 and 17K05154}
\thanks{C.\,H. Lam and H.\ Shimakura were partially supported by JSPS Program for Advancing Strategic International Networks to Accelerate the Circulation of Talented Researchers ``Development of Concentrated Mathematical Center Linking to Wisdom of the Next Generation".}

\newcommand{\sfr}[2]{\leavevmode\kern-.1em
  \raise.5ex\hbox{\the\scriptfont0 #1}\kern-.1em
  /\kern-.15em\lower.25ex\hbox{\the\scriptfont0 #2}}

\begin{document}

%\maketitle

%\tableofcontents

\begin{abstract}
We continue our program on classification of holomorphic vertex operator algebras of central charge $24$. 
In this article, we show that  there exists a unique strongly regular holomorphic VOA of central charge $24$, up to isomorphism, if its  weight one Lie algebra has the type $C_{4,10}$, $D_{7,3}A_{3,1}G_{2,1}$, $A_{5,6}C_{2,3}A_{1,2}$, $A_{3,1}C_{7,2}$, $D_{5,4}C_{3,2}A_{1,1}^2$, or $E_{6,4}C_{2,1}A_{2,1}$. 
As a consequence, we have verified that the isomorphism class of a strongly regular holomorphic vertex operator algebra of central charge $24$ is determined by its weight one Lie algebra structure if the weight one subspace is nonzero.
\end{abstract}
\maketitle

\tableofcontents

\section{Introduction}

The classification of (strongly regular) holomorphic vertex operator algebras (VOAs) of central charge $24$ is one of the important problems in VOA theory. In 1993, Schellekens \cite{Sc93} obtained a list of $71$ possible Lie algebra structures for the weight one subspace of a holomorphic VOA of central charge $24$. It is also believed that the isomorphism class of a holomorphic VOA of central charge $24$ is uniquely determined by its weight one Lie algebra structure. Recently, there has been much progress towards the classification. First, Schellekens' list was verified mathematically  in \cite{EMS}. In addition, the orbifold construction associated with an automorphism of arbitrary finite order has been established in \cite{EMS}. Using the orbifold construction, it has been verified that all the $71$ Lie algebras in Schellekens' list can be realized as weight one Lie algebras of some holomorphic VOAs of central charge $24$ \cite{DGM,EMS,FLM,Lam,LLin,LS,LS3,LS4,Mi3,SS}. There has been also progress towards the uniqueness part; in \cite{LS5}, we proposed a method, which we call ``Reverse orbifold construction",  for proving the uniqueness part. The main idea is to ``reverse" the original orbifold construction and to try to prove the uniqueness of certain VOAs by studying the conjugacy classes of some automorphisms. 
This technique turns out to be quite promising; indeed, many cases has been proved by this method \cite{EMS2,KLL,LLin,LS5,LS6}. 
%In particular, 
Up to now, the uniqueness for $64$ cases have been established: $24$ cases in \cite{DMb}, $2$ cases in \cite{LS3}, $3$ cases in \cite{LS5}, $3$ cases in \cite{LLin}, $13$ cases in \cite{KLL}, $14$ cases in \cite{EMS2} and $5$  cases in \cite{LS6}.  The weight one Lie algebras of the remaining $7$ cases have the type   $C_{4,10}$, $D_{7,3}A_{3,1}G_{2,1}$, $A_{5,6}C_{2,3}A_{1,2}$, $A_{3,1}C_{7,2}$, $D_{5,4}C_{3,2}A_{1,1}^2$, $E_{6,4}C_{2,1}A_{2,1}$, and $\{0\}$.

In this article, we establish the uniqueness for the remaining $6$ cases above except for the $\{0\}$ case. The main theorem is as follows:
\begin{maintheorem}[see Theorem \ref{T:Main}]
Up to isomorphism, there exists a unique strongly regular holomorphic VOA of central charge $24$ if its  weight one Lie algebra has the type $C_{4,10}$, $D_{7,3}A_{3,1}G_{2,1}$, $A_{5,6}C_{2,3}A_{1,2}$, $A_{3,1}C_{7,2}$, $D_{5,4}C_{3,2}A_{1,1}^2$, or $E_{6,4}C_{2,1}A_{2,1}$. 
\end{maintheorem}
As a consequence, we have established that 
the isomorphism class of a strongly regular holomorphic VOA of central charge $24$ is determined by its weight one Lie algebra if the weight one subspace is nonzero. Remark that the remaining case is a famous conjecture in VOA theory \cite{FLM}, namely, a strongly regular holomorphic VOA of central charge $24$ with trivial weight one subspace is isomorphic to Frenkel-Lepowsky-Meurman's Moonshine VOA. Unfortunately, the technique used in this article cannot be applied to this case  since the VOA has no inner automorphisms if the weight one subspace is zero.

Now let us explain the method proposed in \cite{LS5}. Let $\g$ be a Lie algebra and $\mathfrak{p}$ its subalgebra.
Let $n$ be a positive integer.
We consider the following hypotheses:
\begin{enumerate}[(i)]
\item For any holomorphic VOA $V$ of central charge $c$ with $V_1\cong\g$, there exists an order $n$ automorphism $\sigma$ of $V$ such that $\widetilde{V}_{\sigma}\cong W$ and $V_1^{\sigma}\cong\mathfrak{p}$, where $W$ is another holomorphic VOA of central charge $c$ whose structure is known.
\item Any order $n$ automorphism $\varphi$ of $W$ belongs to a unique conjugacy class if $(W^\varphi)_1\cong\mathfrak{p}$ and $(\widetilde{W}_\varphi)_1\cong\g$.  
\end{enumerate}
Under the above hypotheses, it was shown in \cite{LS5} that any holomorphic VOA $V$ of central charge $c$ is isomorphic to $\widetilde{W}_\varphi$ if $V_1\cong \g$.
%the holomorphic VOA obtained by applying the $\Z_p$-orbifold construction to $\tilde{V}$ and $\tau$, and 
In particular, the holomorphic VOA structure of central charge $c$ is unique if its weight one Lie algebra is isomorphic to $\g$.

Therefore, the main strategy is to find $\sigma$ and $W$ such that (i) and (ii) hold.  
For (i), it is actually quite easy to find an inner automorphism $\sigma\in \Aut(V)$ such that $V_1^{\sigma}\cong\mathfrak{p}$ and $\widetilde{V}_{\sigma}\cong W$ and there are usually several different choices for $\sigma$ and $W$.
The difficult part is to show (ii), i.e.,  any automorphism of $W$ satisfying the conditions $(W^\varphi)_1\cong\mathfrak{p}$ and $(\widetilde{W}_\varphi)_1\cong\g$ belongs to a unique conjugacy class. 
Note that in \cite{EMS2,KLL,LS5,LS6}, $W$ is chosen to be a lattice VOA, whose automorphism group is well-studied (\cite{DN}).
In this article, we will show (ii) for some non-lattice type VOAs $W$ and $n=2$. 

For this purpose, we first use the theory of Kac about the classification of finite automorphisms of (finite-dimensional) simple Lie algebras \cite[Chapter 8]{Kac} (cf.\ \cite{H}). Indeed, it is quite straightforward to show that $\varphi_1|_{W_1}$ is conjugate to $\varphi_2|_{W_1}$ by an inner automorphism in $\Aut (W_1)$ as Lie algebra automorphisms for our cases if $\varphi_1$ and $\varphi_2$ are involutions and $(W^{\varphi_1})_1\cong (W^{\varphi_2})_1$ (cf. Lemma \ref{L:conjinv} and \cite[Lemma 7.4]{EMS2}). It means that there is an inner automorphism $x$ of $W$ such that $\varphi_1 x \varphi_2^{-1} x^{-1}$ acts trivially on $W_1$. Therefore, the main step is to determine the  subgroup of $\Aut (W)$ which acts trivially on $W_1$. We call such a subgroup of $\Aut (W)$  \emph{the inertia group} of $W$ and denote it by $I(W)$. 
In Section \ref{S:inertia}, we determine the inertia groups for four holomorphic VOAs of central charge $24$ as follows:
% whose weight one Lie algebras have the type $E_{6,3}G_{2,1}^3$, $A_{4,5}^2$, $C_{5,3}G_{2,2}A_{1,1}$, and $A_{7,4}A_{1,1}^3$.  
\begin{maintheorem}[cf. Theorems \ref{IE63}, \ref{IA45}, \ref{IC5}, \ref{A74} ]
Let $W$ be a  holomorphic VOA of central charge $24$ such that the weight one Lie algebra $W_1$ has the type $E_{6,3}G_{2,1}^3$, $A_{4,5}^2$, $C_{5,3}G_{2,2}A_{1,1}$, or $A_{7,4}A_{1,1}^3$.
%with the weight one Lie algebra $\g$, where $\g= E_{6,3}G_{2,1}^3$, $A_{4,5}^2$, $C_{5,3}G_{2,2}A_{1,1}$, or $A_{7,4}A_{1,1}^3$. 
Then 
\[
I(W) = 
\begin{cases}
1 & \text{ if }  W_1= E_{6,3}G_{2,1}^3\text{ or } A_{4,5}^2, \\
%1 & \text{ if }  \g= A_{4,5}^2, \\
\langle \sigma_{(0,0,\Lambda_1)}\rangle \cong \Z_2 & \text{ if }  W_1= C_{5,3}G_{2,2}A_{1,1}, \\
\langle \sigma_{(0,\Lambda_1,0,0)}, \sigma_{(0,0,\Lambda_1,0)}, \sigma_{(0,0,0,\Lambda_1)}\rangle \cong \Z_2^3 & \text{ if }  W_1= A_{7,4}A_{1,1}^3,
\end{cases}
\]
where $\Lambda_1$ denotes the fundamental weight of a root system of type $A_1$. 
\end{maintheorem} 

Let $U$ be the subVOA generated by $W_1$. Then any element in $I(W)$ induces a $U$-module isomorphism on $W$ and thus it essentially acts as scalars on irreducible submodules for $U$. Therefore, the decomposition of $W$ as a $U$-module is useful for determining $I(W)$;  
indeed, we (partially) describe them by using the representation theory of the simple affine VOAs \cite{FZ} and the theory of mirror extensions \cite{DJX1,Lin,OS,Xu}.
In addition, we also use the representation theory of the simple current extension (\cite{La01,SY03,Ya04}), quantum dimensions (\cite{DJX}) and the fact that $W$ is a simple current extension of $W^g$ for any finite order automorphism $g$ of $W$ \cite{ADJR} to discuss the possible action of elements of $I(W)$.
For calculating the fusion rules and quantum dimensions for simple affine VOAs, we use the computer package ``Kac" \cite{kac}. 
It turns out that the group $I(W)$ is equal to the set of inner automorphisms $\sigma_\lambda(=\exp(-2\pi\sqrt{-1}\lambda_{(0)}))$, where $\lambda$ range over coweights of $W_1$, for the cases that we considered.
%$\{ \sigma_\lambda\mid \lambda \text{ is a coweight of } W_1\}$ 
Note that this is also true for holomorphic lattice VOAs of central charge $24$ (\cite{LS5}); however, we do not know if  this is true in general.

Thanks to Main Theorem 2, we show that the order $2$ automorphisms are uniquely determined up to conjugation by the fixed point Lie algebras of their weight one spaces for the cases that we considered, and the conformal weights and certain (Lie) weight structures of the corresponding twisted modules  (cf. Lemmas \ref{L:conjD73}, \ref{L:conjA45}, \ref{L:conjC53}, \ref{L:conjA74}, \ref{L:conjA74-2}, \ref{L:conjA74-3}).  Therefore, we can also establish (ii) for all our cases.
%}{blue}

The organization of this article is as follows: in Section \ref{S:2}, we review some basic facts about VOAs and affine Lie algebras. 
In Section 3, we study certain simple current extensions of some simple affine VOAs. The results will be used in Section 6 for computing the inertia groups.
In Section 4, we recall the orbifold construction established in \cite{EMS,MoPhD} and the method of ``reverse orbifold" from \cite{LS5}. In Section 5, we review the notion of quantum dimension \cite{DJX} and  the theory of mirror extension from \cite{DJX1,Lin, Xu}.  In Section 6, we determine explicitly the inertia groups of some holomorphic VOAs of central charge $24$. We also show that certain automorphisms of holomorphic VOAs are uniquely determined, up to conjugation, by the fixed point subalgebras of their weight one Lie algebras and some information of the corresponding twisted modules. In Section 7 and 8, we will prove our main theorem and show that there exists a unique strongly regular holomorphic VOA of central charge $24$, up to isomorphism,  if its  weight one Lie algebra has the type $C_{4,10}$, $D_{7,3}A_{3,1}G_{2,1}$, $A_{5,6}C_{2,3}A_{1,2}$, $A_{3,1}C_{7,2}$, $D_{5,4}C_{3,2}A_{1,1}^2$, or $E_{6,4}C_{2,1}A_{2,1}$.  
%%%%%%%
\begin{center}
{\bf Notations}
\begin{small}
\begin{longtable}{ll}\\
$(\cdot|\cdot)$& the normalized Killing form so that $(\alpha|\alpha)=2$ for any long root $\alpha$.\\
$\langle\cdot|\cdot\rangle$& the normalized symmetric invariant bilinear form on a VOA $V$\\
& so that $\langle \1|\1\rangle=-1$, equivalently, $\langle a|b\rangle\1=a_{(1)}b$ for $a,b\in V_1$.\\
$[i_1,\dots,i_\ell]$& the weight $\sum_{j=1}^\ell i_j\Lambda_j$ or the associated module for the simple affine VOA.\\
$\alpha_i$& a simple root of an irreducible root system;\\
& here we adopt the labeling of \cite[Section 11.4]{Hu}\\
$\Com_V(U)$& the commutant of $U$ in $V$.\\
$I(V)$& the inertia group of $V$, that is, $I(V)=\{g\in\Aut (V)\mid g=id\ \text{on}\ V_1\}$. \\
$\Lambda_i$& the fundamental weight with respect to the simple root $\alpha_i$\\
$L(\g)$& the subVOA generated by a Lie subalgebra $\g$ of the weight one space of a VOA.\\
$L(V_1)$& the subVOA generated by the weight one space $V_1$ of a VOA $V$.\\
$L_\mathfrak{g}(k,0)$& the simple affine VOA associated with a simple Lie algebra $\mathfrak{g}$ at level $k$.\\
$L_\Fg(k,\lambda)$& the irreducible $L_\Fg(k,0)$-module with the highest weight $\lambda$.\\
$\tilde{L}_{X_n}(k,0)$& the simple current extension of $L_{X_n}(k,0)$ described in Section 3.\\
$M^{(u)}$& the $\sigma_u$-twisted $V$-module constructed from a $V$-module $M$ by Li's $\Delta$-operator.\\
$\qdim_V$& the quantum dimension of a $V$-module (see Section \ref{S:QD}).\\
$\sigma_u$& the inner automorphism $\exp(-2\pi\sqrt{-1}u_{(0)})$ of a VOA $V$ associated with $u\in V_1$.\\
$U(1)$& a $1$-dimensional abelian Lie algebra.\\ 
$V^{\sigma}$& the set of fixed points of an automorphism $\sigma$ of a VOA $V$.\\
$V(\sigma)$ & the irreducible $\sigma$-twisted module for a holomorphic VOA $V$.\\
$\tilde{V}_\sigma$& the VOA obtained by the orbifold construction associated with $V$ and $\sigma$.\\
$X_{n,k}$ & (the type of) a simple Lie algebra whose type is $X_n$ and level is $k$.\\
\end{longtable}
\end{small}
\end{center}
%\color{black}
\section{Preliminary}\label{S:2}
In this section, we will review some fundamental results about VOAs.

%% Kac % https://www.nikhef.nl/~t58/Site/Kac.html

\subsection{Vertex operator algebras and weight one Lie algebras}
Throughout this article, all VOAs are defined over the field $\C$ of complex numbers. 

A \emph{vertex operator algebra} (VOA) $(V,Y,\1,\omega)$ is a $\Z$-graded vector space $V=\bigoplus_{m\in\Z}V_m$ over the complex field $\C$ equipped with a linear map
$$Y(a,z)=\sum_{i\in\Z}a_{(i)}z^{-i-1}\in ({\rm End}\ (V))[[z,z^{-1}]],\quad a\in V,$$
the \emph{vacuum vector} $\1\in V_0$ and the \emph{conformal vector} $\omega\in V_2$
satisfying certain axioms (\cite{Bo,FLM}). For $a\in V$ and $i\in\Z$, we call the operator $a_{(i)}$ the \emph{$i$-th mode} of $a$.
Note that the operators $L(m)=\omega_{(m+1)}$, $m\in \Z$, satisfy the Virasoro relation:
$$[L{(m)},L{(n)}]=(m-n)L{(m+n)}+\frac{1}{12}(m^3-m)\delta_{m+n,0}c\ {\rm id}_V,$$
where $c\in\C$ is called the \emph{central charge} of $V$, and $L(0)$ acts by the multiplication of a scalar $m$ on $V_m$.

A linear automorphism $g$ of a VOA $V$ is called a (VOA) \emph{automorphism} of $V$ if $$ g\omega=\omega\quad {\rm and}\quad gY(v,z)=Y(gv,z)g\quad \text{ for all } v\in V.$$
The group of all (VOA) automorphisms of $V$ will be denoted by $\Aut (V)$. 
A \emph{vertex operator subalgebra} (or a \emph{subVOA}) is a graded subspace of
$V$ which has a structure of a VOA such that the operations and its grading
agree with the restriction of those of $V$ and  they share the vacuum vector.
In addition, if they also share the conformal vector, the subVOA is said to be \emph{full}.
For an automorphism $g$ of a VOA $V$, let $V^g$ denote the fixed point set of $g$, which is a full subVOA of $V$.

For $g\in\Aut (V)$ of order $n$, a $g$-twisted $V$-module $(M,Y_M)$ is a $\C$-graded vector space $M=\bigoplus_{m\in\C} M_{m}$ equipped with a linear map
$$Y_M(a,z)=\sum_{i\in(1/n)\Z}a_{(i)}z^{-i-1}\in (\End\ M)[[z^{1/n},z^{-1/n}]],\quad a\in V$$
satisfying a number of conditions (\cite{FHL,DLM2}).
Here the $\C$-grading is given by the $L(0)$-action.
Note that an (untwisted) $V$-module is a $1$-twisted $V$-module
and that a $g$-twisted $V$-module is an (untwisted) $V^g$-module.
For $v\in M_k$, the \emph{conformal weight} of $v$ is $k$. 
If $M$ is irreducible, then there exists $w\in\C$ such that $M=\bigoplus_{m\in(1/n)\Z_{\geq 0}}M_{w+m}$ and $M_w\neq0$.
The number $w$ is called the \emph{conformal weight} of $M$.

A VOA is said to be  \emph{rational} if the admissible module category is semisimple.
A rational VOA is said to be \emph{holomorphic} if it itself is the only irreducible module up to isomorphism.
A VOA $V$ is \emph{of CFT-type} if $V_0=\C\1$ (note that $V_i=0$ for all $i<0$ if $V_0=\C\1$), and is \emph{$C_2$-cofinite} if the codimension in $V$ of the subspace spanned by the vectors of form $u_{(-2)}v$, $u,v\in V$, is finite.
A $V$-module is said to be \emph{self-contragredient} if it is isomorphic to its contragredient module.
A VOA is said to be \emph{strongly regular} if it is rational, $C_2$-cofinite, self-contragredient and of CFT-type.
Note that a strongly regular VOA is simple.

%Combining the known results, we obtain the following:

\begin{theorem}\label{T:HKL} Let $V^0$ be a strongly regular VOA of CFT-type such that the conformal weight of any irreducible $V^0$-module is positive except for $V^0$ itself.
Let $V$ be a simple VOA such that $V$ contains $V^0$ and they share the conformal vector.
Then $V$ is also strongly regular.
\end{theorem}
\begin{proof} By \cite[Proposition 5.2]{ABD}, $V$ is $C_2$-cofinite.
In addition, by \cite[Theorems 3.5 and 3.6]{HKL}, $V$ is rational and self-contragredient.
\end{proof}

For a subVOA $U$ of a VOA $V$, the \emph{commutant} $\Com_V(U)$ of $U$ in $V$ is a subalgebra of $V$ which commutes with $U$ (\cite{FZ}).

\begin{proposition}[{\cite[Lemma 2.1]{ACKL}}]\label{P:ComSimple}
Let $V$ be a simple VOA of CFT-type and $U$ a simple rational subVOA.
Then the commutant $\Com_V(U)$ is simple.
\end{proposition}

Let $V$ be a VOA and $M=(M,Y_M)$ a $V$-module.
For $g\in\Aut (V)$, the \emph{$g$-conjugation} $M\circ g$ of $M$ is defined by $(M,Y_{M\circ g})$, where $Y_{M\circ g}(v,z)=Y_M(gv,z)$.
This conjugation naturally induces an action of $\Aut (V)$ on the set of all isomorphism classes of irreducible $V$-modules.
Note that this action preserves the fusion rules and characters.

\begin{lemma}\label{L:Vg} Let $V^0$ be a full subVOA of $V$.
Let $g\in\Aut (V)$ such that $g(V^0)=V^0$; let $\bar{g}\in\Aut (V^0)$ be the restriction of $g$ to $V^0$.
Then, as $V^0$-modules, $V\circ \bar{g}$ is isomorphic to $V$.
In particular, for any irreducible $V^0$-submodule $M$ of $V$, $V$ contains an irreducible $V^0$-submodule isomorphic to $M\circ \bar{g}$.
\end{lemma}
\begin{proof} It follows from $g\in\Aut (V)$ that $Y(\bar{g}v,z)gu=Y(gv,z)gu=gY(v,z)u$ for $v\in V^0$ and $u\in V$.
Then $g$ is a $V^0$-module isomorphism from $V$ to $V\circ \bar{g}$.
\end{proof}

We now assume that $V$ is of CFT-type.
Then, the weight one subspace $V_1$ has a Lie algebra structure via the $0$-th mode. Moreover, the $n$-th modes
$v_{(n)}$, $v\in V_1$, $n\in\Z$, define  an affine representation of the Lie algebra $V_1$ on $V$.
For a simple Lie subalgebra $\mathfrak{s}$ of $V_1$, the {\it level} of $\mathfrak{s}$ is defined to be the scalar by which the canonical central element acts on $V$ as the affine representation.
When the type of the root system of $\mathfrak{s}$ is $X_n$ and the level of $\mathfrak{s}$ is $k$, we denote by $X_{n,k}$ the type of $\mathfrak{s}$.
For an element $a\in V_1$, $\exp(a_{(0)})$ is an automorphism of $V$, called an {\it inner automorphism} of $V$.

In addition, we assume that $V$ is self-contragredient.
Then there exists a symmetric invariant bilinear form $\langle\cdot|\cdot\rangle$ on $V$, which is unique up to scalar (\cite{Li3}).
We normalize it so that  $\langle\1|\1\rangle=-1$.
Then for $a,b\in V_1$, we have $\langle a|b\rangle\1=a_{(1)}b$.

Assume also that $V_1$ is semisimple.
Let $\mathfrak{H}$ be a (fixed) Cartan subalgebra of $V_1$.
Let $(\cdot|\cdot)$ be the Killing form on $V_1$.
We identify the dual $\mathfrak{H}^*$ with $\mathfrak{H}$ via $(\cdot|\cdot)$ and normalize $(\cdot|\cdot)$ so that $(\alpha|\alpha)=2$ for any long root $\alpha\in\mathfrak{H}$.
%In this article, 
{\it Weights} are defined via $(\cdot|\cdot)$, that is, a vector $v\in V$ has weight $\lambda(\in\mathfrak{H})$ if $x_{(0)}v=(x|\lambda)v$ for all $x\in\mathfrak{H}$.
Remark that for $u\in\mathfrak{H}$, $\sigma_u$ acts on a vector with the weight $\lambda$ as the scalar multiple by $\exp(-2\pi\sqrt{-1}(u|\lambda))$.
The following lemma is immediate from the commutator relations of $n$-th modes (cf.\ {\cite[(3.2)]{DM06}}).

\begin{proposition}\label{Prop:level} 
Let $\mathfrak{s}$ be a simple Lie subalgebra of $V_1$ with level $k$.
Then $\langle\cdot|\cdot\rangle=k(\cdot|\cdot)$ on $\mathfrak{s}$.
\end{proposition}

Next we recall some results related to weight one Lie algebras.

\begin{proposition}[{\cite[Theorem 1.1, Corollary 4.3]{DM06}}]\label{Prop:posl} Let $V$ be a strongly regular VOA.
Then $V_1$ is reductive.
Let $\mathfrak{s}$ be a simple Lie subalgebra of $V_1$.
Then $V$ is an integrable module for the affine representation of $\mathfrak{s}$ on $V$, and the subVOA generated by $\mathfrak{s}$ is isomorphic to the simple affine VOA associated with $\mathfrak{s}$ at positive integral level.
\end{proposition}

\begin{proposition} [{\cite{DMb,EMS,Sc93}}]\label{Prop:V1} Let $V$ be a strongly regular holomorphic VOA of central charge $24$.
Assume that the Lie algebra $V_1$ is neither $\{0\}$ nor abelian.
Then $V_1$ is isomorphic to one of $69$ semisimple Lie algebras in \cite{Sc93}, and the subVOA generated by $V_1$ is full.
In addition, for any simple ideal of $V_1$ at level $k$, the identity$$\frac{h^\vee}{k}=\frac{\dim V_1-24}{24}$$
holds, where $h^\vee$ is the dual Coxeter number of the simple ideal.
\end{proposition}

\subsection{$\Delta$-operator, simple affine VOAs and twisted modules}

In this subsection, we recall the construction of certain twisted modules given by Li \cite{Li}.
We also recall the conformal weight of such a twisted module over a simple affine VOA from \cite{LS3}.

Let $V$ be a VOA of CFT-type.
Let $u\in V_1$ such that $u_{(0)}$ acts semisimply on $V$. Suppose that
there exists a positive integer $n\in\Z_{>0}$ such that the spectra of $u_{(0)}$ on $V$ belong to $({1}/{n})\Z$.
Then $\sigma_u=\exp(-2\pi\sqrt{-1}u_{(0)})$ is an automorphism of $V$ with $\sigma_u^n=1$.

Let $\Delta(u,z)$ be the $\Delta$-operator defined in \cite{Li}, i.e.,
\[
\Delta(u, z) = z^{u_{(0)}} \exp\left( \sum_{i=1}^\infty \frac{u_{(i)}}{-i} (-z)^{-i}\right).
\]

\begin{proposition}[{\cite[Proposition 5.4]{Li}}]\label{Prop:twist}
Let $\sigma$ be an automorphism of $V$ of finite order and
let $u\in V_1$ be as above such that $\sigma(u) = u$.
Let $(M, Y_M)$ be a $\sigma$-twisted $V$-module and
define $(M^{(u)}, Y_{M^{(u)}}(\cdot, z)) $ as follows:
\[
\begin{split}
& M^{(u)} =M \quad \text{ as a vector space;}\\
& Y_{M^{(u)}} (a, z) = Y_M(\Delta(u, z)a, z)\quad \text{ for any } a\in V.
\end{split}
\]
Then $(M^{(u)}, Y_{M^{(u)}}(\cdot, z))$ is a
$\sigma_u\sigma$-twisted $V$-module.
Furthermore, if $M$ is irreducible, then so is $M^{(u)}$.
\end{proposition}

For a $\sigma$-twisted $V$-module $M$ and $a\in V$, we denote by $a_{(i)}^{(u)}$ the operator which corresponds to the coefficient of $z^{-i-1}$ in $Y_{M^{(u)}}(a,z)$, i.e.,
$$Y_{M^{(u)}}(a,z)=\sum_{i\in\Q}a_{(i)}^{(u)}z^{-i-1}\quad \text{for}\quad a\in V.$$

For an element $x\in V_1$, the $0$-th mode of $x$ on $M^{(u)}$ is given by
\begin{equation}
x^{(u)}_{(0)}=x_{(0)}+\langle u|x\rangle {\rm id},\label{Eq:V1h}
\end{equation}
%Set $L^{(u)}(i)=\omega^{(u)}_{(i+1)}$.
and the (conformal) weight operator, i.e., the $1$-st mode of $\omega$, on $M^{(u)}$ is given by 
\begin{equation}
L(0)+u_{(0)}+\frac{\langle u|u\rangle}{2}{\rm id}.\label{Eq:Lh}
\end{equation}

For a simple Lie algebra $\mathfrak{g}$ and a positive integer $k$, let $L_{\mathfrak{g}}(k,0)$ be the simple affine VOA associated with $\mathfrak{g}$ at level $k$ (\cite{FZ}).
When the type of $\mathfrak{g}$ is $X_n$, it is also denoted by $L_{X_n}(k,0)$.
Fix a Cartan subalgebra of $\Fg$ and its simple roots.
A dominant integral weight $\lambda$ of $\g$ has level $k$ if $(\lambda|\theta)\le k$ for the highest weight $\theta$ of $\g$.
Then the set of all inequivalent irreducible $L_{\mathfrak{g}}(k,0)$-module is given by $L_{\mathfrak{g}}(k,\lambda)$, where $\lambda$ ranges over dominant integral weights of level $k$ (\cite{FZ}).
The following lemma is immediate from Proposition \ref{Prop:level} and \eqref{Eq:V1h}.

\begin{lemma}\label{Lem:wtTw}
Let $\g$ be a simple Lie algebra and $k\in\Z_{>0}$.
Let $v$ be a weight vector in $L_\Fg(k,\lambda)$ of weight $\mu$.
Then $v$ is also a weight vector in $L_\Fg(k,\lambda)^{(u)}$
and its weight is $\mu+ku$.
\end{lemma}

The conformal weight of the $\sigma_u$-twisted module $L_\Fg(k,\lambda)^{(u)}$ is discussed in \cite[Lemma 3.6]{LS3} for a simple Lie algebra $\g$.
One can easily generalize it for a semisimple Lie algebra.

\begin{lemma}\label{Lem:lowestwt}{\rm (cf.\ \cite[Lemma 3.6]{LS3})} Let $\mathfrak{g}$ be a semisimple Lie algebra and $\g=\bigoplus_{i=1}^t\mathfrak{g}_i$ be the sum of simple ideals.
Let $k_i$ be a positive integer and let $\lambda_i$ be a dominant integral weight of $\mathfrak{g}_i$ of level $k$.
Let $u$ be an element in the Cartan subalgebra of $\g$ such that $$(u|\alpha)\ge-1$$ for any root $\alpha$ of $\mathfrak{g}$ and let $u=\sum_{i=1}^tu_i$, where $u_i\in\g_i$.
Then the conformal weight of $\left(\bigotimes_{i=1}^t L_{\Fg_i}(k_i,\lambda_i)\right)^{(u)}$ is equal to \begin{equation}
\ell+\sum_{i=1}^t\min\{(u_i|\mu)\mid \mu\in\Pi(\lambda_i)\}+\frac{\langle u|u\rangle}{2},\label{Eq:twisttop}
\end{equation}
where $\ell$ is the conformal weight of $\bigotimes_{i=1}^tL_{\mathfrak{g}_i}(k_i,\lambda_i)$ and $\Pi(\lambda_i)$ is the set of all weights of the irreducible $\mathfrak{g}_i$-module with the highest weight $\lambda_i$.
\end{lemma}

The following lemma is a consequence of the results in \cite[Section 8]{Kac}.
We give a proof for completeness.

\begin{lemma}\label{L:conjinv} Let $\g$ be a (finite-dimensional) simple Lie algebra whose type is not $D_{2n}$, $(n\ge2)$, and let $\Inn(\g)$ be the inner automorphism group of $\g$.
Then the conjugacy class of an involution in $\Inn(\g)$ is uniquely determined by the isomorphism class of the fixed point set.
\end{lemma}
\begin{proof} The conjugacy classes of finite order automorphisms of $\g$ in $\Aut (\g)$ are completely described in \cite[Section 8]{Kac}.
In particular, the assertion holds if we replace $\Inn(\g)$ by $\Aut(\g)$ (cf.\ \cite[Exercise 8.10]{Kac}).
Hence if $\Inn(\g)=\Aut(\g)$, namely, the type of $\g$ is neither $A_{n}$ $(n\ge2)$, $D_{n}$ $(n\ge4)$, nor $E_6$, then there is nothing to prove.

Assume that the type of $\g$ is $A_n$ $(n\ge2)$, $D_{2n+1}$ $(n\ge2)$ or $E_6$.
Let $\tau$ be the order $2$ standard diagram automorphism of $\g$.
Let $\bar{\tau}$ be the restriction to the (fixed) Cartan subalgebra.
Note that $\Aut(\g)=\Inn(\g):\langle\tau\rangle$.

Let $g\in\Inn(\g)$ of order $2$.
It is enough to prove that the conjugacy class of $g$ in $\Inn(\g)$ is equal to that of $g$ in $\Aut(\g)$.
Let $v\in P/2$ such that $g$ is conjugate to $\sigma_v$ in $\Aut(\g)$, where $P$ is the weight lattice of $\g$.
It follows from $|\Aut(\g):\Inn(\g)|=2$ that $g$ is conjugate to $\sigma_v$ or $\sigma_{\bar{\tau}(v)}$ in $\Inn(\g)$.
By the assumption on the type of $\g$, the Weyl group contains $-\bar{\tau}$.
Hence $\sigma_{\bar{\tau}(v)}$ is conjugate to $\sigma_{-v}$ by an element of $\Inn(\g)$.
It follows from $\sigma_{2v}=id$ on $\g$ that $\sigma_v=\sigma_{-v}$ on $\g$.
Thus, $g$ is conjugate to $\sigma_v$ by an element of $\Inn(\g)$, and we obtain this lemma.
\end{proof}

%\color{black}

\section{Representation theory of simple current extensions}

In this section, we recall the notion of simple current modules and simple current extensions.
We also summarize the list of irreducible modules for some simple current extensions of simple affine VOAs based on the results in \cite{La01,SY03,Ya04}.

Throughout this article, we adopt the labeling of simple roots $\alpha_i$ of an irreducible root system as in \cite[Section 11.4]{Hu} and use $\Lambda_i$ to denote the fundamental weight with respect to $\alpha_i$.

\subsection{Simple current extensions and induced modules}\label{S:SCmod}
First, we review a classification of irreducible modules of simple current extensions from \cite{La01,SY03,Ya04}.

Let $V^0$ be a strongly regular VOA.
An irreducible $V^0$-module $M$ is called a \emph{simple current module} if the fusion product $M\boxtimes_{V^0} N$
is an irreducible module for any irreducible $V^0$-module $N$.
Let $D$ be a finite abelian group and let $\{V^\alpha\mid\alpha\in D\}$ be a set of inequivalent irreducible $V^0$-modules such that $V^\alpha\boxtimes_{V^0}V^\beta\cong V^{\alpha+\beta}$.
%For an irreducible $V^0$-module $M$, let $D_M$ denote the subgroup $\{\alpha\in D\mid V^\alpha\boxtimes_{V^0} M\cong M\}$ of $D$.
A simple VOA $V=\bigoplus_{\alpha\in D}V^\alpha$ is called a $D$-graded extension of $V^0$ if $V^0$ is a full subVOA and $V$ carries the natural $D$-grading, i.e., ${V^\alpha}_{(n)} V^\beta\subset V^{\alpha+\beta}$ for all $n\in\Z$. 
In addition, if all the $V^\alpha$ are simple current modules, then $V$ is called a $D$-graded \emph{simple current extension} of $V^0$.
Note that the VOA structure of $V$ is unique up to isomorphism (\cite[Proposition 5.3]{DM}) and that if the conformal weight of $V^\alpha$ is positive for any $\alpha\in D\setminus\{0\}$, then a simple current extension $\bigoplus_{\alpha\in D}V^\alpha$ of $V^0$ is also strongly regular (cf.\ \cite{LY}).

Let $V$ be a $D$-graded simple current extension of $V^0$.
Let $M$ be an irreducible $V$-module and $W$ an irreducible $V^0$-submodule of $M$.
By the definition of the fusion product, the irreducibility of $M$ and the rationality of $V$, $M$ is isomorphic to a $V$-submodule of the induced $V$-module $V\boxtimes_{V^0}W$.
Hence, all irreducible $V$-modules appear in some induced $V$-modules.
%\mycolor{
Notice that in general, the induced $V$-module $V\boxtimes_{V^0}W$ is a $g$-twisted $V$-module for some $g\in D^*$ (\cite[Theorem 3.3]{Ya04}) and it is not necessarily irreducible.

Let us explain the structures of the induced $V$-modules.
Let $W$ be an irreducible $V^0$-submodule of a $V$-module.
Set $D_W=\{\alpha\in D\mid V^\alpha\boxtimes_{V^0} W\cong W\}$.
Then $D_W$ is a subgroup of $D$.
Set $V_{D_W}=\bigoplus_{\alpha\in D_W}V^\alpha$.
Then $V_{D_W}$ is a subVOA of $V$ and a $D_W$-graded simple current extension of $V^0$.
In addition, $V$ is a $D/D_W$-graded simple current extension of $V_{D_W}$.
Let $X$ be an irreducible $V_{D_W}$-submodule of the induced $V_{D_W}$-module $V_{D_W}\boxtimes_{V^0}W$.
By the definition of $D_W$, all irreducible $V^0$-submodules of $X$ are isomorphic to $W$.
By \cite[Theorem 3.1]{La01}, there exists a projective representation $Q$ of $D_W$ such that $X\cong Q\otimes_\C W$.
%Let $N$ be an irreducible $V$-submodule of the induced $V$-module $V\boxtimes_{V_{D_W}}U$.
Set $V_{\alpha+D_W}=\bigoplus_{\beta\in D_W}V^{\alpha+\beta}$.
Then $V=\bigoplus_{\alpha+D_W\in D/D_W}V^{\alpha+D_W}$.
By the definition of $D_W$, $V_{\alpha+D_W}\boxtimes_{V_{D_W}}X\cong X$ if and only if $\alpha+D_W=D_W$.
Hence by \cite[Proposition 3.8]{SY03}, $V\boxtimes_{V_{D_W}}X$ is an irreducible $V$-module.

Now, we assume that $D$ is cyclic.
Clearly, $D_W$ is also cyclic.
Then the projective representation $Q$ is $1$-dimensional, and $X\cong W$ as $V^0$-modules.
Note that the multiplicity of $W$ in the induced module $V_{D_W}\boxtimes_{V^0}W$ is equal to $|D_W|$ and that all irreducible $V_{D_W}$-submodules of $V_{D_W}\boxtimes_{V^0}W$ are parametrized by linear characters of $D_W$ (\cite[Theorem 3.1]{La01}).
Let $\chi\in (D_W)^*$.
Then $\chi$ lifts to an automorphism $\tilde{\chi}$ of $V_{D_W}$ as $\tilde{\chi}(v)=\chi(\alpha)v$ for $v\in V^\alpha$.
Hence we may view $(D_W)^*$ as a subgroup of $\Aut (V_{D_W})$.
The following lemma is immediate from \cite[Theorem 3.7]{La01}.
\begin{lemma}\label{L:DW} The subgroup $(D_W)^*$ acts on the set of isomorphism classes of irreducible $V_{D_W}$-submodules of $V_{D_W}\boxtimes_{V^0}W$ via the conjugation, and this action is regular.
\end{lemma}

Let $N$ be an irreducible $V^0$-module.
Then $V\boxtimes_{V^0} N$ is a $g$-twisted $V$-module for some $g\in D^*$.
By the argument as in Section 3 of \cite{Ya04}, $g=1$ if and only if the differences of conformal weights of irreducible $V^0$-submodules of $V\boxtimes_{V^0} N$ are in $\Z$.
Therefore, one can find all irreducible $V$-modules by using the fusion products, or the fusion rules, of $V^0$.
In the following subsections, we summarize the list of irreducible modules of simple current extensions of (the tensor product of) some simple affine VOAs graded by a cyclic group; these extension VOAs will be realized as some subVOAs of certain holomorphic VOAs of central charge $24$ in the later sections. We also note that the fusion rules of simple affine VOAs and the irreducible modules for their simple current extensions can be computed by the computer package ``Kac" \cite{kac}.

\subsection{Extension of $L_{E_6}(3,0)$}\label{S:EE63}

Let $\tilde{ L}_{E_6}(3,0)=L_{E_6}(3,0)\oplus L_{E_6}(3,3\Lambda_1)\oplus L_{E_6}(3,3\Lambda_6)$ be a $\Z_3$-graded simple current extension of $L_{E_6}(3,0)$ \cite[Proposition 3.8 and Corollary 3.21]{Li01} (cf.\ \cite{DLM96}).
By the argument in Section \ref{S:SCmod}, we obtain the following lemma:
\begin{lemma}
The VOA $\tilde{ L}_{E_6}(3,0)$ has exactly $8$ inequivalent irreducible modules and they are given as in Table \ref{tableE6}.
\end{lemma}

\begin{table}[bht] 
\caption{Irreducible modules for $\tilde{ L}_{E_6}(3,0)$
}\label{tableE6}
\tiny
\begin{tabular}{|c|cc|c|}
\hline
Module & As an $L_{E_6}(3,0)$-module  & & Conformal weight\\ \hline \hline 
$\tilde{ L}_{E_6}(3,0)$ & $L_{E_6}(3,0)\oplus L_{E_6}(3,3\Lambda_1)\oplus L_{E_6}(3,3\Lambda_6)$ & & $0$\\ \hline 
$\tilde{ L}_{E_6}(3,\Lambda_2)$ &$L_{E_6}(3,\Lambda_2)\oplus L_{E_6}(3,\Lambda_1+\Lambda_3) \oplus L_{E_6}(3,\Lambda_5+\Lambda_6)$  & & $4/5$\\ \hline 
$\tilde{ L}_{E_6}(3,\Lambda_1+\Lambda_6)^0$ &$L_{E_6}(3,\Lambda_1+\Lambda_6)$ & & $6/5$\\ \hline 
$\tilde{ L}_{E_6}(3,\Lambda_1+\Lambda_6)^1$ &$L_{E_6}(3,\Lambda_1+\Lambda_6)$ & & $6/5$\\ \hline 
$\tilde{ L}_{E_6}(3,\Lambda_1+\Lambda_6)^2$ &$L_{E_6}(3,\Lambda_1+\Lambda_6)$ & & $6/5$\\ \hline 
$\tilde{ L}_{E_6}(3,\Lambda_4)^0$ &$L_{E_6}(3,\Lambda_4)$ & & $8/5$\\ \hline 
$\tilde{ L}_{E_6}(3,\Lambda_4)^1$ & $L_{E_6}(3,\Lambda_4)$ & & $8/5$\\ \hline 
$\tilde{ L}_{E_6}(3,\Lambda_4)^2$ &$L_{E_6}(3,\Lambda_4)$ & & $8/5$\\
\hline 
\end{tabular}
\end{table}
%\color{red}

\subsection{Extension of $L_{A_4}(5,0)$}\label{S:A45}

Let $\tilde{ L}_{A_4}(5,0)=L_{A_4}(5,0)\oplus\bigoplus_{i=1}^4L_{A_4}(5,5\Lambda_i)$ be a $\Z_5$-graded simple current extension of $L_{A_4}(5,0)$ \cite[Proposition 3.8 and Corollary 3.21]{Li01} (cf.\ \cite[Example 5.11]{DLM96}). 
By the argument in Section \ref{S:SCmod}, we obtain the following lemma:

\begin{lemma}\label{L:A4mod}
The VOA $\tilde{L}_{A_{4}}(5,0)$ has exactly $10$ inequivalent irreducible modules and they are given as in Table \ref{tableA4}, where $[i_1\ i_2\ i_3\ i_4]$ denotes the irreducible $L_{A_4}(5,0)$-module $L_{A_4}(5,\sum_{j=1}^4i_j\Lambda_j)$. In addition, all irreducible $\tilde{L}_{A_4}(5,0)$-modules are self-contragredient. 
\end{lemma}

%\begin{proof}
%That $\tilde{L}_{A_{4}(5,0)}$ has has exactly $10$ inequivalent irreducible modules can be verified easily using the argument in Section \ref{S:SCmod}.
%as in standard theory of simple current extension (cf. \cite{Ya04}). 
%\mycolor{\bf I have not checked the fusion rules of $\tilde{L}_{A_4}(5,0)$.}{red}
%\end{proof}

By using \cite{kac}, one can also determine the quantum dimensions (see Section \ref{S:QD} below) of irreducible $\tilde{L}_{A_4}(5,0)$-modules, which are also listed in Table \ref{tableA4}.  

\begin{table}[bht] 
\caption{Irreducible modules for $\tilde{L}_{A_4}(5,0)$}\label{tableA4}
\tiny
\begin{tabular}{|c|c|c|c|c|}
\hline
Module & As $L_{A_4}(5,0)$-module  & Conformal weight& $\qdim$&$\qdim^2$ \\ \hline \hline 
$\tilde{L}_{A_4}(5,0)$ & $[0000]+[5000]+[0500]+[0050]+[0005]$
%$L_{A_4}(5,0)\oplus L_{A_4}(5,5\Lambda_1)\oplus L_{A_4}(5,5\Lambda_2)$ 
&  $0$& $1$&$1$\\ 
%		& $L_{A_4}(5,5\Lambda_3)\oplus L_{A_4}(5,5\Lambda_4)$ & && \\ 
\hline 
$\tilde{L}_{A_4}(5, \Lambda_1+\Lambda_4) $ & $[1001]+[3100]+[1310]+[0131]+[0013]$
%$  	L_{A_4}(\Lambda_1+\Lambda_4)\oplus L_{A_4}(3\Lambda_1+\Lambda_2)\oplus L_{A_4}( \Lambda_1+3\Lambda_2+\Lambda_3)$  
  	 & $1/2$& $5+2\sqrt{5}$ &$45+20\sqrt5$\\ 	
%  	 & $\oplus  L_{A_4}(\Lambda_2+3\Lambda_3+\Lambda_4) \oplus L_{A_4}(5, \Lambda_3+3\Lambda_4)$& & & \\ 
\hline 
$\tilde{L}_{A_4}(5, 2\Lambda_1+2\Lambda_4) $ & $[2002]+[1200]+[2120]+[0212]+[0021]$
%$  	L_{A_4}(2\Lambda_1+2\Lambda_4)\oplus L_{A_4}(\Lambda_1+2\Lambda_2)\oplus L_{A_4}( 2\Lambda_1+\Lambda_2+2\Lambda_3)$  
  	 & $6/5$& $8+4\sqrt{5}$ &$144+64\sqrt5$\\ 	
%  	 & $\oplus  L_{A_4}(2\Lambda_2+\Lambda_3+2\Lambda_4) \oplus L_{A_4}(5, 2\Lambda_3+\Lambda_4)$& & & \\ 
\hline 
$\tilde{L}_{A_4}(5,2\Lambda_1+\Lambda_3) $ & $[2010]+[2201]+[0220]+[1022]+[0102]$
%$  	L_{A_4}(2\Lambda_1+\Lambda_3)\oplus L_{A_4}(2\Lambda_1+2\Lambda_2+ \Lambda_4)\oplus L_{A_4}( 2\Lambda_3+2\Lambda_3)$  
  	 & $1$& $9+4\sqrt{5}$&$161+72\sqrt5$\\ 	
 % 	 & $\oplus  L_{A_4}(\Lambda_1+2\Lambda_3+2\Lambda_4) \oplus L_{A_4}(5, \Lambda_2+2\Lambda_4)$& & & \\ 
\hline 
$\tilde{L}_{A_4}(5, \Lambda_2+\Lambda_3) $ & $[0110]+[1103]+[1030]+[0301]+[3011]$
%$   	L_{A_4}(\Lambda_2+\Lambda_3)\oplus L_{A_4}(\Lambda_1+\Lambda_2+3\Lambda_4)\oplus 
%  	L_{A_4}( \Lambda_1+3\Lambda_3)$  
  	 & $4/5$ & $8+4\sqrt{5}$&$144+64\sqrt5$\\ 	
%  	 & $\oplus  L_{A_4}(3\Lambda_2+\Lambda_4) \oplus L_{A_4}(5, 3\Lambda_1+\Lambda_3+\Lambda_4)$& & & \\ 
\hline 
$\tilde{L}_{A_4}(5,\Lambda_1+\Lambda_2+\Lambda_3+\Lambda_4)^0$ &$[1111]$
% $L_{A_4}(5,\Lambda_1+\Lambda_2+\Lambda_3+\Lambda_4)$ 
&  $3/2$& $5+2\sqrt{5}$&$45+20\sqrt5$\\ \hline 
$\tilde{L}_{A_4}(5,\Lambda_1+\Lambda_2+\Lambda_3+\Lambda_4)^1$ &$[1111]$
% $L_{A_4}(5,\Lambda_1+\Lambda_2+\Lambda_3+\Lambda_4)$ 
&  $3/2$& $5+2\sqrt{5}$&$45+20\sqrt5$\\ \hline
$\tilde{L}_{A_4}(5,\Lambda_1+\Lambda_2+\Lambda_3+\Lambda_4)^2$ &$[1111]$
% $L_{A_4}(5,\Lambda_1+\Lambda_2+\Lambda_3+\Lambda_4)$ 
&  $3/2$& $5+2\sqrt{5}$&$45+20\sqrt5$\\ \hline
$\tilde{L}_{A_4}(5,\Lambda_1+\Lambda_2+\Lambda_3+\Lambda_4)^3$ &$[1111]$ %$L_{A_4}(5,\Lambda_1+\Lambda_2+\Lambda_3+\Lambda_4)$ 
&  $3/2$& $5+2\sqrt{5}$&$45+20\sqrt5$\\ \hline
$\tilde{L}_{A_4}(5,\Lambda_1+\Lambda_2+\Lambda_3+\Lambda_4)^4$ &$[1111]$
% $L_{A_4}(5,\Lambda_1+\Lambda_2+\Lambda_3+\Lambda_4)$ 
&  $3/2$ & $5+2\sqrt{5}$&$45+20\sqrt5$\\ \hline 	
\end{tabular}
\end{table}

\subsection{Extension of $L_{C_{5}}(3,0)\otimes L_{A_{1}}(1,0)$}\label{S:C5A1}

Set $L(C_{5,3}A_{1,1})=L_{C_{5}}(3,0)\otimes L_{A_{1}}(1,0)$.
Let $\tilde{L}(C_{5,3}A_{1,1})= L(C_{5,3}A_{1,1}) \oplus L_{C_5}(3,3\Lambda_5)\otimes L_{A_1}(1,\Lambda_1)$ be a $\Z_2$-graded simple current extension of $L(C_{5,3}A_{1,1})$.
By the argument in Section \ref{S:SCmod}, we obtain the following lemma:
\begin{lemma}\label{L:SCEC5A1}
The VOA $\tilde{L}(C_{5,3}A_{1,1})$ has exactly $28$ inequivalent irreducible modules. Among them, $2$ have integral conformal weights at least $2$; they are as follows:
\[
\begin{split}
\tilde{L}(C_{5,3}A_{1,1})[ 2\Lambda_4,0] &=  L_{C_5}(3,2\Lambda_4)\otimes L_{A_1}(1,0)\oplus  L_{C_5}(3,2\Lambda_1+\Lambda_5)\otimes L_{A_1}(1,\Lambda_1), \\
\tilde{L}(C_{5,3}A_{1,1})[ 3\Lambda_2,0] &= L_{C_5}(3,3\Lambda_2)\otimes L_{A_1}(1,0)\oplus  L_{C_5}(3,3\Lambda_3)\otimes L_{A_1}(1,\Lambda_1).
\end{split}
\] 
In fact, both have conformal weight $2$.
\end{lemma}

\subsection{Extension of $L_{A_7}(4,0)$}\label{S:A7}

Let $\tilde{L}_{A_7}(4,0) = L_{A_7}(4,0)\oplus L_{A_7}(4,4\Lambda_2)\oplus L_{A_7}(4,4\Lambda_4)\oplus L_{A_7}(4,4\Lambda_6)$ be a $\Z_4$-graded simple current extension of $L_{A_7}(4,0)$
\cite[Proposition 3.8 and Corollary 3.21]{Li01} (cf.\ \cite{DLM96}).
Note that the VOA $L_{A_7}(4,0)$ has exactly $8$ inequivalent simple current modules $L_{A_7}(4,0)$ and $L_{A_7}(4,4\Lambda_i)$, $(i=1,2,\dots,7)$ and that among them, 
$L_{A_7}(4,0)$ and $L_{A_7}(4,4\Lambda_{2j})$, $(j=1,2,3)$ have integral conformal weights.
By the argument in Section \ref{S:SCmod}, we obtain the following lemma:

\begin{lemma}\label{L:A7mod}
The VOA $\tilde{L}_{A_7}(4,0)$ has exactly $32$ inequivalent irreducible modules. Among them, $2$ have integral conformal weights at least $2$; they are as follows:
\[
\begin{split}
\tilde{L}_{A_7}(4,\Lambda_3+3\Lambda_6)=& L_{A_7}(4,\Lambda_3+3\Lambda_6)\oplus L_{A_7}(4,\Lambda_1+3\Lambda_4)\\
& \oplus L_{A_7}(3\Lambda_3+\Lambda_6)\oplus L_{A_7}(4,3\Lambda_1+\Lambda_4),\\
\tilde{L}_{A_7}(4,\Lambda_4+\Lambda_5+\Lambda_7)=&L_{A_7}(4,\Lambda_4+\Lambda_5+\Lambda_7)\oplus L_{A_7}(4,\Lambda_2+\Lambda_3+\Lambda_5+\Lambda_6)\\
&\oplus L_{A_7}(4,\Lambda_1+\Lambda_3+\Lambda_4)\oplus L_{A_7}(4,\Lambda_1+\Lambda_2+\Lambda_6+\Lambda_7).
\end{split}
\] 
\end{lemma}

\section{$\Z_2$-orbifold construction and the uniqueness of a holomorphic VOA}\label{S:Orb}

In this section, we will review the $\Z_2$-orbifold construction associated with a holomorphic VOA and an automorphism of order $2$ from \cite{EMS,MoPhD} and a method for proving the uniqueness of a holomorphic VOA by the weight one Lie algebra structure from \cite{LS5}.

Let $V$ be a strongly regular holomorphic VOA and $g\in \Aut (V)$ of order $2$.
Let $V(g)$ be the unique irreducible $g$-twisted $V$-module (cf. \cite[Theorem 1.2]{DLM2}). 
Then there exists a linear automorphism $\phi_1(g)$ of $V(g)$ such that $\phi_1(g)^2=id$ and for all $v\in V$ and $i\in\Z,$  
$$\phi_1(g)Y_{V(g)}(v,z)\phi_1(g)^{-1}=Y_{V(g)}(gv,z).$$ 
Note that $\phi_1(g)$ is unique up to a multiplication of $\pm1$.
Set $\phi_0(g)=g$ and $V(g^0)=V$.  For $0\leq j, k\leq 1$, denote 
\[
W^{(j,k)} = \{ w\in V(g^j) \mid \phi_j(g) w = (-1)^kw\}.
\]

Let $V^g$ be the fixed point subspace of $g$, which is a full subVOA of $V$. Note that $W^{(0,0)}=V^g$ and all $W^{(j,k)}$'s are inequivalent irreducible $V^g$-modules (cf. \cite[Theorem 2]{MT}). It was also shown recently in \cite{CM,Mi} that $V^g$ is strongly regular.
Moreover, any irreducible $V^g$-module is a submodule of $V[g^i]$ for some $i$; there exist exactly $4$ inequivalent irreducible $V^g$-modules (see \cite[Theorem 3.3]{DRX}) and they are represented by $\{W^{(j,k)}\mid 0\le j,k\le 1\}$.
By calculating the $S$-matrix of $V^g$, it was proved in \cite{EMS,MoPhD} that all irreducible $V^g$-modules $W^{(j,k)}$ are simple current modules. It implies that the set of isomorphism classes of irreducible $V^g$-modules, denoted by $R(V^g)$, forms an abelian group of order $4$ under the fusion product. 
We often identify an element in $R(V^g)$ with its representative irreducible $V^g$-module. Then $R(V^g)=\{W^{(j,k)}\mid 0\le j,k\le 1\}$.

Under the assumption
\begin{enumerate}[{\rm (I)}]
\item the conformal weight of $V(g)$ belongs to $(1/2)\Z_{>0}$,
\end{enumerate}
it has been proved in \cite{EMS,MoPhD} that the abelian group $R(V^g)$ is isomorphic to $\Z_2\times\Z_2$.
Moreover,  one can choose the $\phi_1(g)$ so that
\begin{itemize}
\item $W^{(i,j)}\fusion_{V^g} W^{(k, \ell)} \cong  W^{(i+k, j+\ell)}$, where $\fusion_{V^g}$ is the fusion product of $V^g$-modules and the sums $i+k,j+\ell$ are considered in $\Z_2=\{0,1\}$;
\item the conformal weight of $W^{(i,j)}$ belongs to $ji/2+\Z$.
\end{itemize}

\begin{theorem}[\cite{EMS,MoPhD}]\label{Thm:EMS} The $V^g$-module 
$$\widetilde{V}_g=W^{(0,0)}\oplus W^{(1,0)}$$ has a strongly regular holomorphic VOA structure as a $\Z_2$-graded simple current extension of $V^g$.
\end{theorem}
This construction of $\widetilde{V}_g$ is called the \emph{$\Z_2$-orbifold construction} associated with $V$ and $g$.
Note that $\widetilde{V}_g$ is uniquely determined by $V$ and $g$, up to isomorphism, and that $V\cong W^{(0,0)}\oplus W^{(0,1)}$ as $V^g$-modules.

The following dimension formula has been mentioned in \cite{Mo} and proved in \cite{MoPhD}.

\begin{theorem}
\label{Thm:Dimformula} 
Let $V$ be a strongly regular holomorphic VOA of central charge $24$ and $g\in\Aut (V)$ of order $2$ satisfying Condition (I).
Then the following equation holds:
$$\dim (\widetilde{V}_{g})_1=3\dim (V^{g})_1-\dim V(g)_{1/2}-\dim V_1+24.$$
\end{theorem}

In order to prove the uniqueness of a holomorphic VOA equipped with a given weight one Lie algebra, we use the the following theorem in \cite{LS5}:

\begin{theorem}\label{Thm:Rev} 
Let $\g$ be a Lie algebra and $\mathfrak{p}$ a subalgebra of $\g$. 
Let $n\in \Z_{>0}$ and let $W$ be a strongly regular holomorphic VOA of central charge $c$.
Assume that for any strongly regular holomorphic VOA $V$ of central charge $c$ whose weight one Lie algebra is $\g$, there exists an order $n$ automorphism $\sigma$ of $V$ such that the following conditions hold:
\begin{enumerate}[{\rm (a)}]
\item $\g^{\sigma}\cong\mathfrak{p}$;
\item $\sigma$ satisfies Condition (I) and $\widetilde{V}_{\sigma}$ is isomorphic to $W$.
\end{enumerate}
In addition, we assume that any order $n$ automorphism  $\varphi$ of $W$ satisfying condition (I) and the conditions (A) and (B) below belongs to a unique conjugacy class in $\Aut (W)$:
\begin{enumerate}[{\rm (A)}]
\item $(W^\varphi)_1$ is isomorphic to $\mathfrak{p}$;
\item $(\widetilde{W}_\varphi)_1$ is isomorphic to $\g$.
\end{enumerate}
Then any strongly regular holomorphic VOA of central charge $c$ with weight one Lie algebra $\g$ is isomorphic to $\widetilde{W}_\varphi$.
In particular, such a holomorphic VOA is unique up to isomorphism.
\end{theorem}

\section{Quantum dimensions and mirror extensions}
In this section, we review the notion of quantum dimensions from \cite{DJX} and the theory of mirror extensions from \cite{Lin} (see also \cite{DJX1, Xu}). The results will be used in Section \ref{S:inertia} for determining the inertia groups of certain holomorphic VOAs of central charge $24$.  

\subsection{Quantum dimensions}\label{S:QD}

Let $V$ be a VOA of central charge $c$. Let $ M = \bigoplus_{ n \in \Z_{\ge0}} M_{ \lambda +n} $ be an irreducible $V$-module, where $ \lambda$ is the conformal weight of $M$. The \emph{character} of $M$ is defined as
\begin{equation*}
     \ch M(q) := q^{\lambda-c/24}\sum_{n\in \Z_{\ge0}} (\dim {M_{ \lambda +n}}) q^n,
\end{equation*}
where $ q=e^{2\pi \sqrt{-1} z }$ and $ z  $ is in the complex upper half-plane.
For a direct sum $\bigoplus_{i=1}^sM^i$ of irreducible $V$-modules, its character is defined as $\sum_{i=1}^s \ch M^i(q)$.
The following notion of the quantum dimension is introduced by Dong et al.~\cite{DJX}.
\begin{definition}\label{Qdim}
   Suppose $  \ch V(q) $ and $  \ch M(q) $ exist. The \emph{quantum dimension of $M$ over $V$} is defined as
   \begin{equation}
      {\qdim}_V M := \lim_{ y \to 0^+} \frac{ \ch M (\sqrt{-1} y)}{ \ch V ( \sqrt{-1} y)},
   \end{equation}
where $ y$ is a positive real number.
\end{definition}

Now, we assume that $V$ is strongly regular.
It is proved in \cite{Zhu} and \cite{DLM2} that for any irreducible $V$-module $M$, the character $\ch M(q)$ converges to a holomorphic function on the domain $ \left| q \right| <1$.
We also assume that the conformal weight of any irreducible $V$-module is positive, except for $V$ itself. 
Then $V$ is $g$-rational for any finite automorphism $g$ of $V$ (\cite{ADJR}). 
The following fundamental properties of the quantum dimension are established in \cite[Remark 3.5, Propositions 4.9 and 4.17, Lemma 4.14]{DJX}.

\begin{proposition}[{\cite{DJX}}]\label{P:qdim} 
Let $M$, $M^1$, $M^2$ be $V$-modules. Then the following hold:
\begin{enumerate}[{\rm (1)}]%[label=(\roman*)]
   \item $  \qdim_{V} M \ge 1$.
Moreover, $M$ is a simple current module if and only if $ \qdim_V M  =1$;
   \item $\qdim_V$ is multiplicative, that is $ \qdim_V ( M^1 \boxtimes M^2) =\qdim_V  M^1 \cdot \qdim_V M^2$, where $ M^1 \boxtimes M^2$ denotes the fusion product;
   \item  $ \qdim_V M = \qdim_V M^\prime$, where $ M ^\prime $ is the contragredient module of $M$.
\end{enumerate}
\end{proposition}

Combining \cite[Theorem 6.3]{DJX} and Propositions \ref{P:qdim} (1), we obtain the following:

\begin{proposition}[{\cite{DJX}}]\label{P:DJX} Let $A$ be a finite abelian automorphism group of $V$ and let $\chi\in A^*$ be an irreducible character of $A$.
Then the eigenspace $\{v\in V\mid a(v)=\chi(a)v\ (a\in A)\}$ of $A$ with the character $\chi$ is a simple current module for the subVOA $V^A$ of $V$.
In particular, $V$ is an $A^*$-graded simple current extension of $V^A$. 
\end{proposition}
By the definition of the quantum dimension and Propositions \ref{P:qdim} (1) and \ref{P:DJX}, we obtain the following lemma:
\begin{lemma}\label{L:QD} Let $U$ be a strongly regular full subVOA of $V$.
For a simple current module $M$ for $V$, we have $\qdim_{U}V=\qdim_{U}M$.
In particular, for an order $2$ automorphism $g$ of $V$ with $U\subset V^g$, we have $\qdim_{U}V^g=\qdim_{U}V^-$, where $V^-=\{v\in V\mid g(v)=-v\}$.
\end{lemma}
%\color{black}

\subsection{Mirror extensions}
Let $V$ be a strongly regular VOA and let $V^1$ be strongly regular subVOA of $V$ such that the commutant $V^2=\Com_V(V^1)$ is also strongly regular and $\Com_V(V^2)=V^1$.
For $i\in\{1,2\}$, let $I^i$ be the set of all isomorphism classes of irreducible $V^i$-modules.
Let 
\begin{equation}
V=\bigoplus_{M^i\in I^i}m_{M^1,M^2}M^1\otimes M^2\label{Eq:decmirror}
\end{equation}
be the decomposition of $V$ as $V^1\otimes V^2$-modules, where $m_{M,N}$ is the multiplicity and we identify an element of $I^i$ 
with its representative.
For $\{i,j\}=\{1,2\}$, set $J^i=\{M^i\in I^i\mid m_{M^1,M^2}\neq0\ \text{for some } M^j\in I^j\}$.
The fundamental properties on this decomposition are studied in \cite{Lin} (see also \cite{KMi} and \cite{DJX}).

\begin{proposition}\label{P:Mirror} 
\begin{enumerate}[{\rm (1)}]
\item For any $M^1\in J^1$, there exists a unique $M^2\in J^2$ such that $m_{M^1,M^2}\neq0$; we denote this bijection by $\psi:J^1\to J^2$, $M^1\mapsto M^2$.
In addition, $m_{M^1,M^2}=1$ if $m_{M^1,M^2}\neq0$, and $\psi$ is extended to an isomorphism of the fusion rings from $\Z[J^1]$ to $\Z[J^2]$, namely, $\psi(M^1)\boxtimes \psi(M^2)=\psi(M^1\boxtimes M^2)$.
\item Let $K\subset J^1$ such that $\bigoplus_{M\in K}M$ has a simple VOA structure as an extension of $V^1$.
Then $\bigoplus_{M\in K}\psi(M)$ also has a simple VOA structure as an extension of $V^2$.
\item $\qdim_{V^1} M=\qdim_{V^2}\psi(M)$ for all $M\in J^1$.
\item If $V$ is holomorphic, then $I^i=J^i$ for $i\in\{1,2\}$.
\end{enumerate}
\end{proposition}
\begin{proof} (1) follows from \cite[Theorem 3.1]{Lin}.
(2) has been proved in \cite[Theorem 3.5]{Lin}.
The quantum dimensions are determined by the $S$-matrix (see \cite{DJX}).
Hence they are also determined by the fusion rules, and (3) follows from (1).
(4) follows from \cite[Theorem 2]{KMi}.
\end{proof}

\subsection{Embedding of $L_{A_7}(4,0)\otimes L_{A_3}(8,0)$ in $ L_{A_{31}}(1,0)$}
In this subsection, we apply Proposition \ref{P:Mirror} to the embedding $L_{A_7}(4,0)\otimes L_{A_3}(8,0)\subset L_{A_{31}}(1,0)$.
Note that $L_{A_7}(4,0)$ and $L_{A_3}(8,0)$ are commutants of  each other, and they are strongly regular.
By Proposition \ref{P:Mirror}, we obtain $$ L_{A_{31}}(1,0)\cong\bigoplus_{a\in\Pi}L_{A_7}(4,\psi(a))\otimes L_{A_3}(8,a)$$ as $L_{A_7}(4,0)\otimes L_{A_3}(8,0)$-modules, where $\Pi$ is a set of dominant integral weights of $A_3$ at level $8$.

It is known (cf.\ \cite[Section 3.1.3]{Xu}) that $L_{D_{10}}(1,0)$ contains $L_{A_3}(8,0)$ as a full subVOA and 
\begin{equation}
L_{D_{10}}(1,0)\cong\bigoplus_{a\in \Pi_0} L_{A_3}(8,a)\label{Eq:D10}
\end{equation}
as $L_{A_3}(8,0)$-modules, where $\Pi_0$ is the subset of $\Pi$ consisting of the weights in Table \ref{tableMir}; we use the notation $[i_1\ i_2\ \dots i_\ell]$ to denote the weight $\sum_{j=1}^\ell i_j\Lambda_j$.
Note that all the multiplicities are one in the decomposition \eqref{Eq:D10}.

In \cite[Theorem 4.1]{OS} (cf.\ \cite{Xu}),  the map $\psi$ is explicitly described in more general setting: $L_{A_{n-1}}(m,0)\otimes L_{A_{m-1}}(n,0)\subset L_{A_{mn-1}}(1,0)$; we list $\psi(a)$ for $a\in\Pi_0$ in Table \ref{tableMir}.
%\color{black}

\begin{table}[bht] 
\caption{The set $\Pi_0$ and the map $\psi$ 
}\label{tableMir}
%\tiny
\begin{tabular}{|c|c|c|c|}
\hline
$a\in\Pi_0$ &$\psi(a)$ &  Conformal weight of $L_{A_3}(8,a)$& $\dim L_{A_3}(8,a)_1$\\ \hline \hline 
[000]&[0000000]&$0$&$15$\\\hline
[800]&[0400000]&$3$&$0$\\\hline 
[080]&[0004000]&$4$&$0$\\\hline 
[008]&[0000040]&$3$&$0$\\\hline 
[214]&[1011000]&$2$&$0$\\\hline 
[121]&[1100011]&$1$&$175$\\\hline 
[141]&[0110110]&$2$&$0$\\\hline 
[412]&[0001101]&$2$&$0$\\\hline 
\end{tabular}
\end{table}

\section{Inertia groups}\label{S:inertia}
Throughout this section, we also adopt the labeling of simple roots $\alpha_i$ of an irreducible root system as in \cite[Section 11.4]{Hu}.

\subsection{Definition of the inertia group}
Let $V$ be a VOA of CFT-type. 
We first define the inertia group of a VOA.
\begin{definition}
Define the \textit{inertia group} $I(V)$ of $V$ by
\[
I(V)=\{g\in \Aut (V)\mid g(x)=x \text{ for all } x\in V_1\}. 
\]
\end{definition} 

Assume that $V$ is strongly regular and that $V_1$ is semisimple.
For a subalgebra $\mathfrak{a}$ of $V_1$, let $L(\mathfrak{a})$ denote the subVOA of $V$ generated by $\mathfrak{a}$.
Let $V_1=\g_{1} \oplus \cdots \oplus \g_{r}$ be the decomposition of $V_1$ as a direct sum of simple ideals. 
Then by Proposition \ref{Prop:posl}, we have 
\[
L(\g_i)\cong L_{\g_i}(k_i,0)\quad \text{and}\quad L(V_1) \cong L_{\g_1}(k_1, 0) \otimes \cdots \otimes  L_{\g_r}(k_r, 0),
\]
where $k_i$ is the (positive integral) level of the corresponding affine representation of $\g_i$ for $1\leq i\leq r$. 
We often denote $\g_i$ by its type $X_{n,k}$ if the type is $X_n$ and the level is $k$.
We also assume $L(V_1)$ is a full subVOA of $V$.
%, that is, the conformal vectors of $L(V_1)$ and $V$ are the same. 
Then any irreducible $L(V_1)$-submodule of $V$ is isomorphic to $\bigotimes_{i=1}^rL_{\g_i}(k_i,\lambda_i)$, where $\lambda_i$ is a dominant integral weight of $\g_i$ at level $k_i$.
By definition, $I(V)$ acts trivially on $L(V_1)$, and any element of $I(V)$ induces an $L(V_1)$-module isomorphism of $V$. 
%\color{black}

\begin{lemma}\label{L:mult}
Let $M$ be an irreducible $L(V_1)$-submodule of $V$ and $g\in I(V)$. 
If the multiplicity of $M$ in $V$ is $1$, then 
$g|_{M}=\lambda\cdot id_M$ for some constant $\lambda\in \C$. 
\end{lemma} 

\begin{proof}
Since $M$ has multiplicity one in $V$, $g$ induces an $L(V_1)$-module isomorphism from $M$ to $M$. Hence, by Schur's lemma, we have the desired conclusion.  
\end{proof}

\begin{remark}
For a simple current $L(V_1)$-submodule $W$ of $V$, by the same argument as in \cite[Proposition 5.1 (2)]{DMZ}, the multiplicity of $W$ in $V$ is one. Hence,   
$g|_{W}=\lambda\cdot id_W$ for some $n$-th root of unity $\lambda$, where $n$ is the smallest positive integer such that $W^{\fusion_{L(V_1)}^n} \cong  L(V_1)$.
\end{remark}

%\subsection{Lattice VOA}
%Let $V_N$ be a lattice VOA associated with a Niemeier lattice $N$. Let $Q$ be the corresponding root lattice. 
The inertia groups of the Niemeier lattice VOAs are determined in \cite{LS5} as follows:
\begin{proposition}
Let $N$ be a Niemeier lattice and let $Q$ be the root sublattice of $N$. 
If $Q\neq\emptyset$, then $I(V_N)=\{\sigma_u\mid u\in Q^*/N\}\cong Q^*/N$, where $Q^*=\{v\in \R\otimes_\Z Q\mid (v|Q)\in\Z\}$ and $\sigma_u= \exp(-2 \pi \sqrt{-1} u_{(0)})$.  
\end{proposition} 

\begin{remark}  
In general, if $V$ is a $D$-graded simple current extension of $L(V_1)$, then $I(V)= D^*$, the set of all linear characters of $D$.
\end{remark}

\begin{remark} For any coweight $\lambda$ of $V_1$, $\sigma_\lambda\in I(V)$.
In fact, in the following cases, we prove that $I(V)$ is the set of such inner automorphisms.
However, we do now know whether this is true in general.
\end{remark}

In the following subsections, we will compute the inertia groups for some explicit examples of holomorphic VOAs of central charge $24$ (see Main Theorem 2 in Introduction for a summary). 
The results will be used to show the following theorem:
%that certain automorphisms of  holomorphic VOAs of central charge $24$ are uniquely determined up to conjugation. 

\begin{theorem}\label{T:conj} Let $V$ be a strongly regular holomorphic VOA of central charge $24$ and let $g\in\Aut (V)$  be of order $2$.
Assume that the types of $V_1$ and $V_1^g$ are given as in one of the rows of Table \ref{Ta:conj}.
%has the type $E_{6,3}G_{2,1}^3$, $A_{4,5}^2$, $C_{5,3}C_{2,2}A_{1,1}$ or $A_{7,4}A_{1,1}^3$.
%Let $g\in\Aut (V)$ of order $2$ such that the type of $V^g_1$ is given in a row of Table \ref{Ta:conj}.
Then the conjugacy class of $g$ is uniquely determined by $V^g_1$, conformal weight of $V(g)$ and the additional data given in the same row of Table \ref{Ta:conj}.
\end{theorem}

%\begin{tiny}
\begin{longtable}[bht]{|c|c|c|c|c|} 
\caption{Uniqueness of some order $2$ automorphisms.
}\label{Ta:conj} \\
\hline 
$V_1$  & $V^g_1$&conf. wt of $V(g)$& additional data&reference\\
 \hline%\hline 
$E_{6,3}G_{2,1}^3$ &$D_5G_2A_1^4U(1)$&&&Lemma \ref{L:conjD73}\\\hline 
$A_{4,5}^2$&$C_2^2$&&&Lemma \ref{L:conjA45}\\\hline 
$C_{5,3}C_{2,2}A_{1,1}$ &$A_{4}A_{1}^3U(1)^2$&&&Lemma \ref{L:conjC53}\\\hline 
$A_{7,4}A_{1,1}^3$&$A_6A_1^2U(1)^2$&$1$&&Lemma \ref{L:conjA74}\\\hline 
$A_{7,4}A_{1,1}^3$&$A_4A_2A_1^2U(1)^2$&$1$&$(\tilde{V}_g)_1\not\cong A_7A_1^3$&Lemma \ref{L:conjA74-2}\\\hline 
$A_{7,4}A_{1,1}^3$&$A_5A_1^4U(1)$&$1$&${\rm (Lie)\ rank}(\tilde{V}_g)_1=10$&Lemma \ref{L:conjA74-3}\\\hline 
\end{longtable}
%\end{tiny}
%\color{black}
\subsection{Weight one Lie algebra of type $E_{6,3}G_{2,1}^3$}

Let $V$ be a strongly regular holomorphic VOA of central charge $24$ such that 
$V_1$ has the type $E_{6,3}G_{2,1}^3$. 
It was proved in \cite{LS5} that $V$ is isomorphic to the holomorphic VOA constructed by Miyamoto \cite{Mi3}.

Let $V_1=\bigoplus_{i=1}^4\g_i$ be the direct sum of simple ideals $\g_i$, where $\g_1$ and $\g_i$ $(i=2,3,4)$ are simple ideals of type $E_{6,3}$ and $G_{2,1}$, respectively.
Set $\g_{234}=\bigoplus_{i=2}^4\g_i$.

\begin{lemma}\label{L:E63}
The commutant $\Com_{V}(L(\g_{234}))$ is isomorphic to $\tilde{L}_{E_6}(3,0)$, the $\Z_3$-graded simple current extension of $L_{E_6}(3,0)$ described in Section \ref{S:EE63}.
In addition, we have $\Com_V(\Com_{V}(L(\g_{234})))=L(\g_{234})$.
\end{lemma}
\begin{proof}Recall that there are exactly $20$ inequivalent irreducible $L_{E_6}(3,0)$-modules.
One can  check directly that only three of them, $L_{E_6}(3,0)$, $L_{E_6}(3,3\Lambda_1)$ and $L_{E_6}(3,3\Lambda_6)$, have integral conformal weights.
Note that they are simple current modules and form a cyclic group of order $3$ under the fusion product.
By Proposition \ref{P:ComSimple}, $\Com_{V}(L(\g_{234}))$ is simple.
Hence, we have $\Com_{V}(L(\g_{234}))\cong L_{E_6}(3,0)$ or $\tilde{L}_{E_6}(3,0)$.

The subVOA $L(\g_{234})$ has exactly $8$ inequivalent irreducible modules and itself is the only one with integral conformal weight.
Hence $\Com_V(\Com_V(L(\g_{234})))=L(\g_{234})$.
By Proposition \ref{P:Mirror} (1) and (4), the commutant $\Com_{V}(L(\g_{234}))$ also has exactly $8$ irreducible modules.
Hence $\Com_{V}(L(\g_{234}))\cong\tilde{L}_{E_6}(3,0)$.
\end{proof}

By the lemma above, we now identify $\Com_{V}(L(\g_{234}))$ with $\tilde{L}_{E_6}(3,0)$.
Note that $L(\g_{234})\cong L_{G_2}(1,0)^{\otimes3}$.

\begin{lemma}\label{L:decE6} As a $\tilde{L}_{E_6}(3,0)\otimes L_{G_2}(1,0)^{\otimes3}$-module, the VOA $V$ decomposes as
\begin{align*}
V\cong &\tilde{L}_{E_6}(3,0)\otimes L_{G_2}(1,0)^{\otimes3}\oplus \bigoplus_{i=1}^3\left(\tilde{L}_{E_6}(3,\Lambda_4)^{s_i}\otimes (L_{G_2}(1,\Lambda_1)\otimes L_{G_2}(1,0)^{\otimes 2})^{\xi^{i-1}}\right)\oplus\\
&\bigoplus_{i=1}^3\left(\tilde{L}_{E_6}(3,\Lambda_1+\Lambda_6)^{t_i}\otimes(L_{G_2}(1,0)\otimes L_{G_2}(1,\Lambda_1)^{\otimes2})^{\xi^{i-1}}\right)\oplus\tilde{L}_{E_6}(3,\Lambda_2)\otimes L_{G_2}(1,\Lambda_1)^{\otimes3},
\end{align*} for some $\{s_1,s_2,s_3\}=\{t_1,t_2,t_3\}=\{0,1,2\}$, 
up to permutations on the latter three tensor components, where $\xi$ is an order $3$ cyclic permutation on the latter three tensor components.
\end{lemma}
\begin{proof} Recall that $L_{G_2}(1,0)$ has exactly two inequivalent irreducible modules $L_{G_2}(1,0)$ and $L_{G_2}(1,\Lambda_1)$ and they have conformal weights $0$ and $2/5$, respectively.
Hence by Table \ref{tableE6}, all irreducible $\tilde{L}_{E_6}(3,0)\otimes L_{G_2}(1,0)^{\otimes3}$-modules with integral conformal weights are given as in Table \ref{E63G21^3-module}, up to permutation on the latter $3$ copies of $L_{G_2}(1,0)$.
By \eqref{Eq:decmirror} and Proposition \ref{P:Mirror} (4), we obtain this lemma.
\end{proof}
\begin{lemma}\label{E6id}
Any element of $I(V)$ acts on $\tilde{L}_{E_6}(3,0)(=\Com_{V}(L(\g_{234})))$ as the identity.
\end{lemma}

\begin{proof}
Since $g$ stabilizes $L(\g_{234})$, it also stabilizes the commutant $\tilde{L}_{E_6}(3,0)$.
It follows from $g\in I(V)$ that $g$ acts trivially on the subVOA $L_{E_6}(3,0)$ of $\tilde{L}_{E_6}(3,0)$.
Since $\tilde{L}_{E_6}(3,0)$ is a simple current extension of $L_{E_6}(3,0)$ graded by $D\cong \Z_3$, the restriction of $g$ to $\tilde{L}_{E_6}(3,0)$ belongs to $D^*$.
By Lemma \ref{L:DW}, the action of $D^*$ on $\{\tilde{L}_{E_6}(3,\Lambda_4)^j\mid j=0,1,2\}$ by conjugation is regular.
Note that the restriction of $g$ to $L(\g_{234})$ is the identity and it stabilizes any irreducible  $L(\g_{234})$-module by the conjugation.

Suppose, for a contradiction, that the restriction of $g$ to $\tilde{L}_{E_6}(3,0)$ is not the identity.
Then $g$ does not preserves the decomposition of $V$ into $\tilde{L}_{E_6}(3,0)\otimes L(\g_{234})$-modules (see Lemma \ref{L:decE6}), which contradicts Lemma \ref{L:Vg};
indeed, the $g$-conjugation of $\tilde{L}_{E_6}(3,\Lambda_4)^{s_1}\otimes L_{G_2}(1,\Lambda_1)\otimes L_{G_2}(1,0)^{\otimes 2}$ is isomorphic to $\tilde{L}_{E_6}(3,\Lambda_4)^{x}\otimes L_{G_2}(1,\Lambda_1)\otimes L_{G_2}(1,0)^{\otimes 2}$ for some $x\neq s_1$, which does not appear in the decomposition in Lemma \ref{L:decE6}.
Hence we obtain the result.
\end{proof}

\begin{table}[bht]
\caption{Irreducible modules of  $\tilde{L}_{E_6}(3,0)\otimes L_{G_2}(1,0)^{\otimes3}$ with integral conf. weight} \label{E63G21^3-module}
\begin{tabular}{|l|c|}
\hline
Irreducible module & conformal weight \\\hline
$\tilde{L}_{E_6}(3,0)\otimes L_{G_2}(1,0)^{\otimes3}$& $0$\\\hline
$\tilde{L}_{E_6}(3,\Lambda_4)^i\otimes L_{G_2}(1,\Lambda_1)\otimes L_{G_2}(1,0)^{\otimes2}$ $(i=1,2,3)$& $2$ \\\hline
$\tilde{L}_{E_6}(3,\Lambda_1+\Lambda_6)^i\otimes L_{G_2}(1,\Lambda_1)^{\otimes2}\otimes L_{G_2}(1,0)$ $(i=1,2,3)$&$2$\\\hline
$\tilde{L}_{E_6}(3,\Lambda_2)\otimes L_{G_2}(1,\Lambda_1)^{\otimes3}$&$2$ \\\hline
\end{tabular}
\end{table}

\begin{theorem}\label{IE63}
Let $V$ be a strongly regular holomorphic VOA of central charge $24$ such that 
$V_1$ has the type $E_{6,3}G_{2,1}^3$. Then $I(V)=1$. 
\end{theorem}

\begin{proof}
By Lemma \ref{L:E63}, we identify $\Com_{V}(L(\g_{234}))$ with $\tilde{L}_{E_6}(3,0)$.
Set $U=\tilde{L}_{E_6}(3,0)\otimes L(\g_{234})$.
Let $g\in I(V)$ and let $M$ be an irreducible $U$-submodule of $V$ such that $M\neq U$.
By Lemma \ref{L:decE6}, the multiplicity of $M$ is one. 
Then $g|_{U} =id $ by Lemma \ref{E6id}, and $g$ acts as a scalar $\lambda_M\in\C^\times$ on $M$ (see Lemma \ref{L:mult}). 
Recall the following fusion product of $L_{G_2}(1,0)$: $$L_{G_{2}}(1, \Lambda_1) \fusion_{L_{G_{2}}(1, 0)} L_{G_{2}}(1, \Lambda_1) = L_{G_{2}}(1,0) \oplus L_{G_{2}}(1, \Lambda_1),$$ 
which shows that any irreducible $L(\g_{234})$-module is self-contragredient.
By Proposition \ref{P:Mirror} (1), any irreducible $\tilde{L}_{E_6}(3,0)$-module is also self-contragredient.
Hence $M$ is self-contragredient, and $\lambda_M\in\{\pm1\}$. 

Suppose, for a contradiction, that $\lambda_M=-1$.
By the fusion product above, $U\oplus M$ is a subVOA of $V$.
Since both $U$ and $M$ are self-contragredient, the restriction of the invariant form of $V$ to $U\oplus M$ is non-degenerate.
Hence $U\oplus M$ is simple.
Since the conformal weight of $M$ is at least $2$, we have $(U\oplus M)_1=U_1$.
Since $U$ is strongly regular, so is $U\oplus M$ by Theorem \ref{T:HKL}.
By $\lambda_M=-1$, the restriction of $g$ to $U\oplus M$ has order $2$ and $M$ is the $-1$-eigenspace.
By Proposition \ref{P:DJX}, $M$ must be a simple current module, which is a contradiction;
indeed, $M$ has a non-simple current module $L_{G_2}(1,\Lambda_1)$ as a component of the tensor product.
Thus $\lambda_M=1$ for all irreducible $U$-submodules $M$ of $V$, and $g$ is the identity on $V$.
\end{proof}

\begin{remark} For any coweight $\lambda$ of $V_1$, we have $\sigma_{\lambda}=id_V$ on $V$.
Indeed, the weights of $E_{6}$ in the decomposition \eqref{L:decE6} (cf.\ Table \ref{tableE6}) are  all in the root lattice of $E_6$, and for $G_2$, $(\Lambda_j^\vee|\Lambda_i)\in\Z$, $i,j\in\{1,2\}$. 
\end{remark}

The following lemma is obtained from the explicit construction of $V$ from the lattice VOA associated with the Niemeier lattice $N(E_6^4)$ in \cite{Mi3}.

\begin{lemma}\label{L:permG23} There exists an automorphism of $V$ which acts on the three simple ideals of $V_1$ of type $G_2$ as an order $3$ permutation and acts trivially on the simple ideal of type $E_6$.
\end{lemma}
\begin{proof} Let $\tau$ be an order $3$ automorphism of $V_{N(E_6^4)}$ so that the orbifold construction associated with $V_{N(E_6^4)}$ and $\tau$ gives $V$;
it can be realized as $\tau=\varphi\cdot \sigma_{\Lambda_4/3}$, where $\varphi$ is an order $3$ lift of an order $3$ isometry $\varphi_0$ of $N(E_6^4)$ which acts on three root lattices of type $E_6$ as an order $3$ permutation and $\Lambda_4$ is the fundamental weight for the remaining root lattice of type $E_6$ (cf. \cite{LS5}).
Note that the Lie subalgebra $(V_{N(E_6^4)}^\tau)_1$ has type $E_{6,3}A_{2,1}^3$.

Let $f\in\Aut (V_{N(E_6^4)})$ be a lift of an isometry of $N(E_6^4)$ which acts trivially on the three root lattices of type $E_6$ permuted by $\varphi_0$ and acts on the remaining one by a diagram automorphism of order $3$ of the extended Dynkin diagram of type $E_6$; in addition, we may choose $f$ so that $f$ commutes with $\tau$.
Then $f$ preserves $V_{N(E_6^4)}^\tau$.
Let $\bar{f}$ denote the restriction of $f$ to $V_{N(E_6^4)}^\tau$.
Then $\bar{f}$ acts on the three simple ideals of type $A_{2,1}^3$ as an order $3$ permutation and acts trivially on the simple ideal of type $E_6$.
Note that $V_{N(E_6^4)}^\tau$ has exactly $5$ inequivalent  irreducible modules with integral conformal weights; 
three of them are eigenspaces of $\tau$ in $V_{N(E_6^4)}$ and the other two are subspaces of irreducible $\tau^i$-twisted $V_{N(E_6^4)}$-modules.
Since $\bar{f}$ preserves the set of three eigenspaces of $\tau$ by conjugation, it also preserves the set of the other two.
Note that $(\tilde{V}_{N(E_6^4)})_\tau$ is a simple current extension of $V_{N(E_6^4)}^\tau$.
Hence by \cite[Theorem 3.3]{Sh04},  $\bar{f}$ can be extended to an automorphism $\hat{f}$ of $V(\cong(\tilde{V}_{N(E_6^4)})_\tau)$.
Since every simple ideal of type $A_2$ of $(V_{N(E_6^4)}^\tau)_1$ is contained in a simple ideal of type $G_2$ of $V_1$, $\hat{f}$ satisfies the desired conditions.
\end{proof}

\begin{lemma}\label{L:conjD73}  
The conjugacy class of an order $2$ automorphism $g$ of $V$ in $\Aut (V)$ is unique if $V^g_1$ is a  Lie algebra of type $D_{5}G_{2}A_{1}^4U(1)$.
\end{lemma}
\begin{proof} 
Let $\bar{g}$ be the restriction of $g$ to $V_1$.
Since the Lie ranks of $V_1$ and $V^g_1$ are the same, $\bar{g}$ is inner and it preserves every simple ideal of $V_1$.
By the type of $V^g_1$ (cf.\ \cite[Section 8]{Kac}), %Since the simple Lie algebra of type $E_6$ does not contain a semisimple Lie algebra of type $D_5A_1$, 
the $\bar{g}$-fixed points of the simple ideals of type $E_{6}$, $G_2$, $G_2$, $G_2$ have the type $D_{5}U(1)$, $A_1^2$, $A_1^2$, $G_2$, respectively.
%and $\bar{g}$ acts trivially on one simple ideal of type $G_2$ and the $\bar{g}$-fixed points of the other ideal of type $G_2^2$ has the type $A_1^4$.
By Lemma \ref{L:permG23}, we may specify the simple ideal of type $G_2$ on which $g$ acts trivially, up to conjugation by an element in $\Aut (V)$.
Hence, by Lemma \ref{L:conjinv}, $\bar{g}$ is unique up to conjugation by an element in $\Aut (V)$.

Let $g'\in\Aut (V)$ such that $V^{g'}_1\cong V^g_1$.
Then there exists $x\in\Aut (V)$ such that $g'_{|V_1}=(xgx^{-1})_{|V_1}$ by the argument above.
Clearly $g'xgx^{-1}=id$ on $V_1$, that is, $g'xgx^{-1}\in I(V)$.
Since $I(V)=1$ by Theorem \ref{IE63}, we have $g'=xgx^{-1}$.
Thus we obtain the desired result.    
\end{proof}

\subsection{Weight one Lie algebra of type $A_{4,5}^2$}

Let $V$ be a strongly regular holomorphic VOA of central charge $24$ such that 
$V_1$ has the type $A_{4,5}^2$. 
Recall that such a VOA has been constructed in \cite{EMS} and the uniqueness has also been achieved in \cite{EMS2} (see \cite{LS6} for an alternative construction and an alternative proof for uniqueness).

Let $V_1=\mathfrak{g}_1\oplus \mathfrak{g}_2$ be the decomposition into the direct sum of simple ideals of type $A_{4,5}$.
Let $L(\mathfrak{g}_1)$ be the subVOA generated by $\mathfrak{g}_1$.
Then $L(\mathfrak{g}_1)\cong L_{A_4}(5,0)$, and the commutant $\Com_V(L(\mathfrak{g}_1))$ contains the subVOA $L(\mathfrak{g}_2)(\cong L_{A_4}(5,0))$ generated by $\mathfrak{g}_2$.

\begin{lemma}\label{L:ComA45} For $i=1,2$, the commutant $\Com_V(L(\mathfrak{g}_i))$ is isomorphic to the $\Z_5$-graded simple current extension $\tilde{L}_{A_4}(5,0)$ of $L_{A_4}(5,0)$ described in Section \ref{S:A45}.
\end{lemma}
\begin{proof} It is enough to discuss the case $i=1$.
By the explicit construction of $V$ from the lattice VOA associated with the Niemeier lattice with root lattice $A_{4}^6$ in \cite{EMS}, one can directly check that $\Com_V(L(\mathfrak{g}_1))$ contains irreducible $L(\mathfrak{g}_2)$-submodules isomorphic to $L_{A_4}(5,5\Lambda_i)$, $i=1,2,3,4$ (see also \cite[Case 11 in Lemma 6.4]{EMS2}) and their products are non-zero.
Since these are simple current modules and closed under the fusion products, $\Com_V(L(\mathfrak{g}_1))$ contains a simple subVOA isomorphic to $\tilde{L}_{A_4}(5,0)(= L_{A_4}(5,0)\oplus\bigoplus_{i=1}^4L_{A_4}(5,5\Lambda_i))$.
By Table \ref{tableA4}, there are no irreducible $\tilde{L}_{A_4}(5,0)$-modules with integral conformal weights at least $2$, and hence $\Com_V(L(\mathfrak{g}_1))$ is isomorphic to $\tilde{L}_{A_4}(5,0)$.
\end{proof}

Let $U=\Com_V(L(\mathfrak{g}_2))\otimes \Com_V(L(\mathfrak{g}_1))$.
By the lemma above, $U\cong \tilde{L}_{A_4}(5,0)^{\otimes 2}$.
Note that $U$ is a $\Z_5^2$-graded simple current extension of $L_{A_4}(5,0)^{\otimes 2}$; we denote by $D$ the abelian group $\Z_5^2$ associated with the grading.
Then the dual $D^*(\cong D)$ of $D$ is naturally embedded to $\Aut (U)$.

\begin{lemma}\label{L:DecA4}
As a module of $U$, the VOA $V$ decomposes as 
\[
\tiny 
\begin{split}
V\cong &\left(\tilde{L}_{A_4}(5,0)\otimes \tilde{L}_{A_4}(5,0)\right)\oplus  \left( \tilde{L}_{A_4}(5, \Lambda_1+\Lambda_4)\otimes \tilde{L}_{A_4}(5,\Lambda_1+\Lambda_2+\Lambda_3+\Lambda_4)^{j_0}\right)\\
& \oplus \left(\tilde{L}_{A_4}(5,\Lambda_1+\Lambda_2+\Lambda_3+\Lambda_4)^{i_0} \otimes \tilde{L}_{A_4}(5, \Lambda_1+\Lambda_4)\right) \oplus  \left(\tilde{L}_{A_4}(5,2\Lambda_1+\Lambda_3)\otimes  \tilde{L}_{A_4}(5,2\Lambda_1+\Lambda_3)\right)\\&
\oplus  \left( \tilde{L}_{A_4}(5,2\Lambda_1+2\Lambda_4) \otimes \tilde{L}_{A_4}(5,\Lambda_2+\Lambda_3) \right)
\oplus \left( \tilde{L}_{A_4}(5,\Lambda_2+\Lambda_3)  \otimes \tilde{L}_{A_4}(5,2\Lambda_1+2\Lambda_4)\right) \\
& \oplus\bigoplus_{k=1}^4  \left( \tilde{L}_{A_4}(5,\Lambda_1+\Lambda_2+\Lambda_3+\Lambda_4)^{i_k} 
      \otimes \tilde{L}_{A_4}(5,\Lambda_1+\Lambda_2+\Lambda_3+\Lambda_4)^{j_k} \right)
\end{split}
\]
for some $\{i_k\mid k=0,1,2,3,4\}=\{j_k\mid k=0,1,2,3,4\}=\{0,1,2,3,4\}$.
\end{lemma}

\begin{proof}
We note that the conformal weights of all irreducible $\tilde{L}_{A_4}(5,0)$-modules are summarized in Table \ref{tableA4}.
It follows from $V_1=U_1$ that the conformal weight of any irreducible $U$-submodule is an integer at least $2$ except for $U$ itself.
By Proposition \ref{P:Mirror} (1) and (4), every irreducible $\tilde{L}_{A_4}(5,0)$-module appears in the first and second component of the tensor product once.
Hence we obtain the desired decomposition for some $i_k$ and $j_k$.
\end{proof}

\begin{lemma}\label{L:gUA4}
Any element of $I(V)$ acts on $U$ as the identity.
\end{lemma}
\begin{proof} The proof is similar to that of Lemma \ref{E6id}.
Let $g\in I(V)$.
By $U=\Com_V(L(\mathfrak{g}_2))\otimes\Com_V(L(\mathfrak{g}_1))$, we have $g(U)=U$ and $g_{|U}\in D^*$.
Since $U$ is a simple current extension of $L(V_1)$ graded by $D\cong\Z_5^2$, by Lemma \ref{L:DW} (see also Table \ref{tableA4}), the conjugation by $D^*\subset\Aut (V)$ on the set of isomorphism classes of irreducible $\tilde{L}(A_{4,5})^{\otimes2}$-modules $\tilde{L}_{A_4}(5,\Lambda_1+\Lambda_2+\Lambda_3+\Lambda_4)^i\otimes\tilde{L}_{A_4}(5,\Lambda_1+\Lambda_2+\Lambda_3+\Lambda_4)^j$, $0\le i,j\le 4$, is regular.
If $g_{|U}$ is not the identity, then $V\circ g\not\cong V$ as $U$-modules by Lemma \ref{L:DecA4}, which contradicts Lemma \ref{L:Vg}.
\end{proof}

\begin{theorem}\label{IA45}
Let $V$ be a strongly regular holomorphic VOA of central charge $24$ such that 
$V_1$ has the type $A_{4,5}^2$. Then $I(V)=1$. 
\end{theorem}

\begin{proof}
Let $g\in I(V)$. By Lemma \ref{L:gUA4} $g|_U=id_U$. By Lemmas \ref{L:mult} and \ref{L:DecA4}, $g$ acts as a scalar on every irreducible $U$-submodule of $V$. 
Since all of them are self-contragredient modules (see Lemma \ref{L:A4mod}), these scalars are $\pm 1$.   

Set $T=\Com_V(L(\g_2))$ and $T^c=\Com_V(L(\g_1))$.
Then $U=T\otimes T^c$.
By Theorem \ref{T:HKL}, both $T$ and $T^c$ are strongly regular.
Let $S_T$ and $S_{T^c}$ be the set of all inequivalent irreducible $T$- and $T^c$-modules, respectively.
By Proposition \ref{P:Mirror} (1) and (4), there exists a bijection $\varphi:S_T\to S_{T^c}$ such that as $T\otimes T^c$-modules,
$$V\cong\bigoplus_{M\in S_T} M\otimes \varphi(M).$$  

Let $S_T^{\pm} =\{M\in S_T\mid g|_{M\otimes \varphi(M)}=\pm id_{M\otimes \varphi(M)}\}$. Then the $\pm1$-eigenspace of $g$ is $\bigoplus_{M\in S_T^\pm} M\otimes \varphi(M)$.
If $g\neq id_V$, then the $-1$-eigenspace is a simple current module of $V^g$ by Proposition \ref{P:DJX}.
By Proposition \ref{P:qdim} (1) and Lemma \ref{L:QD}, we have the following equality
\[
\sum_{M\in S_T^+} ({\qdim}_T M)^2 = \sum_{M\in S_T^-} ({\qdim}_T M)^2.
\]
Here we use the fact that $\qdim_T M=\qdim_{T^c} \varphi(M)$ (see Proposition \ref{P:Mirror} (3)).
However, there are no partition $S_T=S_T^+\cup S_T^-$ satisfying the equation above;
indeed, by Table \ref{tableA4}, $$\sum_{M\in S_T}({\qdim}_T M)^2=720+320\sqrt5$$ and there are no subsets $I\subset S_T$ such that $\sum_{M\in I}(\qdim_T M)^2=360+160\sqrt5$. Hence $S_T=S_T^+$, and $g=id_V$. 
\end{proof}

\begin{remark} Since the dominant integral weights for $A_4$ in Table \ref{tableA4} are roots, we have $\sigma_\lambda=id$ on $V$ for any (co)weight $\lambda$ of $V_1$.
\end{remark}

By using the lemma above, we obtain the following:

\begin{lemma}\label{L:conjA45} 
The conjugacy class of an order $2$ automorphism $g$ of $V$ is unique if $V^g_1$ is a semisimple Lie algebra of type $C_{2}^2$.
\end{lemma}
\begin{proof} Let $\bar{g}$ be the restriction of $g$ to $V_1$.
Then $\bar{g}$ preserves every simple ideal of type $A_{4}$; otherwise $V_1^{\bar{g}}$ has the type $A_{4}$.
Hence both the $\bar{g}$-fixed points of the simple ideals have the type $C_2$.
Note that the restriction of $\bar{g}$ to every simple ideal is not inner.
By \cite[Exercise 8.10]{Kac} (see also \cite[Lemma 7.4]{EMS2}), $\bar{g}$ is unique up to conjugation by an inner automorphism of $V_1$.

By Theorem \ref{IA45}, we have $I(V)=1$, and one can apply the same argument as in Lemma \ref{L:conjD73}, and 
we obtain the desired result.    
\end{proof}

\subsection{Weight one Lie algebra of type $C_{5,3}G_{2,2}A_{1,1}$}

Let $V$ be a strongly regular holomorphic VOA of central charge $24$ such that 
$V_1$ has the type $C_{5,3}G_{2,2}A_{1,1}$. Such a VOA has been constructed in \cite{LS3,EMS} and the uniqueness has been proved in \cite{EMS2}. 

Let $\g_1$, $\g_2$ and $\g_3$ be simple ideals of $\g$ of type $C_{5,3}$, $G_{2,2}$ and $A_{1,1}$, respectively.
Set $\g_{23}=\g_2\oplus \g_3$.
Then $V_1=\g_1\oplus\g_2\oplus\g_3$ and $\Com_V(L(\g_{23}))$ contains $L(\g_1)\cong L_{C_5}(3,0)$.
By Theorem \ref{T:HKL} and Proposition \ref{P:ComSimple}, $\Com_V(L(\g_{23}))$ is strongly regular.
Notice that $L_{C_5}(3,0)$ has exactly $3$ inequivalent irreducible modules whose conformal weights are integers at least $2$, namely, 
\[
L_{C_5}(3,2\Lambda_4),\quad  L_{C_5}(3,3\Lambda_2),\quad \text{ and }\quad  L_{C_5}(3,\Lambda_2+2\Lambda_5), 
\]
and that these $3$ modules are not simple current modules.
Therefore,  
$$
\Com_V(L(\g_{23}))\cong L_{C_5}(3,0) \oplus m_1 L_{C_5}(3,2\Lambda_4) \oplus m_2  L_{C_5}(3,3\Lambda_2)\oplus m_3 L_{C_5}(3,\Lambda_2+2\Lambda_5)
$$
as $L_{C_5}(3,0)$-modules, where $m_1,m_2,m_3$ are multiplicities.

\begin{lemma}\label{m3=0}
We have $m_3=0$. 
\end{lemma}

\begin{proof} Suppose, for a contradiction, that $m_3\neq0$.
By computing the fusion rules, the multiplicities of $L_{C_5}(3,0)$, $L_{C_5}(3,2\Lambda_4)$, $L_{C_5}(3,3\Lambda_2)$ and  $L_{C_5}(3,\Lambda_2+2\Lambda_5)$ are zero in the both fusion products 
$$ L_{C_5}(3,2\Lambda_4)\boxtimes L_{C_5}(3,\Lambda_2+2\Lambda_5)\quad \text{and}\quad L_{C_5}(3,3\Lambda_2)\boxtimes L_{C_5}(3,\Lambda_2+2\Lambda_5).$$
Hence $m_1=m_2=0$; otherwise, 
$Y(u,z)v=0$ for all $u\in m_1L_{C_5}(3,2\Lambda_4)$ $\oplus m_2 L_{C_5}(3,3\Lambda_2)$ and $v\in 
L_{C_5}(3,\Lambda_2+2\Lambda_5)$, which contradicts that $\Com_V(L(\g_{23}))$ is simple (see \cite[Proposition 11.9]{DL}). 
Moreover, the multiplicity of $L_{C_5}(3,\Lambda_2+2\Lambda_5)$ is zero in the fusion product
$$L_{C_5}(3,\Lambda_2+2\Lambda_5)\boxtimes L_{C_5}(3,\Lambda_2+2\Lambda_5).$$
Hence the linear map that acts as $1$ on $ L_{C_5}(3,0)$ and $-1$ on $m_3 L_{C_5}(3,\Lambda_2+2\Lambda_5)$ defines an order $2$ automorphism of $\Com_V(L(\g_{23}))$. It implies that $m_3L_{C_5}(3,\Lambda_2+2\Lambda_5)$ is a simple current module for $ L_{C_5}(3,0)$ (see Proposition \ref{P:DJX}), which is a contradiction. 
Hence, $m_3=0$. 
\end{proof}

By using computer, it is easy to verify that $L(V_1)(\cong L_{C_5}(3,0)\otimes L_{G_2}(2,0)\otimes L_{A_1}(1,0))$ has $27$ inequivalent irreducible modules which have integral conformal weights at least $2$; they are given as in Table \ref{tableC53G22A11}.

\begin{table}[bht]
\caption{Irreducible modules for $ L_{C_5}(3,0)\otimes L_{G_2}(2,0)\otimes L_{A_1}(1,0)$ with integral conformal weight at least $2$} \label{tableC53G22A11}
\begin{tabular}{|c|c|c|c|}
\hline
Highest weight &  $(2\Lambda_4,0,0)$&$(\Lambda_2+2\Lambda_5,0,0)$&$(3\Lambda_4,\Lambda_2,0)$\\
& $(3\Lambda_2,0,0)$&$(\Lambda_2+\Lambda_3+\Lambda_5,\Lambda_1,0)$&$(3\Lambda_5,0,\Lambda_1)$\\
&$(2\Lambda_3,\Lambda_1,0)$&$(2\Lambda_5,2\Lambda_1,0)$&$(2\Lambda_4+\Lambda_5,\Lambda_1,\Lambda_1)$\\
&$(2\Lambda_1+\Lambda_4,\Lambda_1,0)$&$(\Lambda_1+\Lambda_3+\Lambda_4,2\Lambda_1,0)$&$(2\Lambda_3+\Lambda_5,2\Lambda_1,\Lambda_1)$\\
&$(2\Lambda_2,2\Lambda_1,0)$&$(2\Lambda_2+\Lambda_4,\Lambda_2,0)$&$(\Lambda_3+2\Lambda_4,\Lambda_2,\Lambda_1)$\\
&$(\Lambda_1+\Lambda_5,\Lambda_2,0)$&$(3\Lambda_3,0,\Lambda_1)$&\\
&$(2\Lambda_1+\Lambda_2,\Lambda_2,0)$&$(\Lambda_1+2\Lambda_4,\Lambda_1,\Lambda_1)$&\\
&$(2\Lambda_1+\Lambda_5,0,\Lambda_1)$&$(2\Lambda_2+\Lambda_5,\Lambda_1,\Lambda_1)$&\\
&$(\Lambda_2+\Lambda_3,\Lambda_1,\Lambda_1)$&$(\Lambda_1+\Lambda_2+\Lambda_4,2\Lambda_1,\Lambda_1)$&\\
&$(\Lambda_5,2\Lambda_1,\Lambda_1)$&$(\Lambda_4+\Lambda_5,\Lambda_2,\Lambda_1)$&\\
&$(3\Lambda_1,\Lambda_2,\Lambda_1)$&$(\Lambda_1+2\Lambda_3,\Lambda_2,\Lambda_1)$&\\
 \hline
Conformal weight & $2$&$3$&$4$ \\ \hline
\end{tabular}
\end{table}

Decompose $V$ as an $ L(V_1)$-module:
%Now suppose 
\[
V\cong L_{C_5}(3,0)\otimes L_{G_2}(2,0)\otimes L_{A_1}(1,0) \oplus \bigoplus_{ 
(\lambda_1, \lambda_2, \lambda_3)} m_{\lambda_1, \lambda_2, \lambda_3} L_{C_5}(3,\lambda_1)\otimes L_{G_2}(2,\lambda_2)\otimes L_{A_1}(1,\lambda_3),
\]
where the multiplicities $m_{\lambda_1, \lambda_2, \lambda_3}$ are non-negative integers and $(\lambda_1, \lambda_2, \lambda_3)$ are the highest weights listed as in Table \ref{tableC53G22A11}. 

Let $\mathfrak{H}$ be a Cartan subalgebra of $V_1$. For any finite dimensional $V_1$-submodule $M$ of $V$, $M$ can be decomposed as  a sum of weight spaces for $\mathfrak{H}$. Let $\Pi(M)\subset \mathfrak{H}^*$ be the set of weights of $M$. For any positive integer $j$ and $z\in \mathfrak{H}$, define 
\[
S^j_M(z) = \sum_{\mu\in \Pi(M)} m_\mu \mu(z)^j, 
\]
where $m_\mu$ is the multiplicity of $\mu$ in $M$. In \cite{EMS}, the following formula 
is proved:
\[
S^2_{V_2}(z) = 32808\langle z, z\rangle - 2\dim(V_1) \langle z, z\rangle. 
\] 
Let $z=(0,0,\alpha_1)$ and $(0, \alpha_2,0)$, where $\alpha_1$ is a root of $A_1$ and $\alpha_2$ is a long root of $G_2$. Then we have two linear equations on $m_{\lambda_1, \lambda_2, \lambda_3}\in\Z_{\ge0}$ for $(\lambda_1,\lambda_2,\lambda_3)$ in column $2$ of Table \ref{tableC53G22A11}. 
In addition, we have $\dim V_2= 196884$.
By using computer, the system of linear equations can be solved as follows:

\begin{lemma}\label{L:lambda}
Let $(\lambda_1,\lambda_2,\lambda_3)$ be a weight listed in column 2 of Table \ref{tableC53G22A11}. 
Then the multiplicity $m_{\lambda_1,\lambda_2,\lambda_3}$ is given as follows: $$m_{\lambda_1,\lambda_2,\lambda_3}=\begin{cases}0&\text{if}\ (\lambda_1,\lambda_2,\lambda_3)=(2\Lambda_1+\Lambda_2,\Lambda_2,0),\\ 1&\text{otherwise}.
\end{cases}$$
\end{lemma}

Combining Lemmas \ref{m3=0} and \ref{L:lambda}, we obtain the following:

\begin{lemma}\label{hatC53}
The VOA $\Com_V(L(\g_{23}))$ has the decomposition 
\[
L_{C_5}(3,0) \oplus L_{C_5}(3,2\Lambda_4)\oplus L_{C_5}(3,3\Lambda_2)
\]
as a module for $L(\g_1)(\cong L_{C_5}(3,0))$. 
\end{lemma}

Next, let us consider the commutant $\Com_V(L(\g_{2}))$.
Set $\g_{13}=\g_1\oplus \g_3$.
Then $L(\g_{13})$ is a full subVOA of $\Com_V(L(\g_{2}))$.
Note that $L(\g_{13})=L(C_{5,3}A_{1,1})\cong L_{C_5}(3,0)\otimes L_{A_1}(1,0)$.

\begin{lemma}\label{L:ComG22}
\begin{enumerate}[{\rm (1)}]
\item The VOA $\Com_V(L(\g_2))$ contains $\tilde{L}(C_{5,3}A_{1,1})$ as a full subVOA, where $\tilde{L}(C_{5,3}A_{1,1})$ is the $\Z_2$-graded simple current extension of $L(C_{5,3}A_{1,1})$ discussed in Section \ref{S:C5A1}.
%\item 
In addition, as $\tilde{L}(C_{5,3}A_{1,1})$-modules, 
\begin{equation}
\Com_V(L(\g_2))\cong \tilde{L}(C_{5,3}A_{1,1})\oplus\tilde{L}(C_{5,3}A_{1,1})[ 2\Lambda_4,0]\oplus \tilde{L}(C_{5,3}A_{1,1})[3\Lambda_2,0]\label{Eq:TC5A1},
\end{equation} 
where the irreducible $\tilde{L}(C_{5,3}A_{1,1})$-modules $\tilde{L}(C_{5,3}A_{1,1})[ \lambda,0]$ are described in Lemma \ref{L:SCEC5A1}.
\item For $g\in I(V)$, the restriction of $g$ to $\Com_V(L(\g_2))$ belongs to $\langle \sigma_{(0,\Lambda_1)}\rangle$, where $\Lambda_1$ is the fundamental weight of $\g_3$ (of type $A_1$).
\end{enumerate}
\end{lemma}

\begin{proof}
Set $U=\Com_V(L(\g_2))$.
By Lemma \ref{m3=0}, $m_{(\Lambda_2+2\Lambda_5,0,0)}=0$. 
By Lemma \ref{L:lambda}, we also have the multiplicities $m_{2\Lambda_4,0,0}=m_{3\Lambda_2,0,0}=m_{(2\Lambda_1+\Lambda_5,0,\Lambda_1)}=1$.
By Table \ref{tableC53G22A11}, as a module for $L(\g_{13})(\cong L_{C_5}(3,0)\otimes L_{A_1}(1,0))$,
\[
\begin{split}
U\cong [0,0]\oplus[2\Lambda_4,0]\oplus[3\Lambda_2,0]\oplus m_{3\Lambda_5,0,\Lambda_1}[3\Lambda_5,\Lambda_1]\oplus[2\Lambda_1+\Lambda_5,\Lambda_1]\oplus m_{3\Lambda_3,0,\Lambda_1}[3\Lambda_3,\Lambda_1]
\end{split}
\]
where $[\lambda_1,\lambda_2]$ means the irreducible $L(\g_{13})$-module $L_{C_5}(3,\lambda_1)\otimes L_{A_1}(1,\lambda_2)$.

Notice that $\sigma=\sigma_{(0,0, \Lambda_1)}= \exp(2\pi \sqrt{-1} (0,0, \Lambda_1)_{(0)})$ associated with the fundamental weight $(0,0,\Lambda_1)$ of $A_1$ defines an automorphism of order $2$ on $U$. The fixed point set $U^+$ with respect to $\sigma$ is given by 
\begin{equation}
U^+=[0,0]\oplus[2\Lambda_4,0]\oplus[3\Lambda_2,0]\label{Eq:C5}
\end{equation}
and the $-1$-eigenspace $U^-$ is given by 
\begin{equation}
U^-=m_{3\Lambda_5,0,\Lambda_1}[3\Lambda_5,\Lambda_1]\oplus[2\Lambda_1+\Lambda_5,\Lambda_1]\oplus m_{3\Lambda_3,0,\Lambda_1}[3\Lambda_3,\Lambda_1]\label{Eq:C52}
\end{equation}
%Then by Proposition \ref{P:DJX}, $U^-$ is a simple current for $U^+$. Hence, by Proposition \ref{P:Mirror} (3), we have 
By Lemma \ref{L:QD}, we have
\[
{\qdim}_{L(\g_{13})} U^+ = {\qdim}_{L(\g_{13})} U^-.
\]
Recall the following fusion products for the simple current module $[3\Lambda_5,\Lambda_1]$:
$$[2\Lambda_4,0]\boxtimes[3\Lambda_5,\Lambda_1]=[2\Lambda_1+\Lambda_5,\Lambda_1],\quad [3\Lambda_2,0]\boxtimes[3\Lambda_5,\Lambda_1]=[3\Lambda_3,\Lambda_1].$$
Hence, by Proposition \ref{P:qdim} (1) and (2), we have
\begin{equation}
\begin{split}  
&{\qdim}_{L(\g_{13})}[0,0]={\qdim}_{L(\g_{13})}[3\Lambda_5,\Lambda_1]=1,\\
&{\qdim}_{L(\g_{13})}[2\Lambda_4,0]={\qdim}_{L(\g_{13})}[2\Lambda_1+\Lambda_5,\Lambda_1],\\
&{\qdim}_{L(\g_{13})}[3\Lambda_2,0]={\qdim}_{L(\g_{13})}[3\Lambda_3,\Lambda_1].
\end{split}\label{Eq:C53}
\end{equation}
Thus by \eqref{Eq:C5}, \eqref{Eq:C52} and \eqref{Eq:C53}, we obtain $$(m_{3\Lambda_5,0,\Lambda_1}-1)+(m_{3\Lambda_3,0,\Lambda_1}-1){\qdim}_{L(\g_{13})}[3\Lambda_2,0]=0.$$
%Since $[3\Lambda_2,0]$ is not simple current for $L(\g_{13})$, we have .
Since $[3\Lambda_5,\Lambda_1]$ is a simple current module but $[3\Lambda_2,0]$ is not for $L(\g_{13})$, we have $m_{3\Lambda_5,0,\Lambda_1}\in\{0,1\}$ and $\qdim_{L(\g_{13})}[3\Lambda_2,0]>1$.
% since $[3\Lambda_2,0]$ is not simple current for $L(\g_{13})$. 
Thus we obtain $m_{3\Lambda_5,0,\Lambda_1}=m_{3\Lambda_3,0,\Lambda_1}=1$.
Since $[3\Lambda_5,\Lambda_1]$ is self-contragredient and $U$ is simple, the invariant form on $[3\Lambda_5,\Lambda_1]$ is non-degenerate.
Hence the $\Z_2$-graded extension $[0,0]\oplus [3\Lambda_5,\Lambda_1]$ is simple, and it is isomorphic to $\tilde{L}(C_{5,3}A_{1,1})$.
By Lemma \ref{L:SCEC5A1} and the multiplicities above, we obtain (1).

By using ``Kac", we obtain $$32<{\qdim}_{L(\g_{13})}[2\Lambda_4,0]<33,\quad 16<{\qdim}_{L(\g_{13})}[3\Lambda_2,0]<17.$$
By Lemma \ref{L:SCEC5A1} and Proposition \ref{P:Mirror} (3), we obtain the following quantum dimensions for irreducible $\tilde{L}(C_{5,3}A_{1,1})$-modules:
\begin{align}\label{qdimAC}
\notag&{\qdim}_{\tilde{L}(C_{5,3}A_{1,1})}\tilde{L}(C_{5,3}A_{1,1})={\qdim}_{L(\g_{13})}[0,0]=1,\\
&{\qdim}_{\tilde{L}(C_{5,3}A_{1,1})}\tilde{L}(C_{5,3}A_{1,1})[2\Lambda_4,0]={\qdim}_{L(\g_{13})}[2\Lambda_4,0]> 32,\\
\notag&{\qdim}_{\tilde{L}(C_{5,3}A_{1,1})}\tilde{L}(C_{5,3}A_{1,1})[3\Lambda_2,0]={\qdim}_{L(\g_{13})}[3\Lambda_2,0]<17.
\end{align} 
Let $g\in I(V)$. Then $g$ stabilizes $L(\g_2)$ and so does $U=\Com_V(L(\g_2))$.
Let $U^0$ be the subVOA of $U$ isomorphic to $\tilde{L}(C_{5,3}A_{1,1})$.
Since $U^0$ is a simple current extension of $L(C_{5,3}A_{1,1})$ graded by $D\cong\Z_2$, we have $g_{|U^0}\in D^*=\langle \sigma_{(0,\Lambda_1)}\rangle$.
By replacing $g$ by $g\sigma_{(0,\Lambda_1)}^{-1}$ if necessary, 
%By the argument above, $g^{-1}\sigma_{(0,\Lambda_1)}=id$ on $U$, that is, $g=\sigma_{(0,\Lambda_1)}$ on $U$.
%First, 
we may assume that $g_{|U^0}$ is the identity.
Then by \eqref{Eq:TC5A1}, $g$ acts on each irreducible $U^0$-submodule of $U$ as a scalar.  
Since $\tilde{L}(C_{5,3}A_{1,1})[2\Lambda_4,0]$ and $\tilde{L}(C_{5,3}A_{1,1})[3\Lambda_2,0]$ are self-contragredient, the scalars are $\pm1$.
Suppose, for a contradiction, that the restriction ${g}_{|U}$ of $g$ to $U$ is not the identity.
Then by Proposition \ref{P:DJX}, the fixed point subspace and the $-1$-eigenspace of $g_{|U}$ has the same quantum dimension over $U^0$, which is impossible by \eqref{qdimAC}.
Hence ${g}_{|U}$ is the identity and we obtain (2).
\end{proof}
%\color{black}

Recall that the affine VOA $L(\g_2)(\cong L_{G_2}(2,0))$ has exactly $4$ inequivalent irreducible modules, namely,
\[
L_{G_2}(2,0),\ L_{G_2}(2,\Lambda_1),\ L_{G_2}(2,2\Lambda_1),\text{ and }  L_{G_2}(2,\Lambda_2).
\]
Let $a= {\qdim}_{L_{G_2}(2,0)} L_{G_2}(2,\Lambda_1)$, $b= {\qdim}_{L_{G_2}(2,0)} L_{G_2}(2,2\Lambda_1)$ and $c= {\qdim}_{L_{G_2}(2,0)} L_{G_2}(2,\Lambda_2)$. Then by Proposition \ref{P:qdim} (2) and their fusion rules, we have 
\begin{equation}\label{abc}
\begin{split}
a^2&= 1+a+b+c,\\  b^2&=1+a+b,\\ c^2&=1+b. 
\end{split}
\end{equation}
Note that $a> b>c >1$ since $a,b,c\in\R_{\ge1}$.

Let $M(\Lambda)= {\rm Hom}_{L_{G_2}(2,0)} (V, L_{G_2}(2,\Lambda))$ for $\Lambda= 0, \Lambda_1, 2\Lambda_1, \Lambda_2$. Then $M(0)=\Com_V(L(\g_2))$ and $M(\Lambda), \Lambda = 0, \Lambda_1, 2\Lambda_1, \Lambda_2,$ are irreducible $M(0)$-modules.
As a module for $M(0)\otimes L_{G_2}(2,0)$, the holomorphic VOA $V$ decomposes as 
\[
\begin{split}
V= & \bigoplus_{\Lambda\in\{0,\Lambda_1,2\Lambda_1,\Lambda_2\}}M(\Lambda)\otimes L_{G_2}(2,\Lambda).
\end{split}
\]

By Proposition \ref{P:Mirror} (3), for $\Lambda\in\{0,\Lambda_1,2\Lambda_1,\Lambda_2\}$, we have
\begin{equation}\label{a2b2c2}
{\qdim}_{M(0)\otimes L_{G_2}(2,0)} (M(\Lambda) \otimes L_{G_2}(2,\Lambda))= ({\qdim}_{L_{G_2}(2,0)} L_{G_2}(2,\Lambda))^2.
\end{equation}

\begin{theorem}\label{IC5}
$I(V)=\langle \sigma_{(0,0,\Lambda_1)}\rangle \cong \Z_2$. 
\end{theorem}

\begin{proof} Clearly $\langle \sigma_{(0,0,\Lambda_1)}\rangle\subset I(V)$.
Let $g\in I(V)$. By Lemma \ref{L:ComG22} (2), $g|_{M(0)} = \sigma_{(0,0,\Lambda_1)}^i|_{M(0)}$ for some $i=0,1$.
Hence we may assume that $g=id$ on $M(0)$ by replacing $g$ by $g\sigma_{(0,0,\Lambda_1)}$ if necessary.
Then $g=id$ on $M(0)\otimes L_{G_2}(2,0)$, and $g$ acts on each irreducible $M(0)\otimes L_{G_2}(2,0)$-submodule of $V$ as a scalar. 
Since all the irreducible $M(0)\otimes L_{G_2}(2,0)$-submodules of $V$ are self-contragredient, the scalar must be $\pm 1$. 
%Hence the $(-1)$-eigenspace $V^-$ of $g$ in $V$ is a simple current module for $V^+=\{v\in V\mid g(v)=v\}$.
If $g\neq id$ on $V$, then by Lemma \ref{L:QD}, the fixed point subspace and the $-1$-eigenspace of $g$ have the same quantum dimension over $M(0)\otimes L_{G_2}(2,0)$.
However, by \eqref{abc} and \eqref{a2b2c2}, one can easily verify that it has no solutions. Hence, $g=id $ on $V$ if $g|_{M(0)}=id_{M(0)}$.
Thus we have $g\in\langle \sigma_{(0,0,\Lambda_1)}\rangle$ as desired. 
\end{proof}

\begin{lemma}\label{L:conjC53}  The conjugacy class of an order $2$ automorphism $g$ of $V$ is unique if $V^g_1$ is a Lie algebra of type $A_{4}A_{1}^2U(1)^2$.
\end{lemma}
\begin{proof} 
Let $\bar{g}$ be the restriction of $g$ to $V_1$.
Since the Lie ranks of $V_1$ and $V^{\bar{g}}_1$ are the same, $\bar{g}$ is inner and it preserves every simple ideal of $V_1$.
By the type of $V_1^g$ (cf.\ \cite[Section 8]{Kac}), the $\bar{g}$-fixed points of the simple ideals of type $C_{5}$, $G_2$ and $A_1$ have the type $A_4U(1)$, $A_1^2$, $U(1)$, respectively.
%Since a simple Lie algebra of type $C_{5}$ does not contain a semisimple Lie algebra of type $A_{4}A_1$, the $\bar{g}$-fixed points of the simple ideal of type $C_5$ has the type $A_{4,6}U(1)$.
%By the level, the $\bar{g}$-fixed points of $G_{2,2}$ and $A_{1,1}$ have the type $A_{1,6}A_{1,2}$ and $U(1)$, respectively.
By Lemma \ref{L:conjinv}, $\bar{g}$ is unique up to conjugation by an inner automorphism of $V_1$; we may assume that $\bar{g}=\sigma_{(1/2)(\Lambda_5,\Lambda_1,\Lambda_1)}$ at the beginning of the proof.
Since $I(V)=\langle \sigma_{(0,0,\Lambda_1)}\rangle$ by Theorem \ref{IC5}, we have $g=\sigma_{(1/2)(\Lambda_5,\Lambda_1,\Lambda_1)}$ or $g=\sigma_{(1/2)(\Lambda_5,\Lambda_1,-\Lambda_1)}$.
Clearly, on a root system of type $A_1$, the weights $\Lambda_1$ and $-\Lambda_1$ are conjugate by an element of the Weyl group.
Hence  $\sigma_{(1/2)(\Lambda_5,\Lambda_1,\Lambda_1)}$ and $\sigma_{(1/2)(\Lambda_5,\Lambda_1,-\Lambda_1)}$ are conjugate by an inner automorphism of $V$.
Thus we obtain the desired result.    
\end{proof}

\subsection{Weight one Lie algebra of type $A_{7,4}A_{1,1}^3$}

Let $V$ be a strongly regular holomorphic VOA of central charge $24$ such that 
$V_1$ has the type $A_{7,4}A_{1,1}^3$.
Recall that such a holomorphic VOA has been constructed in \cite{Lam} and that the uniqueness has also been established in \cite{LS6}. 
Let $V_1=\bigoplus_{i=1}^4\g_i$ be the sum of simple ideals, where  $\g_1$ and $\g_i$ ($2\le i\le 4$) has the type $A_{7,4}$ and $A_{1,1}$, respectively.
Set $\g_{234}=\bigoplus_{i=2}^4\g_i$.

Now, we recall an alternative construction from \cite[Section 3.1.2]{Xu} by using mirror extension;
%The main idea is due to \cite[Section 3.1.2]{Xu};
a holomorphic conformal net associated with $SU(8)_4\times SU(2)_1\times SU(2)_1\times SU(2)_1$ has been constructed. 
By translating the conformal net setting in \cite{Xu} to the VOA setting (cf.\ \cite{HKL,DJX,Lin}), we can apply the results in \cite{Xu} to obtain a holomorphic VOA of central charge $24$ with the weight one Lie algebra of the type $A_{7,4}A_{1,1}^3$.
Remark that the uniqueness result in \cite{LS6} shows that $V$ and the holomorphic VOA studied by \cite{Xu} are the same, up to isomorphism.
Hence we know the $L_{A_7}(4,0)$-module structure of $\Com_V(L(\g_{234}))$ from \cite{Xu};
$\Com_V(L(\g_{234}))$ contains $\tilde{L}_{A_7}(4,0)$, a $\Z_4$-graded simple current extension of $L(\g_1)(\cong L_{A_7}(4,0))$ (see Section \ref{S:A7}), as a full subVOA and
decomposes as a module of $\tilde{L}_{A_7}(4,0)$ as follows:
\begin{equation}
\Com_V(L(\g_{234}))\cong\tilde{L}_{A_7}(4,0)\oplus\tilde{L}_{A_7}(4,\Lambda_4+\Lambda_5+\Lambda_7),\label{Eq:A74dec}
\end{equation}
where the $\tilde{L}_{A_7}(4,0)$-module $\tilde{L}_{A_7}(4,\Lambda_4+\Lambda_5+\Lambda_7)$ is described in Lemma \ref{L:A7mod}.

Recall that $L_{A_1}(1,0)$ has exactly $2$ inequivalent irreducible modules $L_{A_1}(1,0)$ and $L_{A_1}(1,\Lambda_1)$ and both are simple current modules.
Since the conformal weight of $L_{A_1}(1,\Lambda_1)$ is $1/4$, $L(\g_{234})(\cong L_{A_1}(1,0)^{\otimes3})$ is the only irreducible $L(\g_{234})$-module with integral conformal weight.
Hence $\Com_V(\Com_V(L(\g_{234})))=L(\g_{234})$.
For $\lambda_i\in\{0,\Lambda_1\}$, $i\in\{1,2,3\}$, set $$M(\lambda_1,\lambda_2,\lambda_3)={\rm Hom}_{L(\g_{234})}(V,\bigotimes_{i=1}^3 L_{A_{1}}(1,\lambda_i)).$$
Then $M(0,0,0)=\Com_V(L(\g_{234}))$ and $M(\lambda_1,\lambda_2,\lambda_3)$ are irreducible $M(0,0,0)$-modules.
As a module of $M(0,0,0)\otimes L(\g_{234})$, the holomorphic VOA $V$ decomposes as 
\begin{equation}\label{Eq:decA74}
V\cong\bigoplus_{\lambda_1,\lambda_2,\lambda_3\in\{0,\Lambda_1\}}M(\lambda_1,\lambda_2,\lambda_3)\otimes \left(\bigotimes_{i=1}^3L_{A_1}(1,\lambda_i)\right).
\end{equation}
By Propositions \ref{P:DJX} (1) and \ref{P:Mirror} (3), $M(\lambda_1,\lambda_2,\lambda_3)$ are simple current modules.
Hence $V$ is a $\Z_2^3$-graded simple current extension of $M(0,0,0)\otimes L(\g_{234})$.
The explicit description of the $L_{A_7}(4,0)$-module structure of $M(\lambda_1,\lambda_2,\lambda_3)$ for a specified order of the tensor product $L_{A_1}(1,0)^{\otimes3}$ can be found in \cite[Section 3.1.3]{Xu} as follows:
% and the proof in VOA setting is similar.
 
\begin{lemma}\label{L:modCA11}
The VOA $M(0,0,0)$ has exactly $8$ inequivalent irreducible modules, and as $L_{A_7}(4,0)$-modules, they decompose as:
\[
\tiny
\begin{split}
%\Com_V(L(A_{1,1}^3))
M(0,0,0)\cong& L_{A_7}(4,0)\oplus L_{A_7}(4,4\Lambda_2)\oplus L_{A_7}(4,4\Lambda_4)\oplus L_{A_7}(4,4\Lambda_6)\oplus 
L_{A_7}(4,\Lambda_4+\Lambda_5+\Lambda_7)\\
&\oplus L_{A_7}(4,\Lambda_2+\Lambda_3+\Lambda_5+\Lambda_6)
\oplus L_{A_7}(4,\Lambda_1+\Lambda_3+\Lambda_4)\oplus L_{A_7}(4,\Lambda_1+\Lambda_2+\Lambda_6+\Lambda_7),\\
M(\Lambda_1,0,0) \cong& L_{A_7}(4,4\Lambda_1)\oplus L_{A_7}(4,4\Lambda_3)\oplus 
L_{A_7}(4,4\Lambda_5) \oplus  L_{A_7}(4,4\Lambda_7)\oplus L_{A_7}(4,\Lambda_3+\Lambda_4+\Lambda_6+\Lambda_7)\\
&\oplus L_{A_7}(4,\Lambda_1+\Lambda_2+\Lambda_4+\Lambda_5)
\oplus L_{A_7}(4,\Lambda_2+\Lambda_3+\Lambda_7)\oplus L_{A_7}(4,\Lambda_1+\Lambda_5+\Lambda_6),\\
M(0,\Lambda_1,\Lambda_1) \cong&  L_{A_7}(4,2\Lambda_5 + 2\Lambda_7)\oplus  L_{A_7}(4,2\Lambda_3 + 2\Lambda_5)\oplus L_{A_7}(4,2\Lambda_1 + 2\Lambda_3)\oplus 
L_{A_7}(4,2\Lambda_1 + 2\Lambda_7) \\
&\oplus L_{A_7}(4,\Lambda_3+\Lambda_5) \oplus L_{A_7}(4,\Lambda_1+\Lambda_3+2\Lambda_6)
\oplus L_{A_7}(4,\Lambda_1+2\Lambda_4+\Lambda_7)\oplus L_{A_7}(4,2\Lambda_2+ \Lambda_5+\Lambda_7),\\
M(\Lambda_1,\Lambda_1,\Lambda_1) \cong&  L_{A_7}(4,2\Lambda_4 + 2\Lambda_6)\oplus  L_{A_7}(4,2\Lambda_3 + 2\Lambda_4)\oplus L_{A_7}(4, 2\Lambda_2)\oplus 
L_{A_7}(4,2\Lambda_6)\\
& \oplus L_{A_7}(4,\Lambda_2+\Lambda_4+ 2\Lambda_7) \oplus L_{A_7}(4,\Lambda_2+2\Lambda_5)
\oplus L_{A_7}(4,2\Lambda_3+\Lambda_6)\oplus L_{A_7}(4,2\Lambda_2+ \Lambda_5+\Lambda_7),\\
M(0,\Lambda_1,0)\cong
% \cong
%& L_{A_7}(4,\Lambda_3+\Lambda_6+\Lambda_7)\oplus L_{A_7}(4,\Lambda_1+\Lambda_4+ \Lambda_5+\Lambda_6) \\
%& \oplus L_{A_7}(4,\Lambda_2+\Lambda_3+ \Lambda_4+\Lambda_7)\oplus L_{A_7}(4,\Lambda_1+ \Lambda_2+\Lambda_5), \\
M(0,0,\Lambda_1) \cong
& L_{A_7}(4,\Lambda_3+\Lambda_6+\Lambda_7)\oplus L_{A_7}(4,\Lambda_1+\Lambda_4+ \Lambda_5+\Lambda_6)\\
& \oplus L_{A_7}(4,\Lambda_2+\Lambda_3+ \Lambda_4+\Lambda_7)\oplus L_{A_7}(4,\Lambda_1+ \Lambda_2+\Lambda_5), \\
M(\Lambda_1,\Lambda_1,0) \cong
%& L_{A_7}(4,\Lambda_3+\Lambda_4+\Lambda_5)\oplus L_{A_7}(4,\Lambda_1+\Lambda_2+ \Lambda_3+\Lambda_6)\\
%& 
%\oplus L_{A_7}(4,\Lambda_1+\Lambda_4+\Lambda_7)\oplus L_{A_7}(4,\Lambda_2+ \Lambda_5+\Lambda_6+\Lambda_7), \\
M(\Lambda_1,0,\Lambda_1) \cong 
& L_{A_7}(4,\Lambda_3+\Lambda_4+\Lambda_5)\oplus L_{A_7}(4,\Lambda_1+\Lambda_2+ \Lambda_3+\Lambda_6)\\
& \oplus L_{A_7}(4,\Lambda_1+\Lambda_4+\Lambda_7)\oplus L_{A_7}(4,\Lambda_2+ \Lambda_5+\Lambda_6+\Lambda_7). 
\end{split}
\] 
\end{lemma}

\begin{lemma}\label{iAT}
Any element of $I(V)$ acts on $M(0,0,0)$ as the identity. 
\end{lemma}
\begin{proof} Let $g\in I(V)$. Then $g$ stabilizes $L(\g_{234})$ and so  does $\Com_V(L(\g_{234}))=M(0,0,0)$.
Recall from \eqref{Eq:A74dec} that $\Com_V(L(\g_{234}))$ contains a full subVOA $\tilde{L}_{A_7}(4,0)$.
Since $g$ acts on each irreducible $L_{A_7}(4,0)$-submodule of $\Com_V(L(\g_{234}))$ as a scalar, it stabilizes $\tilde{L}_{A_7}(4,0)$. 
Since $\tilde{L}_{A_7}(4,0)$ is a simple current extension of $L_{A_7}(4,0)$ graded by $D\cong \Z_4$ (see Section \ref{S:A7}), we have $g= \xi^j\in D^*$ on $\tilde{L}_{A_7}(4,0)$ for some $j$, where $\xi\in \Aut( \tilde{L}_{A_7}(4,0))$ is defined by 
\[
\xi=\begin{cases}
1 & \text{ on } L_{A_7}(4,0),\\
(\sqrt{-1})^i & \text{ on } L_{A_7}(4,4\Lambda_{2i}),\ (i=1,2,3).\\
%\sqrt{-1} & \text{ on } L_{A_7}(4,4\Lambda_2),\\
%-1 & \text{ on } L_{A_7}(4,4\Lambda_4),\\
%-\sqrt{-1} & \text{ on } L_{A_7}(4,4\Lambda_6).
\end{cases}
\]  

Suppose, for a contradiction, that $g=\xi^j \neq id$ on $\tilde{L}_{A_7}(4,0)$. Set $f= g$ if $j=2$ and $f=g^2$ if $j=1,3$. Then $f$ acts trivially on $L_{A_7}(4,4\Lambda_4)$ and acts as $-1$ on both $L_{A_7}(4,4\Lambda_2)$ and $L_{A_7}(4,4\Lambda_6)$. 
By \eqref{Eq:A74dec}, the $\tilde{L}_{A_7}(4,0)$-module structure of $\Com_V(L(\g_{234}))$ is 
$$
\Com_V(L(\g_{234}))\cong\tilde{L}_{A_7}(4,0)\oplus\tilde{L}_{A_7}(4,\Lambda_4+\Lambda_5+\Lambda_7),
$$
and by Lemma \ref{L:A7mod}, the $L_{A_7}(4,0)$-module structure of $\tilde{L}_{A_7}(4,\Lambda_4+\Lambda_5+\Lambda_7)$ is
\begin{align*}
\tilde{L}_{A_7}(4,\Lambda_4+\Lambda_5+\Lambda_7)\cong&L_{A_7}(4,\Lambda_4+\Lambda_5+\Lambda_7)\oplus L_{A_7}(4,\Lambda_2+\Lambda_3+\Lambda_5+\Lambda_6)\\
&\oplus L_{A_7}(4,\Lambda_1+\Lambda_3+\Lambda_4)\oplus L_{A_7}(4,\Lambda_1+\Lambda_2+\Lambda_6+\Lambda_7).
\end{align*}
Note that $$\tilde{L}_{A_7}(4,\Lambda_4+\Lambda_5+\Lambda_7)\cong\tilde{L}_{A_7}(4,0)\boxtimes_{L_{A_7}(4,0)}L_{A_7}(4,\Lambda_4+\Lambda_5+\Lambda_7),$$
and 
\begin{equation}\label{nscA74}
\begin{split}
L_{A_7}(4,\Lambda_2+\Lambda_3+\Lambda_5+\Lambda_6)& \cong L_{A_7}(4,4\Lambda_2)\boxtimes_{L_{A_7}(4,0)}L_{A_7}(4,\Lambda_4+\Lambda_5+\Lambda_7)\\
L_{A_7}(4,\Lambda_1+\Lambda_3+\Lambda_4)& \cong L_{A_7}(4,4\Lambda_4)\boxtimes_{L_{A_7}(4,0)}L_{A_7}(4,\Lambda_4+\Lambda_5+\Lambda_7),\\
L_{A_7}(4,\Lambda_1+\Lambda_2+\Lambda_6+\Lambda_7)
& \cong L_{A_7}(4,4\Lambda_6)\boxtimes_{L_{A_7}(4,0)}L_{A_7}(4,\Lambda_4+\Lambda_5+\Lambda_7). 
\end{split}
\end{equation}
By the fusion rules, one can verify that $L_{A_7}(4,\Lambda_2+\Lambda_3+\Lambda_5+\Lambda_6)$ is self-contragredient as an $L_{A_7}(4,0)$-module and hence $f$ acts on it as $\varepsilon\in\{\pm1\}$.
By \eqref{nscA74}, $f$ acts on $L_{A_7}(4,\Lambda_1+\Lambda_3+\Lambda_4)$, $L_{A_7}(4,\Lambda_1+\Lambda_2+\Lambda_6+\Lambda_7)$ and $L_{A_7}(4,\Lambda_4+\Lambda_5+\Lambda_7)$ as $-\varepsilon$, $\varepsilon$ and $-\varepsilon$, respectively.

Assume $\varepsilon=1$.
Then 
\[
\begin{split}
 (\Com_V(L(\g_{234})))^f = &L_{A_7}(4,0)\oplus L_{A_7}(4,4\Lambda_4)\oplus  L_{A_7}(4,\Lambda_2+\Lambda_3+\Lambda_5+\Lambda_6)\\ &
\oplus L_{A_7}(4,\Lambda_1+\Lambda_2+\Lambda_6+\Lambda_7)
\end{split}
\] 
forms a proper subVOA of $\Com_V(L(\g_{234}))$. 
By Proposition \ref{P:Mirror} (2) and Section 5.3 (cf.\ \cite[Section 3.1.3]{Xu}), along with the uniqueness of $V$, we see that $\Com_V(L(\g_{234}))$ corresponds to an extension of $L_{A_3}(8,0)$ isomorphic to $L_{D_{10}}(1,0)$ via the mirror extension of the embedding $L_{A_7}(4,0)\otimes L_{A_3}(8,0)\subset L_{A_{31}}(1,0)$.
In fact, by \eqref{Eq:D10} and Table \ref{tableMir}, the subVOA $(\Com_V(L(\g_{234})))^f$ corresponds the proper subVOA $X$ of $L_{D_{10}}(1,0)$. As an $L_{A_3}(8,0)$-module, $X$ decomposes as
\[
%\tiny
\begin{split}
&X\cong L_{A_3}(8,0)\oplus L_{A_3}(8,8\Lambda_2)
\oplus L_{A_3}(8,\Lambda_1+2\Lambda_2+\Lambda_3)
\oplus L_{A_3}(8,\Lambda_1+4\Lambda_2+\Lambda_3).
\end{split}
\]
By Table \ref{tableMir}, we have $\dim X_1=190$, and $X_1=L_{D_{10}}(1,0)_1$.
Then $L_{D_{10}}(1,0)$ is generated by $X$, which contradicts that $X$ is a proper subVOA of $L_{D_{10}}(1,0)$.

Assume $\varepsilon=-1$.
Then 
\[
\begin{split}
 &(\Com_V(L(\g_{234})))^f \\
 = & L_{A_7}(4,0)\oplus L_{A_7}(4,4\Lambda_4)\oplus  L_{A_7}(4,\Lambda_4+\Lambda_5+\Lambda_7) %\\ & 
\oplus L_{A_7}(4,\Lambda_1+\Lambda_3+\Lambda_4)
\end{split}
\] 
forms a proper subVOA of $\Com_V(L(\g_{234}))$.
Let $W$ be the $(-1)$-eigenspace of $f$.
Then by Proposition \ref{P:DJX}, $\Com_V(L(\g_{234}))=(\Com_V(L(\g_{234})))^f\oplus W$ is a $\Z_2$-graded simple current extension.
%In particular, $\qdim_{\Com_V(L(\g_{234}))^f} W=1$, and
By Lemma \ref{L:QD}, we have $\qdim_{L_{A_7}(4,0)}\Com_V(L(\g_{234}))^f=\qdim_{L_{A_7}(4,0)} W$.
By the similar argument as above,  
\[
\begin{split}
X\cong L_{A_3}(8,0)\oplus L_{A_3}(8,8\Lambda_2)
\oplus L_{A_3}(8,2\Lambda_1+\Lambda_2+4\Lambda_3)\oplus L_{A_3}(8,4\Lambda_1+\Lambda_2+2\Lambda_3)
\end{split}
\] 
forms a proper subVOA of $L_{D_{10}}(1,0)$ and 
\[
\begin{split}
Y\cong  L_{A_3}(8,8\Lambda_1)\oplus L_{A_3}(8,8\Lambda_3)
\oplus L_{A_3}(8,\Lambda_1+2\Lambda_2+\Lambda_3)
\oplus L_{A_3}(8,\Lambda_1+4\Lambda_2+\Lambda_3)
\end{split}
\]
is a module of $X$. Moreover, $X\oplus Y\cong L_{D_{10}}(1,0)$. 
By Proposition \ref{P:Mirror} (3) and Table \ref{tableMir}, we have $\qdim_{L_{A_3}(8,0)}X=\qdim_{L_{A_7}(4,0)}(\Com_V(L(\g_{234})))^f$ and $\qdim_{L_{A_3}(8,0)}Y=\qdim_{L_{A_7}(4,0)} W$.
Hence $\qdim_{L_{A_3}(8,0)}X=\qdim_{L_{A_3}(8,0)}Y$, and $Y$ is a simple current module for $X$.
Since $X\oplus Y$ is isomorphic to $L_{D_{10}}(1,0)$ as a VOA, we must have $Y\fusion_X Y =X$; otherwise, the multiplicity of $X$ in $Y\fusion _X Y$ is zero and $Y$ is a proper submodule of $X\oplus Y$ as an $X\oplus Y$-module, which is not possible as $L_{D_{10}}(1,0)$ is simple. 
Hence $X\oplus Y$ is a $\Z_2$-graded simple current extension and $X$ is the fixed points of the order $2$ automorphism associated with the grading.
However, the weight one Lie algebra of $X$ has the type $A_3$, which cannot be realized as a fixed point Lie subalgebra of type $D_{10}$ for any order $2$ automorphism (cf. \cite[Chapter X, Theorem~6.1, TABLE II and pp.513--515]{H}), which is a contradiction.  

Thus $g$ is the identity on $M(0,0,0)=\Com_V(L(\g_{234}))$.
\end{proof}

\begin{theorem}\label{A74}
We have $I(V)=\langle \sigma_{(0,\Lambda_1,0,0)}, \sigma_{(0,0,\Lambda_1,0)}, \sigma_{(0,0,0, \Lambda_1)}\rangle\cong \Z_2^3$. 
\end{theorem}

\begin{proof}
Let $g\in I(V)$. By Lemma \ref{iAT},  $g=id$ on $\Com_V(L(\g_{234}))$. 
Since $V$ is a simple current extension of $\Com_V(L(\g_{234}))\otimes L_{A_1}(1,0)^{\otimes 3}$ graded by $E\cong\Z_2^3$, we have  $I(V)=E^*=\langle \sigma_{(0,\Lambda_1,0,0)}, \sigma_{(0,0,\Lambda_1,0)}, \sigma_{(0,0,0, \Lambda_1)}\rangle\cong \Z_2^3$ as a subgroup, where $\Lambda_1$ is the fundamental weight of $A_{1,1}$.
\end{proof}

Let us discuss some automorphisms of $V$ and their conjugacy classes.
We now specify the order of the (isomorphic) simple ideals $\g_2$, $\g_3$ and $\g_4$ so that $V$ decomposes as in Lemma \ref{L:modCA11}.

\begin{lemma}\label{L:flip} The following two inner automorphisms are conjugate in $\Aut (V)$:
$$\sigma_{((1/2)(\Lambda_1+\Lambda_7),\Lambda_1,\Lambda_1,0)},\quad\sigma_{((1/2)(\Lambda_1+\Lambda_7),\Lambda_1,0,\Lambda_1)}.$$

\end{lemma}
\begin{proof}In \cite[Section 6.1]{Lam}, along with the uniqueness result for the case $A_{7,4}A_{1,1}^3$ (\cite{LS6}), $V$ was constructed as a framed VOA.
In addition, $V$ has an order $2$ automorphism $g$ such that $(\widetilde{V}_g)_1$ has the structure $D_{5,8}A_{1,2}$.
By the explicit action of $g$ on $V_1$, we know that $V_1^g$ has the structure $D_{4,8}A_{1,2}U(1)$.
In particular, $g$ flips two simple ideals of type $A_{1,1}$.

%Let $V_1=\bigoplus_{i=1}^4\g_i$, where the types of $\g_1$ and $\g_i$, $(2\le i\le 4)$ are $A_{7,4}$ and $A_{1,1}$, respectively.

Clearly, $g$ preserves $L(\g_{234})$ and so does $\Com_V(L(\g_{234}))(=M(0,0,0))$.
Note that the $g$-conjugation of $M(\Lambda_1,0,0)$ is neither $M(0,\Lambda_1,0)$ nor $M(0,0,\Lambda_1)$ since the number of irreducible $L_{A_7}(4,0)$-submodules are different (see Lemma \ref{L:modCA11}).
By Lemma \ref{L:Vg}, $g$ must preserves $M(\Lambda_1,0,0)\otimes L_{A_1}(1,\Lambda_1)\otimes L_{A_1}(1,0)^{\otimes2}$ in \eqref{Eq:decA74}.
Hence $g(\g_2)=\g_2$, and $g(\g_3)=\g_4$.
Note that $\g_1^g$ has the type $D_4$ and $\g_2^g=U(1)$. 
%Clearly, $g(\g_1)=\g_1$ and $g(\g_2)=\g_2$.
By \cite[Proposition 8.1]{Kac}, up to conjugation by inner automorphisms of $V$, we may assume that $g$ preserves the (fixed) Cartan subalgebra $\mathfrak{H}$ of $V_1$ and $g_{|\g_1}=\sigma_{x}\tau$ and $g_{|\g_2}=\sigma_{\Lambda_1/2}$, where $\tau\in\Aut(\g_1)$ is the order $2$ diagram automorphism on $\mathfrak{H}\cap\g_1$ and $x\in\mathfrak{H}^\tau$.
Thus $g$ gives the desired conjugation since $\Lambda_1+\Lambda_7$ is fixed by $\tau$ in $\g_1$.
\end{proof}

\begin{lemma}\label{L:conjA74}  
The conjugacy class of an order $2$ automorphism $g$ of $V$ is unique if $V^g_1$ is a Lie algebra of type $A_{6}A_{1}^2U(1)^2$ and the conformal weight of $V(g)$ is one.
%belongs to $(1/2)\Z$ and $(\tilde{V}_g)_1$ has dimension $120$.
\end{lemma}
\begin{proof}
%Let $f$ be an order $2$ automorphism of $V_1$ such that $V^f_1$ is a semisimple Lie algebra of type $A_{6,4}A_{1,1}^2U(1)^2$.
Let $\bar{g}$ be the restriction of $g$ to $V_1$.
Since the Lie ranks of $V_1$ and $V^g_1$ are the same, $\bar{g}$ is inner and it preserves every simple ideal of $V_1$.
By the type of $V_1^g$ (cf.\ \cite[Section 8]{Kac}), the fixed point sets $\g_1^{\bar{g}}$ and $\g_{234}^{\bar{g}}$ have the type $A_{6,4}U(1)$ and $A_{1,1}^2U(1)$, respectively.
By Lemma \ref{L:conjinv}, $\bar{g}$ is unique up to conjugation by an inner automorphism of $V_1$;
%$\bar{g}$ is unique up to conjugation by an element in $\Inn(V_1)$.
%By the argument above, 
we may assume that $\bar{g}=\sigma_{u}$ for some $$u\in\left\{\frac12(\Lambda_1,\Lambda_1,0,0), \frac12(\Lambda_1,0,\Lambda_1,0),\frac12(\Lambda_1,0,0,\Lambda_1)\right\}.$$
By Theorem \ref{A74}, we have $g=\sigma_v$ on $V$ for some $$v\in u+\Z\langle (0,0,0,\Lambda_1), (0,0,\Lambda_1,0), (0,\Lambda_1,0,0)\rangle.$$
%Then $g\in\sigma_uI(V)$.

We now recall the $\Com_V(L(\g_{234}))\otimes L(\g_{234})$-modules structure of $V$ in \eqref{Eq:decA74} (see also Lemma \ref{L:modCA11}).
%, where $\g_{234}$ is the semisimple ideal of $V_1$ of type $A_{1,1}^3$.
Note that the weight $(4\Lambda,\Lambda_1,0,0)$ appears in the module $M(\Lambda_1,0,0)\otimes L_{A_1}(1,\Lambda_1)\otimes L_{A_1}(1,0)^{\otimes2}$.
If $u=(1/2)(\Lambda_1,0,\Lambda_1,0)$ or $u=(1/2)(\Lambda_1,0,0,\Lambda_1)$, then we have $(u|(4\Lambda,\Lambda_1,0,0))\in(1/4)\Z$.
% the inner product of $v$ and weights in $M(\Lambda_1,0,0)\otimes L_{A_1}(1,\Lambda_1)\otimes L_{A_1}(1,0)^{\otimes2}$ belongs to $1/4+(1/2)\Z$.
Hence $\sigma_v$ has order $4$ on $V$, which is a contradiction.
Thus $u=(1/2)(\Lambda_1,\Lambda_1,0,0)$.
%%\color{black}
%$I(V)=\langle \sigma_{(0,0,0,\Lambda_1)},\sigma_{(0,0,\Lambda_1,0)},\sigma_{(0,\Lambda_1,0,0)}\rangle$.
Since $\sigma_{(1/2)(\Lambda_1,\Lambda_1,0,0)}$ is conjugate to $\sigma_{(1/2)(\Lambda_1,-\Lambda_1,0,0)}$ by inner automorphism of $V$, we have $$v\in\left\{(\frac12\Lambda_1,\frac12\Lambda_1,0,0),\ (\frac12\Lambda_1,\frac12\Lambda_1,\Lambda_1,0),\  (\frac12\Lambda_1,\frac12\Lambda_1,0,\Lambda_1),\ (\frac12\Lambda_1,\frac12\Lambda_1,\Lambda_1,\Lambda_1)\right\}.$$
Suppose, for a contradiction, that $v\neq (\frac12\Lambda_1,\frac12\Lambda_1,\Lambda_1,\Lambda_1)$.
Then $\langle v|v\rangle/2<1$, and by \eqref{Eq:Lh}, the vacuum vector $\1\in V^{(v)}(\cong V(g))$ has conformal weight less than one, which is a contradiction.
Thus we obtain $v=((1/2)\Lambda_1,(1/2)\Lambda_1,\Lambda_1,\Lambda_1)$, and the conjugacy class of $g$ is unique in $\Aut (V)$.
\end{proof}

\begin{remark} For $v=((1/2)\Lambda_1,(1/2)\Lambda_1,\Lambda_1,\Lambda_1)$, the automorphism $\sigma_v\in \Aut (V)$ has order $2$; one can  check directly that the inner product $(v|\Lambda)\in(1/2)\Z$ for the highest weight $\Lambda$ of any irreducible $L(V_1)$-module with integral conformal weight at least $2$.
\end{remark}

\begin{lemma}\label{L:conjA74-2}  
The conjugacy class of an order $2$ automorphism $g$ of $V$ is unique if $V^g_1$ is a Lie algebra of type $  A_{4}A_{2} A_{1}^2U(1)^2$, the conformal weight of $V(g)$ is one and the type of $(\tilde{V}_g)_1$ is not $A_{7,4}A_{1,1}^3$.
\end{lemma}
\begin{proof} 
Let $\bar{g}$ be the restriction of $g$ to $V_1$.
Since the Lie ranks of $V_1$ and $V^g_1$ are the same, $\bar{g}$ is inner, and it preserves every simple ideal of $V_1$. 
By the type of $V_1^g$ (cf.\ \cite[Section 8]{Kac}), the fixed point sets $\g_1^{\bar{g}}$ and $\g_{234}^{\bar{g}}$ have the type $A_{4}A_{2}U(1)$ and $A_{1}^2U(1)$, respectively.
By Lemma \ref{L:conjinv}, up to conjugation by an inner automorphism of $V_1$, we may assume that $\bar{g}=\sigma_u$ for some $$u\in\left\{\frac12(\Lambda_3,\Lambda_1,0,0),\frac12(\Lambda_3,0,\Lambda_1,0), \frac12(\Lambda_3,0,0,\Lambda_1)\right\}.$$
By Theorem \ref{A74}, $g=\sigma_{v}$ on $V$ for some $$v\in u+\Z\langle (0,0,0,\Lambda_1),{(0,0,\Lambda_1,0)},(0,\Lambda_1,0,0)\rangle.$$
%\color{black}

By the same argument as in the proof of Lemma \ref{L:conjA74}, 
if $u=(1/2)(\Lambda_3,0,\Lambda_1,0)$ or $(1/2)(\Lambda_3,0,0,\Lambda_1)$, 
then the order of $g$ is $4$ on $V$, which is a contradiction.
Hence we obtain $u=(1/2)(\Lambda_3,\Lambda_1,0,0)$.
Note that $\sigma_{(1/2)(\Lambda_3,\Lambda_1,0,0)}$ is conjugate to $\sigma_{(1/2)(\Lambda_3,-\Lambda_1,0,0)}$ by inner automorphism of $V$.
Hence $$v\in\left\{(\frac12\Lambda_3,\frac12\Lambda_1,0,0),\ (\frac12\Lambda_3,\frac12\Lambda_1,\Lambda_1,0),\  (\frac12\Lambda_3,\frac12\Lambda_1,0,\Lambda_1),\ (\frac12\Lambda_3,\frac12\Lambda_1,\Lambda_1,\Lambda_1)\right\}.$$

Suppose, for a contradiction, that $v=(\frac12\Lambda_3,\frac12\Lambda_1,\Lambda_1,0)$ or $(\frac12\Lambda_1,\frac12\Lambda_1,0,\Lambda_1)$.
Then $\langle v|v\rangle/2=5/4$, and by \eqref{Eq:Lh}, the vacuum vector $\1$ belongs to $(V^{(v)})_{5/4}$, which contradicts that the conformal weight of $V(g)$ is one.

Suppose, for a contradiction, that $v=(\frac12\Lambda_3,\frac12\Lambda_1,0,0)$.
Then $\langle v|v\rangle/2=1$.
By \eqref{Eq:decA74} and Lemma \ref{L:modCA11}, $V$ contains an irreducible $L(V_1)$-submodule $M$ with the highest weight $\lambda=(\Lambda_1+\Lambda_5+\Lambda_6, \Lambda_1,0,0)$.
Note that the conformal weight of $M$ is $2$.
Then $M_2$ contains a weight vector $x$ of weight $r(\lambda)=(-\Lambda_2-\Lambda_3-\Lambda_7,-\Lambda_1,0,0)$, where $r$ is the product of the longest elements in the Weyl group of $V_1$.
Now we view $x$ as an element in $M^{(v)}\subset V^{(v)}$.
Since $(r(\lambda)|v)=-2$, we have $x\in (M^{(v)})_1$ (see \eqref{Eq:Lh}) and its weight in $M^{(v)}$ is  $$r(\lambda)+(2\Lambda_3,\frac12\Lambda_1,0,0)=(-\Lambda_2+\Lambda_3-\Lambda_7,-\frac12\Lambda_1,0,0)$$ (see \eqref{Eq:V1h}).
Then this weight and the set $\{(\beta,0,0,0)\mid \beta\in \{\alpha_1, \alpha_2, \alpha_7, \alpha_6, \alpha_5, \alpha_4\}\}$ form a set of simple roots of type $A_7$ (with respect to the normalized inner product of $V_1$) in $(\tilde{V}_{\sigma_v})_1$.
Hence $(\tilde{V}_{\sigma_v})_1$ contains a Lie subalgebra of type $A_7$.
%By \eqref{Eq:decA74} and Lemma \ref{L:modCA11}, one can directly check that the conformal weight of $V^{(v)}$ is $1$.
By the dimension formula (Theorem \ref{Thm:Dimformula}), we have $\dim (\tilde{V}_{\sigma_v})_1=72$.
Hence $(\tilde{V}_{\sigma_v})_1$ must have the type $A_{7,4}A_{1,1}^3$, which contradicts the assumption.
Thus $v=(\frac12\Lambda_3,\frac12\Lambda_1,\Lambda_1,\Lambda_1)$, and we obtain this proposition.
\end{proof}

\begin{remark} For $v=(\frac12\Lambda_3,\frac12\Lambda_1,\Lambda_1,\Lambda_1)$, the type of $(\tilde{V}_{\sigma_v})_1$ must be $D_{5,4}C_{3,2}A_{1,1}^2$.
Indeed, one can directly check that $(\tilde{V}_{\sigma_v})_1$ contains a Lie subalgebra of type $D_5$ by the following way:
$V$ contains an irreducible $L(V_1)$-submodule $M$ with the highest weight   $\lambda'=(\Lambda_3+\Lambda_5, 0, \Lambda_1, \Lambda_1)$ (see Lemma \ref{L:modCA11}). 
By  the same argument, $(M^{(v)})_1$ contains a weight vector of weight
$$r(\lambda')+(2\Lambda_3,\frac12\Lambda_1,\Lambda_1,\Lambda_1)=(\Lambda_3-\Lambda_5, -\frac{1}2\Lambda_1, 0,0).$$
Then this weight and the roots $\{(\beta,0,0,0)\mid \beta\in\{\alpha_4, \alpha_5, \alpha_6, \alpha_7\}\}$ form a set of simple roots of type $D_5$.
\end{remark}

\begin{lemma}\label{L:conjA74-3}  
The conjugacy class of an order $2$ automorphism $g$ of $V$ is unique if $V^g_1$ is a Lie algebra of type $ A_{5}A_{1}^4U(1)$, the conformal weight of $V(g)$ is one and the Lie rank of $(\tilde{V}_g)_1$ is $10$.
\end{lemma}
\begin{proof} 
Let $\bar{g}$ be the restriction of $g$ to $V_1$.
Since the Lie ranks of $V_1$ and $V^g_1$ are the same, $\bar{g}$ is inner and it preserves every simple ideal of $V_1$;
the fixed point sets $\g_1^{\bar{g}}$ and $\g_{234}^{\bar{g}}$ have the types $A_{5}A_{1}U(1)$ and $A_{1}^3$, respectively.
By Lemma \ref{L:conjinv}, up to conjugation by an inner automorphism of $V$, we may assume that $\bar{g}=\sigma_u$ and $u=\frac12(-\theta,0,0,0)$, where $\theta(=-\Lambda_1-\Lambda_7)$ is the highest root.

By Theorem \ref{A74}, 
$g=\sigma_{v}$ on $V$ for some $v\in u+\Z\langle (0,0,0,\Lambda_1),{(0,0,\Lambda_1,0)},(0,\Lambda_1,0,0)\rangle$.
Since the conformal weight of $V(g)$ is one, we have $\langle v|v\rangle\in\Z$.
Since $\langle \theta|\theta\rangle=8$ on $\g_1$ and $\langle \Lambda_1|\Lambda_1\rangle=1/2$ on $\g_{234}$, we have
$$v\in\left\{(-\frac\theta2,0,0,0),\ (-\frac\theta2,\Lambda_1,\Lambda_1,0),\  (-\frac\theta2,\Lambda_1,0,\Lambda_1),\ (-\frac\theta2,0,\Lambda_1,\Lambda_1)\right\}.$$

Suppose, for a contradiction, that $v=(-\theta/2,0,0,0)$.
By \eqref{Eq:decA74}, Lemmas \ref{L:A7mod} and \ref{L:modCA11}, $V$ contains an irreducible $L(V_1)$-module $M$ with the highest weight $(\Lambda_1+\Lambda_2+\Lambda_6+\Lambda_7,0,0,0)$.
The conformal weight of $M$ is $2$ and $M_2$ contains a weight vector $x$ of weight $(2\theta,0,0,0)$.
Then $x\in (M^{(v)})_1$ and its weight is $0$.
By \cite[Lemma 8.1 b)]{Kac},  the Lie rank of $(\widetilde{V}_{\sigma_v})_1$ is greater than $10$, which is a contradiction.

Suppose, for a contradiction, that $v=(-\theta/2,0,\Lambda_1,\Lambda_1)$.
Then $\langle v|v\rangle=3$.
By \eqref{Eq:decA74} and Lemma \ref{L:modCA11}, $V$ contains an irreducible $L(V_1)$-module $M$ with the highest weight $(2\theta,0,\Lambda_1,\Lambda_1)$.
Note that $M$ has a weight vector of weight $(2\theta,0,-\Lambda_1,-\Lambda_1)$ and $$((2\theta,0,-\Lambda_1,-\Lambda_1)|v)=-3.$$
By Lemma \ref{Lem:lowestwt}, the conformal weight of $M^{(v)}$ is at most $1/2(=2-3+{3}/{2})$, which contradicts that the conformal weight of $V(g)(\cong V^{(v)})$ is one.

Thus $v=(-\theta/2,\Lambda_1,\Lambda_1,0)$ or $v=(-\theta/2,\Lambda_1,0,\Lambda_1)$.
By Lemma \ref{L:flip}, the associated inner automorphisms are conjugate in $\Aut (V)$.
Therefore we obtain this proposition.
\end{proof}

\section{Orbifold constructions associated with inner automorphisms}
%\mycolor{Throughout this section, %we fix a set of simple roots of the weight one Lie algebra.
%we adopt the labeling of simple roots $\alpha_1,\dots,\alpha_\ell$ of (the root system of) a simple Lie algebra of rank $\ell$ as in \cite[Section 11.4]{Hu}.
%Let $\Lambda_1,\dots,\Lambda_\ell$ be the fundamental weights so that $(\Lambda_i|\alpha_j^\vee)=\delta_{i,j}$, where $\alpha_j^\vee=2\alpha_j/(\alpha_j|\alpha_j)$.
%with respect to $\alpha_1,\dots,\alpha_\ell$.
%We often use $(k_1, \dots, k_\ell)$ to denote the dominant integral weight $\sum_{i=1}^\ell k_i\Lambda_i$, where $k_1, \dots, k_\ell$ are non-negative integers.  
%}{blue}
The main purpose of this section is to prove the following theorem:

\begin{theorem}\label{T:rev}
Let $V$ be a strongly regular holomorphic VOA of central charge $24$.
Let $V_1$  and $u$ be described as in one of the rows of Table \ref{T:GM}.
Then $\sigma_u$ is an order $2$ automorphism of $V$ and satisfies Condition (I) (see Section \ref{S:Orb}). Moreover, the weight one Lie algebra of $(\tilde{V}_{\sigma_u})$ has the type given on the same row.
\end{theorem}

\begin{tiny}
\begin{longtable}[bht]{|c|c|c|c|c|c|c|c|} 
\caption{Lie algebra structures for the orbifold constructions
}\label{T:GM} \\
\hline 
$V_1$  & $\dim V_1$&$u$&$\langle u|u\rangle$& $V_1^{\sigma_u}$&$\dim V_1^{\sigma_u}$&$(\tilde{V}_{\sigma_u})_1$&$\dim (\tilde{V}_{\sigma_u})_1$\\
 \hline%\hline 
$C_{4,10}$  &$36$ & $\Lambda_2$&$10$  &$C_{2,10}^2$&$20$&$A_{4,5}^2$&$48$ \\ \hline 
 $D_{7,3}A_{3,1}G_{2,1}$ &$120$& $\frac12(\Lambda_5, \Lambda_2, 0)$  &$4$&$D_{5,3}G_{2,1}A_{1,3}^2A_{1,1}^2U(1)$&$72$&$E_{6,3}G_{2,1}^3$&$120$  \\ \hline %
 $A_{5,6}C_{2,3}A_{1,2}$  & $48$& $\frac12(\Lambda_1,\Lambda_1,0)$ &$2$  &$A_{4,6}A_{1,6}A_{1,2}U(1)^2$&$32$&$C_{5,3}G_{2,2}A_{1,1}$&$72$\\ \hline 
 $C_{7,2}A_{3,1}$  & $120$& $\frac12(\Lambda_7,\Lambda_2)$ &$2$ &$A_{6,4}A_{1,1}^2U(1)^2$&$56$&$A_{7,4}A_{1,1}^3$&$72$\\\hline 
 $D_{5,4}C_{3,2}A_{1,1}^2$    &$72$ & $\frac12(\Lambda_5, \Lambda_3, 0,0)$&$2$  &$A_{4,4}A_{2,4}A_{1,1}^2U(1)^2$&$40$&$A_{7,4}A_{1,1}^3$&$72$ \\\hline %     
 $E_{6,4}C_{2,1}A_{2,1}$    & $96$&$\frac12(\Lambda_2,2\Lambda_1,\Lambda_1+\Lambda_2)$  &$3$ &$A_{5,4}A_{1,4}A_{1,1}^3U(1)$&$48$&$A_{7,4}A_{1,1}^3$&$72$\\\hline % 
\end{longtable}
\end{tiny}

\subsection{General method}\label{S:GM}

Let $V$ be a VOA and $u\in V_1$ satisfying the assumptions of Theorem \ref{T:rev}.
One can easily determine $\langle u|u\rangle$ and the type of $V_1^{\sigma_u}$, which are described in Table \ref{T:GM}.

Let $L(V_1)$ be the subVOA generated by $V_1$.
When the type of $V_1$ is $C_{4,10}$, $D_{7,3}A_{3,1}G_{2,1}$, $A_{5,6}C_{2,3}A_{1,2}$, $C_{7,2}A_{3,1}$, $D_{5,4}C_{3,2}A_{1,1}^2$ or $E_{6,4}C_{2,1}A_{2,1}$, then $L(V_1)$ has  $46$, $23$, $266$, $11$, $71$ or $29$ inequivalent irreducible modules with integral conformal weight at least $2$, respectively; for the tables of their highest weights, see Appendix A. Note that the table for the case  $A_{5,6}C_{2,3}A_{1,2}$ is omitted because of the large number of such modules.
These tables also archive the inner product of the highest weight and $u$ and the conformal weight of the $\sigma_u$-twisted module.

Let $M$ be an irreducible $L(V_1)$-submodule $M$ of $V$.
Let $\lambda$ be the highest weight of $M$.
Then the conformal weight of $M$ is at least two or $0$.
Note that the conformal weight of $M$ is $0$ if and only if $M=L(V_1)$.
By the tables in Appendix A, we have $(u|\lambda)\in(1/2)\Z$, and $\sigma_u^2=id$ on $M$. 
Clearly, $\sigma_u$ is not identity.
Hence the order of $\sigma_u$ is $2$ on $V$.
In addition, by the tables in Appendix A, we know that the conformal weight of $M^{(u)}$ is at least one; note that the conformal weight of $L(V_1)^{(u)}$ is equal to $\langle u|u\rangle/2$, which is at least one by Table \ref{T:GM}.
Thus the conformal weight of $V^{(u)}$ is also at least one.

Now, we apply the $\Z_2$-orbifold construction to $V$ and $\sigma_u$, and obtain a strongly regular holomorphic VOA $\tilde{V}_{\sigma_u}$ of central charge $24$.
The dimension of $(\tilde{V}_{\sigma_u})_1$ is determined by the dimension formula in Theorem \ref{Thm:Dimformula} and the information given  in Table  \ref{T:GM}.
The remaining task is to identify the Lie algebra structure of $(\tilde{V}_{\sigma_u})_1$.
Note that there are few possibilities for them by Schellekens' list.
In addition, $(\tilde{V}_{\sigma_u})_1$ has a Lie subalgebra $V_1^{\sigma_u}$ as the fixed point subspace of an order $2$ automorphism of $(\tilde{V}_{\sigma_u})_1$.

If the type of $V_1$ is $C_{4,10}$, $A_{5,6}C_{2,3}A_{1,2}$, $C_{7,2}A_{3,1}$ or $E_{6,4}C_{2,1}A_{2,1}$, then the Lie algebra structure of $(\tilde{V}_{\sigma_u})_1$ is uniquely determined by the dimension and the Lie subalgebra structure.
If the type of $V_1$ is $D_{7,3}A_{3,1}G_{2,1}$ (resp. $D_{5,4}C_{3,2}A_{1,1}^2$), then that of $(\tilde{V}_{\sigma_u})_1$ is $E_{6,3}G_{2,1}^3$ or $D_{7,3}A_{3,1}G_{2,1}$ (resp. $A_{7,4}A_{1,1}^3$ or  $D_{5,4}C_{3,2}A_{1,1}^2 $); we will determine the type of $V_1$ in the subsections below.

\subsection{Case $D_{7,3}A_{3,1}G_{2,1}$}\label{S:D7}
Assume that $V_1$ 
has the type $D_{7,3}A_{3,1}G_{2,1}$.
Let $\g_1$, $\g_2$, $\g_3$ be the simple ideals of $V_1$ of type $D_{7}$, $A_{3}$, $G_{2}$. respectively.
Set $\g_{23}=\g_2\oplus\g_3$.
The subVOA $\Com_V(L(\g_{23}))$ of $V$ contains $L(\g_1)(\cong L_{D_7}(3,0))$ as a full subVOA.
Since $L_{D_7}(3,0)$ and $L_{D_7}(3,\Lambda_1+\Lambda_5)$ are the only irreducible $L_{D_7}(3,0)$-modules with integral conformal weights up to isomorphism (see Table \ref{tableD73} in Appendix A), we have 
\[
\Com_V(L(\g_{23}))\cong L_{D_7}(3,0)\oplus m L_{D_7}(3,\Lambda_1+\Lambda_5)
\]
as $L_{D_7}(3,0)$-modules, where $m$ is the multiplicity. 

\begin{lemma}\label{agc}
The multiplicity $m$ is not zero.
\end{lemma}

\begin{proof}
Since $L(\g_{23})(\cong L_{A_3}(1,0)\otimes L_{G_2}(1,0))$ is the only irreducible module for itself with integral conformal weight (cf.\ Table \ref{tableD73}), we have  $L(\g_{23})=\Com_V(\Com_V(L(\g_{23})))$.
By Proposition \ref{P:Mirror} (1) and (4), the numbers of inequivalent irreducible modules for $\Com_V(L(\g_{23}))$ and $L(\g_{23})$ are the same. 
Since the VOA $L_{D_7}(3,0)$ has $36$ irreducible modules and $L_{A_3}(1,0)\otimes L_{G_2}(1,0)$ has $8$ irreducible modules, the multiplicity $m$ must be non-zero.
\end{proof}

By the lemma above, $V$ contains an irreducible $L(V_1)$-submodule with the highest weight $(\Lambda_1+\Lambda_5,0,0)$.
Note that its conformal weight is $2$ (see Table \ref{tableD73}).
Hence $V_2$ contains a weight vector $x$ of weight $(-\Lambda_1-\Lambda_5,0,0)$.
By Lemma \ref{Lem:wtTw} and \eqref{Eq:Lh}, $x\in (V^{(u)})_1$ and its weight is
$$(-\Lambda_1-\Lambda_5,0,0)+(\frac32\Lambda_5,\frac12\Lambda_2,0)=(\frac{1}2 \Lambda_5 -\Lambda_1, \frac{1}2 \Lambda_2,0).$$ 
Since $V_1$ and $(\tilde{V}_{\sigma_u})_1$ share a Cartan subalgebra, $$\{ (\alpha_1,0,0), \dots, (\alpha_4,0,0), (-\theta,0,0), (\frac{1}2 \Lambda_5 -\Lambda_1, \frac{1}2 \Lambda_2,0)\}$$ is a set of weights of $(\tilde{V}_{\sigma_u})_1$
and it generates a root system of type $E_{6}$ (with respect to a normalized Killing form on $(\tilde{V}_{\sigma_u})_1$), where $\theta$ is the highest root of $D_7$.
We have discussed in Section \ref{S:GM} that the type of $(\tilde{V}_{\sigma_u})_1$ is $E_{6,3}G_{2,1}^3$ or $D_{7,3}A_{3,1}G_{2,1}$.
Hence, it must be $E_{6,3}G_{2,1}^3$.

\subsection{Case $D_{5,4}C_{3,2}A_{1,1}^2$}\label{S:D5}
Assume that $V_1$ has the type $D_{5,4}C_{3,2}A_{1,1}^2 $.
By $\langle u|u\rangle=2$, the conformal weight of $L(V_1)^{(u)}$ is one and $(L(V_1)^{(u)})_1$ is spanned by the vacuum vector (\cite[Lemma 3.4]{LS3}).
Let $M$ be an irreducible $L(V_1)$-submodule of $V$ such that the conformal weight of $M$ is an integer at least two and the associated $\sigma_u$-twisted module has conformal weight $1$; note that such a submodule exists by $\dim (V^{(u)}_{1})=32$.
Let $\Lambda$ be the highest weight of $M$.
By Table \ref{tableD54} in Appendix A, $\Lambda$ is one of the $7$ weights listed in Table \ref{tableHD54}. 
%inequivalent irreducible $L(V_1)$-modules such that the conformal weight is an integer at least $2$ and the associated $\sigma_u$-twisted module has conformal weight $1$; these weights are summarized in Table \ref{tableHD54}.
%Let $M$ be one of them and let $\Lambda$ be its highest weight.
By Lemma \ref{Lem:wtTw}, $(M^{(u)})_1$ contains a weight vector of weight $ r\Lambda + (2\Lambda_5, \Lambda_3,0,0)$, where $r$ is the product of the longest element of the Weyl group of each simple ideal; the weights $ r\Lambda + (2\Lambda_5, \Lambda_3,0,0)$ are also summarized in Table \ref{tableHD54}.

\begin{longtable}[bht]{|c|c|c|} 
\caption{Candidates for $L(V_1)$-submodules
}\label{tableHD54} \vspace{-0.5cm} 
\\ \hline 
Highest weight $\Lambda$  & $r\Lambda+ (2\Lambda_5, \Lambda_3,0,0)$ & Submodule of $V$ ?\\
 \hline\hline 
$(4\Lambda_4, 2\Lambda_3 , 0,0)$  & $(-2\Lambda_5, -\Lambda_3, 0,0)$ &   \\ \hline 
 $(4\Lambda_4, \Lambda_2, 0,0)$  & $(-2\Lambda_5, -\Lambda_2+\Lambda_3, 0,0)$ & No  \\ \hline %
 $(\Lambda_1+2\Lambda_5, 0 ,\Lambda_1,\Lambda_1)$  & $(-\Lambda_1-2\Lambda_4+2\Lambda_5, \Lambda_3,-\Lambda_1,-\Lambda_1)$  & No \\ \hline 
 $(\Lambda_1+2\Lambda_4, 2\Lambda_3, 0,0)$  & $(-\Lambda_1, -\Lambda_3, 0,0)$ & No \\ \hline 
 $(\Lambda_1+2\Lambda_4, \Lambda_2, 0,0)$    & $(-\Lambda_1, -\Lambda_2+\Lambda_3, 0,0)$ & \\ \hline %     
 $(\Lambda_4+\Lambda_5, \Lambda_1+\Lambda_3, 0,0)$    &$(-\Lambda_4+\Lambda_5, -\Lambda_1, 0,0)$ & \\ \hline % 
  $(3\Lambda_4+\Lambda_5, \Lambda_3, 0,0)$  & $(-\Lambda_4 -\Lambda_5, 0,0,0)$ & No   \\ \hline % 
\end{longtable}  

%Now, we assume that $M$ is an irreducible $L(V_1)$-submodule of $V$.
%Remark that such a $M$ exists by $(V^{(u)})_1\neq\{0\}$.

\begin{lemma}\label{L:noweightA7}
%Let $N$ be an irreducible  such that its conformal weight is in $\Z_{\ge2}$ and that of $N^{(u)}$ is $1$.
%Then 
The highest weight of $M$ is neither 
%The VOA $V$ does not contain highest weight module of weight 
$(4\Lambda_4, \Lambda_2,0, 0)$, $(\Lambda_1+2\Lambda_5, 0, \Lambda_1, \Lambda_1)$, $(\Lambda_1+2\Lambda_4, 2\Lambda_3,0,0)$ nor $(3\Lambda_4+\Lambda_5, \Lambda_3,0,0)$.
\end{lemma}

\begin{proof}
%Suppose, for a contradiction, that $V$ contains a highest weight module of weight $(4\Lambda_4, \Lambda_2,0, 0)$, $(\Lambda_1+2\Lambda_5, 0, \Lambda_1, \Lambda_1)$, $(\Lambda_1+2\Lambda_4, 2\Lambda_3,0,0)$, or $(3\Lambda_4+\Lambda_1, \Lambda_3,0,0)$. 
%Then $(\tilde{V}_{\sigma_h})_1$ contains a weight vector of weight 
%\begin{align*}
%\mu_1=(-2\Lambda_5, -\Lambda_2+\Lambda_3, 0,0),\\
%\mu_2=(-\Lambda_1-2\Lambda_4+2\Lambda_5, \Lambda_3,\Lambda_1,\Lambda_1),\\
%\mu_3=-(\Lambda_1, -\Lambda_3, 0,0),\\
%\mu_4= (-\Lambda_4 -\Lambda_5, 0,0,0),
%\end{align*} respectively. 
%Since the levels of simple ideals 
Let $\g_1$, $\g_2$, $\g_3$, $\g_4$ be simple ideals of $V_1$ of type $D_{5,4}$, $C_{3,2}$, $A_{1,1}$ and $A_{1,1}$, respectively.
% be simple ideals of $V_1$ such that $V_1=\bigoplus_{i=1}^4\g_i$.
%are different, 
By the compatibility of levels and inner products in $V_1$ (see Proposition 2.4), we modify the inner product on $V_1$ so that $$((x_1,x_2,x_3,x_4)|(y_1,y_2,y_3,y_4))=(x_1|y_1)_{|\g_1}+2(x_2|y_2)_{|\g_2}+4(x_3|y_3)_{\g_3}+4(x_4|y_4)_{\g_4}.$$
Note that the scalars $2,4,4$ are ratios of levels of $\g_1$ and $\g_2$, $\g_3$, $\g_4$, respectively.

Let $\mu=r\Lambda+ (2\Lambda_5, \Lambda_3,0,0)$ in Table \ref{tableHD54}.
Since $(\tilde{V}_{\sigma_u})_1$ is a semisimple Lie algebra, for any weights $\alpha$, $\beta$ of $(\tilde{V}_{\sigma_u})_1$, we obtain $2(\alpha|\beta)/(\alpha|\alpha)\in\Z$.
However, we obtain a contradiction by 
\begin{align*} \frac{2((0, \alpha_2, 0,0)|\mu)}{(\mu| \mu)}&=-\frac13\quad {\rm if}\quad \mu=(-2\Lambda_5, -\Lambda_2+\Lambda_3, 0,0),\\
\frac{2((0, 0,0, \alpha_1), \mu)}{(\mu| \mu)}&= -\frac{1}{6}\quad {\rm if}\quad \mu=(-\Lambda_1-2\Lambda_4+2\Lambda_5, \Lambda_3,-\Lambda_1,-\Lambda_1),\\
\frac{2((\alpha_1, 0,0,0 )| \mu)}{(\mu| \mu)}&= -\frac{1}{2}\quad {\rm if}\quad \mu=(-\Lambda_1, -\Lambda_3, 0,0),\\
\frac{2((\alpha_4, 0,0,0)| \mu)}{(\mu| \mu)}&= -\frac{1}{2}\quad {\rm if}\quad \mu=(-\Lambda_4 -\Lambda_5, 0,0,0).
\end{align*}
%where $\mu=r\Lambda+ (2\Lambda_5, \Lambda_3,0,0)$ (see Table \ref{tableHD54} for explicit descriptions of $\mu$), 
%and we obtain this lemma.
%\mycolor{\bf In my calculation, the case $(\Lambda_1+2\Lambda_4, 2\Lambda_3,0,0)$ has no contradictions.
%It seems to me that this case is possible for $A_{7,4}A_{1,1}^3$. See below.}{red}
\end{proof}

%In this case, the set $\{ \mu, (0, \alpha_1, 0,0), (0, \alpha_2, 0,0)\}$ will generate a root subsystem . However, $$2((0, \alpha_2, 0,0), \mu)/(\mu, \mu)=-1/3 $$ is not an integer, which is a contradiction. Hence, $V$ does not contain a highest weight module of weight $(4\Lambda_4, \Lambda_2,0, 0)$.  The other cases can be proved similarly. 

%In fact, let $\mu_1=(-\Lambda_1-2\Lambda_4+2\Lambda_5, \Lambda_3,\Lambda_1,\Lambda_1)$, $ \mu_2=-(\Lambda_1, -\Lambda_3, 0,0)$, and $\mu_3= (-\Lambda_4 -\Lambda_5, 0,0,0)$. Then we have 
%\[
%\begin{split}
%{2((0, 0,0, \alpha_1), \mu_1)}/ {(\mu_1, \mu_1)}= -{1}/{6},\\
%{2((\alpha_1, 0,0,0 ), \mu_2)}/ {(\mu_2, \mu_2)}= -{1}/{2},\\
%{2((\alpha_4, 0,0,0), \mu_3)} /{(\mu_3, \mu_3)}= -{1}/{2}.
%\end{split}
%\]

Let us consider the remaining cases.
If the highest weight of $M$ is $(4\Lambda_4, 2\Lambda_3, 0)$, then $(V^{(u)})_1$ contains a vector of weight $-(2\Lambda_5, \Lambda_3,0)$. 
Note that $-(2\Lambda_5, \Lambda_3,0)$ is orthogonal to the root systems associated 
with the Lie subalgebra $ A_{4,4} A_{2,4}A_{1,1}^2$ of $V^{\sigma_u}_1$ and 
it generates a root system of type $A_1$. 
Since the root system $D_{5}C_3A_1^2$ does not contain $A_4A_2A_1^3$ as a root subsystem, $(\tilde{V}_{\sigma_u})_1$ has the type $A_{7,4}A_{1,1}^3$.

If the highest weight of $M$ is $(\Lambda_4+\Lambda_5, \Lambda_1+\Lambda_3,0,0)$ (resp.\ $(\Lambda_1+2\Lambda_4, \Lambda_2, 0,0)$), then $(V^{(u)})_1$ contains a vector of weight $(\Lambda_5-\Lambda_4, -\Lambda_1,0,0)$ (resp.\ $(-\Lambda_1, -\Lambda_2+\Lambda_3, 0,0)$). 
Then this weight and the set $\{ (\alpha_1, 0,0,0)$,  $(\alpha_2, 0,0,0)$, $(\alpha_3, 0,0,0)$, $(\alpha_4, 0,0,0)$, $(0, \alpha_1, 0,0)$, $(0, \alpha_2,0,0)\}$ form a root system of type $A_7$. 
Since $D_{5}C_3A_1^2$ does not contain $A_7$ as a root subsystem, $(\tilde{V}_{\sigma_u})_1$ must be of type $A_{7,4}A_{1,1}^3$.
%Similarly, if $V$ contains a weight vector of weight $(\Lambda_1+2\Lambda_4, \Lambda_2, 0,0)$, then $V^{(h)}$ contains a weight vector of weight $(-\Lambda_1, -\Lambda_2+\Lambda_3, 0,0)$ . Again $\{ (\alpha_1, 0,0,0)$,  $(\alpha_2, 0,0,0)$, $(\alpha_3, 0,0,0)$, $(\alpha_4, 0,0,0)$, $(0, \alpha_1, 0,0)$, $(0, \alpha_2,0,0)$, $(-\Lambda_1, -\Lambda_2+\Lambda_3, 0,0)$ generates a root system of type $A_7$. 
%If $V$ contains a weight vector of weight $(\Lambda_1+2\Lambda_4, 2\Lambda_3,0,0)$, then $(V^{(u)})_1$ contains a weight vector of weight $(-\Lambda_1, -\Lambda_3, 0,0)$.
%Thus, by the argument above, we obtain the following lemma.

%As a summary of this section, we obtain the following theorem.
Thus, the type of $(\tilde{V}_{\sigma_u})_1$ is $A_{7,4}A_{1,1}^3$.
Combining the results in Sections \ref{S:GM}, \ref{S:D7} and \ref{S:D5}, we obtain Theorem \ref{T:rev}.

\section{Main theorem}

\begin{theorem}\label{T:Main}
Up to isomorphism, there exists a unique strongly regular holomorphic VOA of central charge $24$ if its  weight one Lie algebra has the type $C_{4,10}$, $D_{7,3}A_{3,1}G_{2,1}$, $A_{5,6}C_{2,3}A_{1,2}$, $A_{3,1}C_{7,2}$, $D_{5,4}C_{3,2}A_{1,1}^2$, or $E_{6,4}C_{2,1}A_{2,1}$. 
\end{theorem}

\begin{proof}
Let $\mathfrak{g}$ be a semisimple Lie algebra in Table \ref{Tab:Main}.
Let $\mathfrak{p}$ be a Lie subalgebra of $\mathfrak{g}$ whose type is given in the same row of Table \ref{Tab:Main}.
Let $W$ be a strongly regular holomorphic VOA of central charge $24$ such that $W_1$ has the type given in the same row of Table \ref{Tab:Main}.
Applying Theorem \ref{Thm:Rev} to our setting on $\mathfrak{g}$, $\mathfrak{p}$ and $W$, we obtain this theorem;
indeed, by taking $\sigma=\sigma_{u}$, where $u$ is as in Table \ref{T:GM}, the assumptions (a) and (b) hold by Theorem \ref{T:rev}, Table \ref{T:GM} and the uniqueness of $W$;
the assumptions (A) and (B) hold by using the corresponding lemma listed in  Table \ref{Tab:Main} (see also Theorem \ref{T:conj} and Table \ref{Ta:conj}).
Note that all assumptions in the lemmas follow from the fact $(\tilde{W}_\varphi)_1\cong\g$, where $\varphi$ is the order $2$ automorphism associated with the grading of the simple current extension $W$ of $V^\sigma$.
\end{proof}

\begin{tiny}
\begin{longtable}[bht]{|c|c|c|c|} 
\caption{Lie algebra structures for the uniqueness results
%$\mathfrak{H}$-weight of $L(\Lambda)^{(h)}$
}\label{Tab:Main} \\
%\begin{tabular}{|c|c|c|c|}
\hline 
$\mathfrak{g}$  & $\mathfrak{p}$&$W_1$& Ref. for uniqueness of conj. class in $\Aut (W)$\\
 \hline%\hline 
$C_{4,10}$  &$C_{2,10}^2$&$A_{4,5}^2$&Lemma \ref{L:conjA45} \\ \hline 
 $D_{7,3}A_{3,1}G_{2,1}$ &$D_{5,3}G_{2,1}A_{1,3}^2A_{1,1}^2U(1)$&$E_{6,3}G_{2,1}^3$&Lemma \ref{L:conjD73}   \\ \hline %
 $A_{5,6}C_{2,3}A_{1,2}$  & $A_{4,6}A_{1,6}A_{1,2}U(1)^2$&$C_{5,3}G_{2,2}A_{1,1}$& Lemma \ref{L:conjC53}\\ \hline 
 $C_{7,2}A_{3,1}$  &$A_{6,4}A_{1,1}^2U(1)^2$&$A_{7,4}A_{1,1}^3$& Lemma \ref{L:conjA74}\\\hline 
 $D_{5,4}C_{3,2}A_{1,1}^2$   &$A_{4,4}A_{2,4}A_{1,1}^2U(1)^2$&$A_{7,4}A_{1,1}^3$&Lemma \ref{L:conjA74-2} \\\hline %     
 $E_{6,4}C_{2,1}A_{2,1}$    &$A_{5,4}A_{1,4}A_{1,1}^3U(1)$&$A_{7,4}A_{1,1}^3$& Lemma \ref{L:conjA74-3}\\\hline % 
 % $(3\Lambda_4+\Lambda_5, \Lambda_3, 0,0)$  & & $(-\Lambda_4 -\Lambda_5, 0,0,0)$    \\ \hline % 
\end{longtable}
\end{tiny}

\appendix\section{Tables for highest weights}

\begin{tiny}
\begin{longtable}[bht]{|c|c|c|c||c|c|c|c|}
\caption{Irreducible modules with conformal weight in $\Z_{\ge2}$ for $C_{4,10}$
% and their $\sigma_u$-twisted modules.
%conformal weights of $L_{C_4}(10, \Lambda)^{(u)}$
}\label{tableC410}\\
%\tiny
%\begin{tabular}{|c|c|c|c||c|c|c|c|}
\hline
Highest  & Conformal wt & $(u| \Lambda)$& Conformal wt of& Highest  & Conformal wt & $(u| \Lambda)$& Conformal wt of \\
weight&   & & $\sigma_u$-twisted module&  weight&  & & $\sigma_u$-twisted module\\ \hline \hline 
$(0,0,0,10)$ &   $10$ &   $10$ &  $5$ & 
$(0,0,2,3)$ & $3$ & $5$ & $3$\\ \hline      
$(2,0,0,5)$ & $4$ & $6$   & $3$  & 
$(0,0,2,4)$ & $4$ & $6$ &$3$\\ \hline    
$(2,0,0,4)$ & $3$ & $5$ & $3$ &
$(0,0,4,0)$ & $2$ & $4$ & $3$\\ \hline 
$(4,0,0,6)$ & $6$ & $8$ & $3$ &
$(0,0,4,4)$ & $6$ & $8$ & $3$ \\ \hline
$(4,0,0,2)$ & $2$ & $4$ & $3$  &
$ (0,0,6,3)$ & $7$ & $9$ & $3$\\ \hline
$  (6,0,0,1)$ & $2$ &$4$ & $3$  & 
$  (0,0,10,0)$ & $8$ & $10$& $3$\\ \hline
$ (10,0,0,0)$ & $3$  & $5$   &$3$  & 
$ (0,2,6,0)$  & $5$  & $8$  & $2$\\ \hline
$ (6,2,0,2)$ &   $4$ & $7$  & $2$ & 
%$  (0,3,0,0)$   1.0000000   1      & \\ \hline
$(0,3,0,7)$ &$8$ & $10$ & $3$ \\ \hline
$(0,3,2,1)$ & $3$ & $6$ & $2$  & 
$ (2,3,0,4)$ &   $5$  & $8$ & $2$ \\ \hline
$ (0,3,2,3)$  &  $5$  & $8$   &$2$   & 
$ (2,3,0,2)$ &  $3$  & $6$  & $2$\\ \hline
$ (0,3,6,1)$ &  $7$  & $9$  & $3$  & 
$ (6,3,0,0)$  & $3$  & $6$   & $2$\\ \hline
$ (0,5,0,0)$  & $2$  & $5$  & $2$& 
$ (0,5,0,5)$  & $7$ &  $10$ & $2$  \\ \hline 
$ (0,7,2,0)$  & $5$  &  $9$ & $1$  & 
$ (2,7,0,1)$  & $5$  &  $9$ & $1$    \\ \hline 
$ (0,8,0,0)$  & $4$  & $8$ & $1$   & 
$ (0,8,0,2)$ & $6$   & $10$ & $1$  \\ \hline
$ (1,0,3,6)$ & $8$   & $19/2$ & $5/2$   & 
%$ (3,0,1,0)$ & $0000000   1      & \\ \hline
$ (1,0,5,1)$  & $4$  & $13/2$    & $5/2$ \\ \hline
$ (5,0,1,3)$  & $4$  & $13/2$ & $5/2$  & 
$ (1,2,1,1)$  & $2$  & $9/2$  & $5/2$\\ \hline
$ (1,2,1,5)$  & $6$  & $17/2$ & $5/2$  & 
$ (1,4,3,0)$  & $4$  & $15/2$ & $3/2$ \\ \hline
$ (3,4,1,2)$  & $5$  & $17/2$ & $3/2$  &  
$ (1,4,3,1)$  & $5$  & $17/2$ & $3/2$ \\ \hline
$ (3,4,1,1)$  & $4$  & $15/2$&  $3/2$ & 
$ (1,4,5,0)$  & $6$  & $19/2$ & $3/2$\\ \hline
$ (5,4,1,0)$  & $4$  & $15/2$ & $3/2$  & 
$ (2,2,2,2)$  & $4$  & $7$ & $2$ \\ \hline
$ (2,2,4,2)$  & $6$  & $9$ & $2$  & 
$ (4,2,2,0)$  & $3$  & $6$ & $2$\\ \hline
$ (3,0,3,1)$  & $3$  & $11/2$ & $5/2$  & 
$ (3,0,3,3)$  & $5$  & $15/2$ & $5/2$\\ \hline
$ (4,1,4,0)$  & $4$  & $7$ &  $2$ & 
$ (4,1,4,1)$  & $5$  & $8$ & $2$ \\ \hline
%\end{tabular}
\end{longtable}

\begin{longtable}[bht] {|c|c|c|c|}
\caption{Irreducible modules with conformal weights in $\Z_{\ge2}$ for  $D_{7,3}A_{3,1}G_{2,1}$
%\geq 2$ for $D_{7,3}A_{3,1}G_{2,1}$
%and minimal weights of $L(\Lambda)^{(h)}$, $h= \frac{1}2(\Lambda_5, \Lambda_2, 0)$
}\label{tableD73}\\
%\tiny
%\begin{tabular}{|c|c|c|c|}
\hline
Highest weight & Conformal wt & $(u|\Lambda)$ & Conformal wt of the $\sigma_u$-twisted module
%& Minimal weight &Highest  & Minimal weight & $(h| \Lambda)$& Minimal weight 
\\  \hline
$( (0,0,0,0,0,3,0),1,0)$,  $( (0,0,0,0,0,3,0),3,0)$& $  3 $ &  $4$ & $1$    \\ \hline 
$( (3,0,0,0,0,0,0),2,0)$ & $  2 $ &  $5/2$ & $3/2$     \\ \hline
$( (0,0,0,0,0,0,3),1,0)$,  $( (0,0,0,0,0,0,3),3,0)$& $  3 $ & $4$ & $1$    \\ \hline
$( (0,0,0,0,0,1,1),0,1)$ & $  2 $ &  $5/2$ & $3/2$   \\ \hline
$( (1,0,0,0,0,1,0),1,1)$,  $((1,0,0,0,0,1,0),3,1)$& $  2 $ &  $2$ & $2$    \\ \hline
%& $  2 $ &     \\ \hline
$( (1,0,0,0,0,1,1),2,1)$ & $  3 $ &   $7/2$ & $3/2$    \\ \hline
$( (1,0,0,0,0,0,1),1,1)$,  $( (1,0,0,0,0,0,1),3,1)$ & $  2 $ &  $2$ & $2$    \\ \hline
%& $  2 $ &     \\ \hline
$( (1,0,0,0,1,0,0),0,0)$ & $  2 $ &  $3$ & $1$     \\ \hline
$( (0,1,0,0,0,0,1),1,0)$,  $( (0,1,0,0,0,0,1),3,0)$& $  2 $ & $5/2$ & $3/2$ \\ \hline
$( (0,0,0,0,1,0,0),2,0)$ & $  2 $ &  $3$ & $1$   \\ \hline
$( (0,1,0,0,0,1,0) ,1,0)$,  $( (0,1,0,0,0,1,0),3,0)$ & $  2 $ & $2$ & $2$    \\ \hline
%& $  2 $ &     \\ \hline
$((1,0,1,0,0,0,0),0,1)$ & $  2 $ &  $2$ & $2$   \\ \hline
$( (0,0,0,1,0,0,1),1,1)$, $( (0,0,0,1,0,0,1) ,3,1)$ & $  3$ &   $7/2$ & $3/2$   \\ \hline
%& $ 3 $ &     \\ \hline
$( (0,0,1,0,0,0,0) ,2,1)$ & $ 2 $ &   $2$ & $2$   \\ \hline
$( (0,0,0,1,0,1,0) ,1,1)$, $( (0,0,0,1,0,1,0) ,3,1)$& $ 3 $ &  $7/2$ & $3/2$    \\ \hline
%& $ 3 $ &   \\ \hline
%\end{tabular}
\end{longtable}

%\begin{comment}

\begin{longtable}[bht] {|c|c|c|c|}
\caption{Irreducible modules with conformal weights in $\Z_{\ge2}$ for $C_{7,2}A_{3,1}$
%with integral weights $\geq 2$ for $A_{3,1} C_{7,2}$
%and minimal weights of $L(\Lambda)^{(h)}$, $h= \frac{1}2(\Lambda_2,\Lambda_7)$
}\label{tableC72}\\
%\tiny
%\begin{tabular}
\hline
Highest weight & Conformal wt & $(u| \Lambda)$ & Conformal wt of the $\sigma_u$-twisted module
%& Minimal weight &Highest  & Minimal weight & $(h| \Lambda)$& Minimal weight 
\\  \hline 
$((0,0,0,0,0,0,2),(0,1,0))$ &   $4$ & $4$ & $1$ \\ \hline    
$((0,0,0,0,0,1,0),(0,1,0))$ & $2$    & $2$ & $1$ \\ \hline    
$((1,0,0,0,0,0,1),(0,0,0) )$  &$2$    & $2$  & $1$ \\ \hline%    
$((0,0,0,0,2,0,0),(0,0,0))$ & $3$       & $5/2$  & $3/2$ \\ \hline%-     1    
$((0,2,0,0,0,0,0),(0,1,0) )$& $2$     &  $3/2$ & $3/2$  \\ \hline%    
$( (0,0,1,0,0,1,0),(0,0,1))$& $3$     &  $5/2$ & $3/2$ \\ \hline%-     1    
$((0,0,1,0,0,1,0),(1,0,0)  )$ & $3$   &  $5/2$ & $3/2$ \\ \hline%-     1    
$((1,0,0,1,0,0,0),(0,0,1) )$  & $ 2$       &$3/2$   & $3/2$ \\ \hline%-     1    
$((1,0,0,1,0,0,0),(1,0,0) )$  & $2$       & $3/2$  & $3/2$ \\ \hline%-     1    
$( (0,0,1,0,1,0,0),(0,1,0))$  & $3$      & $5/2$  & $3/2$ \\ \hline% -     1    
$((0,1,0,1,0,0,0),(0,0,0))$  & $2$       & $3/2$  & $3/2$  \\ \hline%-     1    
%\end{tabular}
\end{longtable}

\begin{longtable}[bht]{|c|c|c|c|} 
\caption{Irreducible modules with integral weights in $\Z_{\ge2}$ for $D_{5,4}C_{3,2}A_{1,1}^2 $}\label{tableD54} \\
%\begin{tabular}{|c|c|c|c|}
\hline 
Highest weight  & Conformal& $(u| \Lambda)$ & Conformal weight of \\
& weight & & $\sigma_u$-twisted module 
\\  \hline
%$((0,0,0,0,0), (0,0,0) , 0,0\}$  & $   0 $ &    \\ \hline %   0       -     1    
$((4,0,0,0,0), (0,0,0) , 0,0)$  & $   2 $ &  $1$&   $2$  \\ \hline %   2       -     1    
$((0,0,0,4,0), (0,0,0) , 1,1)$  & $   3 $ &  $3/2$&  $3/2$  \\ \hline 
 $((0,0,0,0,4), (0,0,0) , 1,1)$  & $   3 $ & $5/2$ &  $5/2$   \\ \hline %   3       -     1    
 $((0,0,0,4,0), (0,0,2) , 0,0)$  & $   4 $ & $3$ &  $1$   \\ \hline %   4       -     1    
 $((0,0,0,0,4), (0,0,2) , 0,0)$  & $   4 $ & $4$ & $2$   \\ \hline %   4       -     1    
 $((0,0,0,0,0), (0,0,2) , 1,1)$  & $   2 $ & $3/2$& $3/2$   \\ \hline 
 $((4,0,0,0,0), (0,0,2) , 1,1)$  & $   4 $ & $5/2$ &  $5/2$    \\ \hline 
%%%%% Not need : contains weight 1%%%%%%%%%%%%
 $((0,0,0,4,0), (0,1,0) , 0,0)$  & $   3 $ & $2$  & $1$   \\ \hline %   3       -     1    
 $((0,0,0,0,4), (0,1,0) , 0,0)$  & $   3 $ & $3$  & $2$   \\ \hline %   3       -     1    
 %% $((0,0,0,0,0), (0,1,0) , 1,1)$  & $   1 $ &    \\ \hline %   1       -     1    
 $((4,0,0,0,0), (0,1,0) , 1,1)$  & $   3 $ &  $3/2$ &   $5/2$   \\ \hline 
 %% $((0,0,0,0,0), (1,0,1) , 0,0)$  & $   1 $ &    \\ \hline %   1       -     1    
 $((4,0,0,0,0), (1,0,1) , 0,0)$  & $   3 $ &  $2$  & $2$  \\ \hline %   
 $((0,0,0,4,0), (1,0,1) , 1,1)$  & $   4 $ &  $5/2$  & $3/2$  \\ \hline     
  $((0,0,0,0,4), (1,0,1) , 1,1)$  & $   4 $ & $7/2$ & $5/2$   \\ \hline
%%%%% Not need: contains weight 1 %%%%%%%%%%%%
%%  $((0,0,0,1,1), (0,0,0) , 0,0)$  & $  1 $ &    \\ \hline %   1       -     1    
 $((2,0,0,1,1), (0,0,0) , 0,0)$  & $  2 $ &  $3/2$ &   $3/2$ \\ \hline     
 $((1,0,0,2,0), (0,0,0) , 1,1)$  & $  2 $ & $1$ & $3/2$    \\ \hline 
 $( (1,0,0,0,2), (0,0,0) , 1,1)$  & $  2 $ & $3/2$  & $1$   \\ \hline 
 $((1,0,0,2,0), (0,0,2) , 0,0)$  & $  3 $ &  $5/2$  & $1$  \\ \hline 
 $( (1,0,0,0,2), (0,0,2) , 0,0)$  & $  3 $ & $3$ & $3/2$   \\ \hline 
 $((0,0,0,1,1), (0,0,2) , 1,1)$  & $  3 $ &  $5/2$ & $3/2$  \\ \hline 
 $((2,0,0,1,1), (0,0,2) , 1,1)$  & $  4 $ &  $3$ & $2$  \\ \hline %    
%%%%%%%%%%%%%%%%%%%%%%%%%%%%%%%%%%%%%%%
 $((1,0,0,2,0), (0,1,0) , 0,0)$  & $  2 $ &  $3/2$ & $1$  \\ \hline %     
 $( (1,0,0,0,2), (0,1,0) , 0,0)$  & $  2 $ & $2$ &  $3/2$   \\ \hline
 $((0,0,0,1,1), (0,1,0) , 1,1)$  & $  2 $ &  $3/2$&  $3/2$   \\ \hline     
 $((2,0,0,1,1), (0,1,0) , 1,1)$  & $  3 $ &  $2$ &  $2$   \\ \hline    
 $((0,0,0,1,1), (1,0,1) , 0,0)$  & $  2 $ &  $2$ &  $1$  \\ \hline %  
 $((2,0,0,1,1), (1,0,1) , 0,0)$  & $  3 $ &  $5/2$ & $3/2$   \\ \hline %    
 $((1,0,0,2,0), (1,0,1) , 1,1)$  & $  3 $ &  $2$ & $3/2$  \\ \hline %     
 $( (1,0,0,0,2), (1,0,1) , 1,1)$  & $  3 $ & $5/2$  & $2$   \\ \hline 
%%%%%%%%%%%%% Contains non-integral %%%%%%% 
 $( (0,0,0,1,3), (0,0,1) , 0,0)$  & $  3 $ &  $3$ & $3/2$  \\ \hline   
 %993  {25,2,0,0)$  & $  1 $ &    \\ \hline %   1       -     1    
 $( (0,0,0,3,1), (0,0,1) , 0,0)$  & $  3 $ &  $5/2$  & $1$   \\ \hline % 
 $( (3,0,0,0,0), (0,0,1) ,0,0)$  &  $  2 $ &  $3/2$  & $3/2$   \\ \hline %     
 %%%%%%%%%%%%% 2 ineq str.%%%%%%%%%%%%%%%%%
$( (2,0,0,0,0), (0,2,0) , 0,0)$  & $  2 $ & $3/2$ &    $3/2$    \\ \hline %   2       -     1    
$( (0,0,0,2,2), (0,2,0) , 1,1)$  & $  4 $ &  $3/2$ &   $3/2$   \\ \hline 
$( (0,0,0,2,2), (2,0,0) , 0,0)$  & $  3 $ &  $3$  & $2$  \\ \hline %   3       -     1    
 $( (2,0,0,0,0), (2,0,0) , 1,1)$  & $  2 $ & $1$ &  $2$   \\ \hline %   2       -     1    
 %% **( 33x 2 =33) %%%%%%%%%%%%%%%%%%%%%% 
%%%% Conatins non -integral %%%%%%%%%%%%%
$(  (0,0,1,0,0), (0,0,1) , 1,1)$  & $  2 $ &  $3/2$  & $3/2$   \\ \hline 
$( (0,1,0,2,0), (0,0,1) , 1,1)$  & $  3 $ &   $2$  & $3/2$  \\ \hline 
$( (2,0,1,0,0), (0,0,1) , 1,1)$  & $  3 $ &   $2$  & $2$ \\ \hline %   3       -     1    
$( (0,1,0,0,2), (0,0,1) , 1,1)$  & $  3 $ &   $5/2$  & $2$ \\ \hline
 %%%%%%%%%%%%%%%%%%%%%%%%%%%%%%%%%%%%
$(  (0,0,1,0,0), (1,1,0) , 0,1)$, $(  (0,0,1,0,0), (1,1,0) , 1,0)$ & $  2 $ &  $3/2$ &   $3/2$   \\ \hline %   2       -     1    
$( (0,1,0,2,0), (1,1,0) , 0,1)$, $( (0,1,0,2,0), (1,1,0) , 1,0)$  & $  3 $ &   $2$ &   $3/2$  \\ \hline %   3       -     1    
$( (2,0,1,0,0), (1,1,0) , 0,1)$, $( (2,0,1,0,0), (1,1,0) , 1,0)$ & $  3 $ &   $2$& 
 $2$ \\ \hline    
$( (0,1,0,0,2), (1,1,0) , 0,1)$, $( (0,1,0,0,2), (1,1,0) , 1,0)$ & $  3 $ &   $5/2$  & $2$  \\ \hline %   3       -     1    
%  & $  2 $ &    \\ \hline %   2       -     1    
%  & $  3 $ &    \\ \hline %   3       -     1    
%  & $  3 $ &    \\ \hline %   3       -     1    
%  & $  3 $ &    \\ \hline %   3       -     1    
%%%%%%%%%%%%%%%%%%%%%%%%%%%%%%%%%
$( (0,0,1,1,1), (0,0,1) , 0,1)$, $( (0,0,1,1,1), (0,0,1) , 1,0)$  & $  3 $ &  $5/2$ &   $3/2$   \\ \hline %   3       -     1    
$( (1,1,0,0,0), (0,0,1) , 0,1)$, $( (1,1,0,0,0), (0,0,1) , 1,0)$  & $  2 $ &  $3/2$ &   $3/2$   \\ \hline %   2       -     1    
%  & $  3 $ &    \\ \hline %   3       -     1    
%  & $  2 $ &    \\ \hline %   2       -     1   
%% %%% (50x 2 =50) %%%%%%%%%%%%%%
%%%% contains non- integral %%%%%%%%%
$( (0,0,1,1,1), (1,1,0) , 0,0)$  & $  3 $ &   $5/2$ &  $3/2$   \\ \hline 
$( (1,1,0,0,0), (1,1,0) , 0,0)$  & $  2 $ &    $3/2$ &  $3/2$ \\ \hline %   2       -     1    
%%%%%%%%%%%%%%%%%%%%%%%%%%%%%%%%%%%
$((0,0,2,0,0), (0,0,0) , 0,0)$  & $  2 $ &  $3/2$ &   $3/2$   \\ \hline %    
$( (0,2,0,0,0), (0,0,0) , 1,1)$  & $  2 $ &  $1$ &   $2$  \\ \hline %   2       -     1    
$( (0,2,0,0,0), (0,0,2) , 0,0)$  & $  3 $ &   $5/2$&   $3/2$  \\ \hline 
$((0,0,2,0,0), (0,0,2) , 1,1)$  & $  4 $ &   $3$ &  $2$ \\ \hline   
%%%%%%%% (52x 2 =52) %%%%%%%%%%%
$( (0,2,0,0,0), (0,1,0) , 0,0)$  & $  2 $ &  $3/2$ & $3/2$   \\ \hline 
$((0,0,2,0,0), (0,1,0) , 1,1)$  & $  3 $ &   $2$ & $2$ \\ \hline %   3       -     1    
$((0,0,2,0,0), (1,0,1) , 0,0)$  & $  3 $ &   $5/2$ &   $3/2$  \\ \hline 
$( (0,2,0,0,0), (1,0,1) , 1,1)$  & $  3 $ &  $3/2$ &  $5/2$   \\ \hline %   3       -     1    
%%%%%%%%%%%%%%%%%%%%%%%%%%%%
$(  (0,1,0,1,1), (0,2,0) , 0,0)$  & $  3 $ & $5/2$&   $3/2$    \\ \hline 
$( (1,0,1,0,0), (0,2,0) , 1,1)$  & $  3 $ &  $2$  &$2$  \\ \hline %   3       -     1    
$( (1,0,1,0,0), (2,0,0) , 0,0)$  & $  2 $ &  $3/2$ & $3/2$   \\ \hline 
$(  (0,1,0,1,1), (2,0,0) , 1,1)$  & $  3 $ &  $2$  &$2$  \\ \hline %   3       -     1    
%%%%%%%%%%%%%%% (54x2 =54)
%%%%%%% contains non-integral%%%%%%%%%
$(  (0,1,1,0,0), (0,1,1),0,0)$  & $  3 $ &   $5/2$&   $3/2$  \\ \hline %   3       -     1    
$(  (0,1,1,0,0), (1,0,0),0,0)$  & $  2 $ &   $3/2$ &   $3/2$  \\ \hline 
%%%%%%%%%%%%%%%%%%%%%%%%%%%%%
$((1,0,0,1,1),(0,1,1),0,1)$, $(  (1,0,0,1,1),(0,1,1),1,0)$ & $  3 $ & $5/2$ &   $3/2$   
\\ \hline     
%  & $  3 $ &    \\ \hline %   3       -     1    
$(  (1,0,0,1,1),(1,0,0),0,1)$, $(  (1,0,0,1,1),(1,0,0),1,0)$ & $  2 $ &   $3/2$ &   $3/2$  \\ \hline    
%  & $  2 $ &    \\ \hline %   2       -     1    
  \end{longtable}  
%\end{comment}

\begin{longtable}[bht] {|c|c|c|c|}
\caption{Irreducible modules with integral weights $\geq 2$ for $E_{6,4}C_{2,1}A_{2,1}$}\label{tableE64}
%\tiny
%\begin{tabular}
\\
\hline
Highest weight  & Conformal wt & $(h| \Lambda)$ & Conformal wt of the $\sigma_u$-twisted module
%& Minimal weight &Highest  & Minimal weight & $(h| \Lambda)$& Minimal weight 
\\  \hline 
$((4,0,0,0,0,0),(0,0),(0,1))$ & $3$& $5/2$& $2$ \\ \hline    
$((4,0,0,0,0,0),(0,0),(1,0))$ & $3$& $5/2$& $2$ \\ \hline  
$((0,0,0,0,0,4),(0,0),(0,1))$ & $3$& $5/2$& $2$\\ \hline    
$((0,0,0,0,0,4),(0,0),(1,0))$ & $3$& $5/2$& $2$ \\ \hline   
$((1,0,1,0,0,0),(1,0), (0,0))$ & $2$ & $2$& $3/2$ \\ \hline    
$( (1,0,0,0,1,1),(1,0),(1,0))$ & $3$ & $3$& $3/2$ \\ \hline  
$((1,0,0,0,1,1),(1,0),(0,1))$&  $3$  & $3$& $3/2$ \\ \hline    
$((0,1,0,0,0,1),(1,0),(1,0))$&  $2$  & $5/2$& $1$ \\ \hline    
$ ((0,1,0,0,0,1),(1,0),(0,1)) $& $2$ & $5/2$& $1$ \\ \hline   
$  ((2,0,0,0,0,2),(0,1),(0,0)) $& $3$& $5/2$& $2$ \\ \hline    
$  ((0,0,0,0,0,2),(0,1),(1,0)) $& $2$& $2$& $3/2$ \\ \hline    
$  ((0,0,0,0,0,2),(0,1),(0,1)) $& $2$& $2$& $3/2$ \\ \hline    
$  ((2,0,0,0,0,0),(0,1),(1,0))$&  $2$& $2$& $3/2$ \\ \hline  
$  ((2,0,0,0,0,0),(0,1),(0,1))$&  $2$& $2$& $3/2$ \\ \hline   
$  ((0,0,0,0,1,1),(1,0),(0,0))$&  $2$& $2$& $3/2$ \\ \hline   
$((1,1,0,0,0,0),(1,0),(1,0))$&  $2$  & $5/2$& $1$ \\ \hline   
$((1,1,0,0,0,0),(1,0),(0,1))$&  $2$  & $5/2$& $1$ \\ \hline    
$((1,0,1,0,0,1),(1,0),(1,0))$&  $3$  & $3$& $3/2$ \\ \hline   
$((1,0,1,0,0,1),(1,0),(0,1))$&  $3$  & $3$& $3/2$ \\ \hline   
$((0,0,0,1,0,0),(0,1),(0,0))$&  $2$ & $2$& $3/2$ \\ \hline   
$((1,0,0,1,0,0),(0,1),(1,0))$&  $3$ & $3$& $3/2$ \\ \hline   
$((1,0,0,1,0,0),(0,1),(0,1))$&  $3$ & $3$& $3/2$ \\ \hline   
$((0,0,0,1,0,1),(0,1),(1,0))$&  $3$ & $5/2$& $2$ \\ \hline   
$((0,0,0,1,0,1),(0,1),(0,1))$&  $3$ & $5/2$& $2$ \\ \hline   
$((1,1,0,0,0,1),(0,0),(0,0))$&  $2$ & $2$& $3/2$ \\ \hline   
$((0,0,1,0,0,1),(0,0),(0,1))$&  $2$ & $2$& $3/2$ \\ \hline     
$((0,0,1,0,0,1),(0,0),(1,0))$&  $2$ & $2$& $3/2$ \\ \hline     
$((1,0,0,0,1,0),(0,0),(0,1))$&  $2$ & $2$& $3/2$ \\ \hline    
$((1,0,0,0,1,0),(0,0),(1,0))$&  $2$ & $2$& $3/2$ \\ \hline 
%\end{tabular}
\end{longtable}

\end{tiny}

\end{document}